\documentclass{article}
\usepackage{amsmath}
\usepackage{amsthm}
\usepackage{amsfonts}
\usepackage{amssymb}

\usepackage[margin = 1in, twoside]{geometry}

\usepackage{microtype}

\usepackage[UKenglish]{babel}
\usepackage[UKenglish]{isodate}

\PassOptionsToPackage{hyphens}{url}

\usepackage[numbers, square, comma]{natbib}
\usepackage{doi}

\usepackage{enumitem}

\usepackage{graphicx}
\usepackage{subcaption}

\usepackage{comment}

\usepackage{booktabs, tabu}

\usepackage{algpseudocode}

\usepackage{multicol}

\usepackage{bbm}

\usepackage{float}

\usepackage[capitalise]{cleveref}

\usepackage{needspace}
\pretocmd{\section}{\Needspace{4\baselineskip}}{}{}

\usepackage[T1]{fontenc}
\usepackage{bm}

\usepackage[framemethod=1]{mdframed}

\definecolor{background}{rgb}{0.97,0.97,0.97}

\mdfsetup{
    roundcorner=0.1cm,
    skipabove=0.5cm,
    innertopmargin=0.1cm,
    innerbottommargin=0.25cm,
    skipbelow=0,
    nobreak=true,
    backgroundcolor=background
}

\newmdtheoremenv[]{theorem}{Theorem}[section]
\newmdtheoremenv[]{lemma}[theorem]{Lemma}
\newmdtheoremenv[]{proposition}[theorem]{Proposition}
\newmdtheoremenv[]{corollary}[theorem]{Corollary}
\newmdtheoremenv[]{definition}[theorem]{Definition}
\newmdtheoremenv[]{claim}{Claim}
\newmdtheoremenv[]{remark}{Remark}
\newmdtheoremenv[]{problem}{Problem}
\newmdtheoremenv[]{eg}{Example}

\usepackage{tikz}

\usetikzlibrary{backgrounds}

\usepackage{pgfplots}
\pgfplotsset{compat = 1.18}

\usepgfplotslibrary{groupplots}

\usepackage{pgfplotstable}

\newcommand{\Vn}{\ensuremath{{V_n}}}

\newcommand{\Hn}{\ensuremath{{H_n}}}

\newcommand{\Tn}{\ensuremath{{T_n}}}

\newcommand{\Wm}{\ensuremath{{W_m}}}

\newcommand{\eps}{\ensuremath{\varepsilon}}
\newcommand{\ord}{\ensuremath{\cO}}

\newcommand{\id}{\ensuremath{\textsf{Id}}}

\newcommand{\poly}{\ensuremath{\textsf{poly}}}
\newcommand{\SNN}{\ensuremath{\textsf{SNN}}}

\newcommand{\eff}{\ensuremath{\textsf{eff}}}

\usepackage{mathtools}
\usepackage{suffix}
\usepackage{stmaryrd} %

\DeclarePairedDelimiter{\prt}{(}{)}

\DeclarePairedDelimiter{\abs}{\lvert}{\rvert}
\DeclarePairedDelimiter{\norm}{\lVert}{\rVert}
\DeclarePairedDelimiter{\inner}{\langle}{\rangle}
\DeclarePairedDelimiter{\set}{\{}{\}}

\let \oldforall \forall
\let \forall \undefined
\DeclareMathOperator{\forall}{\oldforall}

\let \oldexists \exists
\let \exists \undefined
\DeclareMathOperator{\exists}{\oldexists}

\let \oldtext \text
\renewcommand{\text}[1]{~\oldtext{#1}~}

\newcommand{\cond}{\, : \,}

\DeclareMathOperator{\Lin}{Lin}

\DeclareMathOperator*{\argmin}{arg \, min}

\DeclareMathOperator{\Lip}{Lip}

\usepackage{ifthen}
\newlength{\leftstackrelawd}
\newlength{\leftstackrelbwd}
\def \leftstackrel#1#2{\settowidth{\leftstackrelawd}%
  {${{}^{#1}}$} \settowidth{\leftstackrelbwd}{$#2$}%
  \addtolength{\leftstackrelawd}{- \leftstackrelbwd}%
  \leavevmode \ifthenelse{\lengthtest{\leftstackrelawd>0pt}}%
  {\kern-.5 \leftstackrelawd}{} \mathrel{\mathop{#2} \limits^{#1}}}

\DeclareMathOperator{\grad}{grad}

\DeclareMathOperator{\vspan}{span}

\DeclareMathOperator{\dom}{dom}

\newcommand{\bG}{\ensuremath{\mathbb{G}}}

\newcommand{\bN}{\ensuremath{\mathbb{N}}}

\newcommand{\bP}{\ensuremath{\mathbb{P}}}

\newcommand{\bR}{\ensuremath{\mathbb{R}}}
\newcommand{\bS}{\ensuremath{\mathbb{S}}}

\newcommand{\cC}{\ensuremath{\mathcal{C}}}
\newcommand{\cD}{\ensuremath{\mathcal{D}}}

\newcommand{\cF}{\ensuremath{\mathcal{F}}}

\newcommand{\cJ}{\ensuremath{\mathcal{J}}}

\newcommand{\cO}{\ensuremath{\mathcal{O}}}

\newcommand{\R}{\bR}

\renewcommand{\d}{\mathrm{d}}

\newcommand{\dt}{\d t}

\DeclareFontFamily{U}{matha}{\hyphenchar\font45}
\DeclareFontShape{U}{matha}{m}{n}{
      <5> <6> <7> <8> <9> <10> gen * matha
      <10.95> matha10 <12> <14.4> <17.28> <20.74> <24.88> matha12
      }{}
\DeclareSymbolFont{matha}{U}{matha}{m}{n}

\DeclareMathSymbol{\Lt}{3}{matha}{"CE}
\DeclareMathSymbol{\Gt}{3}{matha}{"CF}

\title{Stable Nonlinear Dynamical Approximation with Dynamical Sampling}
\author{Daan Bon\thanks{Eindhoven University of Technology, Department of Mathematics and Computer Science ({d.l.j.bon@tue.nl}, {b.m.caris@tue.nl}, {o.mula@tue.nl}).} \and Benjamin Caris\footnotemark[1] \and Olga Mula\footnotemark[1]}

\begin{document}

\maketitle

\begin{abstract}
  We present a nonlinear dynamical approximation method for time-dependent Partial Differential Equations (PDEs). The approach makes use of parametrized decoder functions, and provides a general, and principled way of understanding and analyzing stability and accuracy of nonlinear dynamical approximations. The parameters of these functions are evolved in time by means of projections on finite dimensional subspaces of an ambient Hilbert space related to the PDE evolution. For practical computations of these projections, one usually needs to sample. We propose a dynamical sampling strategy which comes with stability guarantees, while keeping a low numerical complexity. We show the effectiveness of the method on several examples in moderate spatial dimension.
\end{abstract}

\section{Introduction}
The efficient numerical treatment of time-dependent PDEs in moderate or high dimension often requires nonlinear approximation methods that can capture structural and dynamical properties of solutions with a reduced amount of degrees of freedom. Finding such compressed representations while preserving accuracy and stability is a crucial challenge. Although the topic has a long history in numerical analysis, there are numerous major open questions which are the subject of very intense research. Overall, the major problem is that there is a significant gap between theoretical results on the approximation power of nonlinear approximations, and results on the practical performance of solvers. This is precisely the issue which we aim to address in this paper. One of the key missing concepts to bridge the gap concerns notions of stability that can be leveraged at the practical level, and that can guarantee that the computed approximations do not deviate very much from the optimal accuracy that one can expect in theory. We contribute to this key question by developing an abstract framework to solve PDEs with nonlinear dynamical approximations. Such approximations are built with {parametrized functions, which we will refer to as} decoder mappings {following standard terminology from approximation theory, information-based complexity and compressed sensing  (see, e.g., \cite{CDPW2021}). The} dynamical aspect means that the input parameters of the decoder evolve in time via an Ordinary Differential Equation (ODE) which is obtained from a least-squares type approximation (sometimes called a Dirac-Frenkel principle, see, e.g.~\cite{lubich2008from}). The idea was originally developed for decoders based on low-rank tensor approximations (see, e.g., \cite{Koch2007Apr,Nonnenmacher2008Dec,Lubich2013May,Uschmajew2013Jul, feppon2018geometric, bachmayr2021existence, charous2023dynamically, bachmayr2023low}), and it has recently been extended to other nonlinear types of decoders such as neural networks or Gaussian mixtures in works such as \cite{berman2023randomized, Bruna2024, zhang2024sequential,  wen2023couplingparameterparticledynamics, AABGMT2024, feischl2024regularized, LN2025}. Our work formulates the task of dynamical approximation for evolutions taking place in Hilbert spaces, where we use abstract expressions for decoder mappings that cover, among others, the types of decoders from previous contributions. The formulation that we develop naturally comes with stability concepts which involve computable quantities that can be used in practice to guarantee stability. The notion of stability that arises seems to have a fundamental meaning that goes beyond the current setting. It is known to also be involved in the stability of other types of tasks such as inverse state estimation (see \cite{MMPY2015, MPPY2015, CDMS2022}), and generalizations of compressed sensing to infinite dimensional spaces (\cite{AHP2013, BGM2019}). In addition to stability, we also show that the approach provably delivers near-optimal accuracy in a sense that will be specified later on, and we even obtain certain types of a posteriori error bounds.

\begin{figure}
  \centering
  \subfloat[$t=0$]{
    \includegraphics[width=0.3\linewidth]{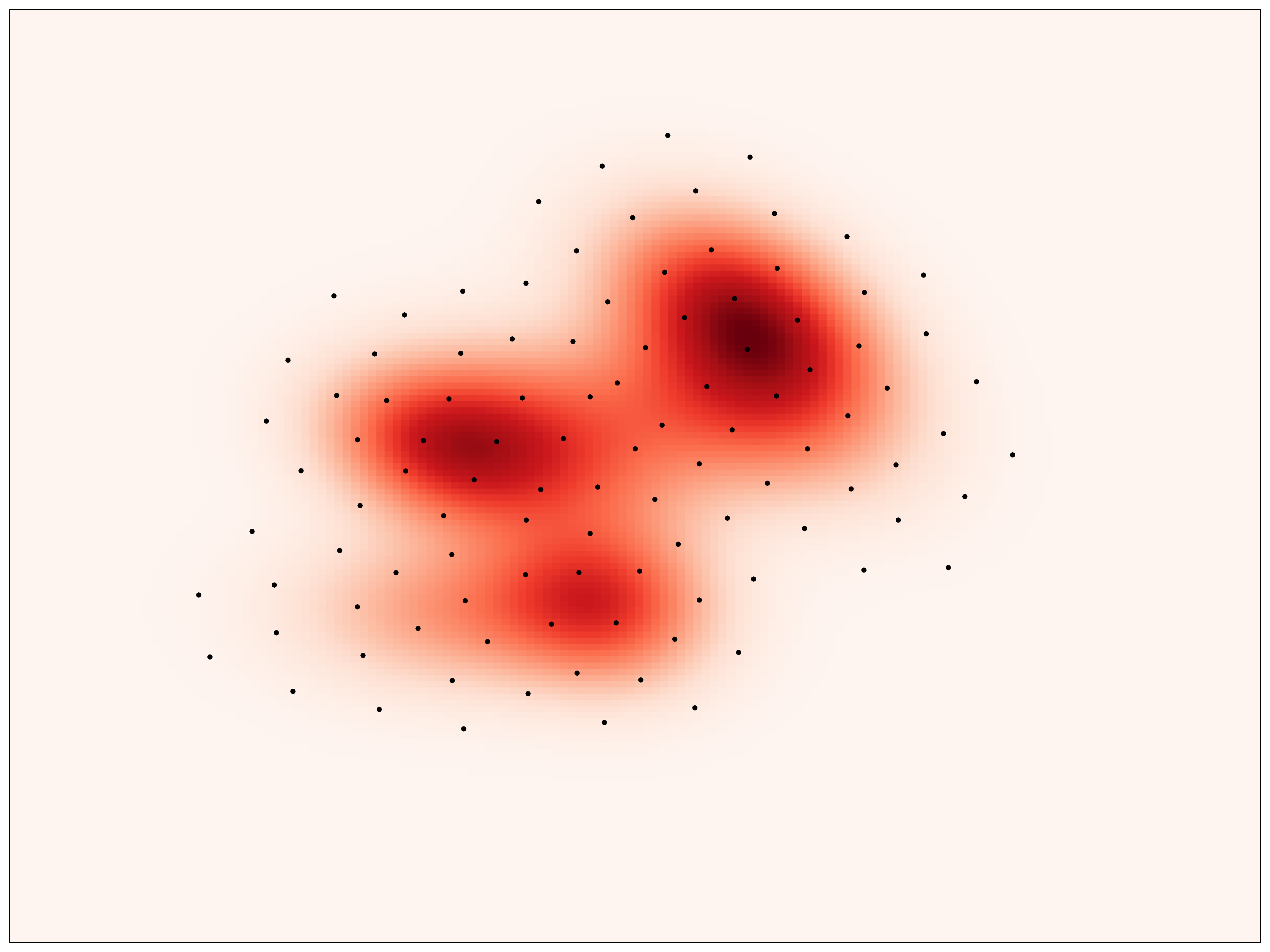}
  }
  \subfloat[$t=T$]{
    \includegraphics[width=0.3\linewidth]{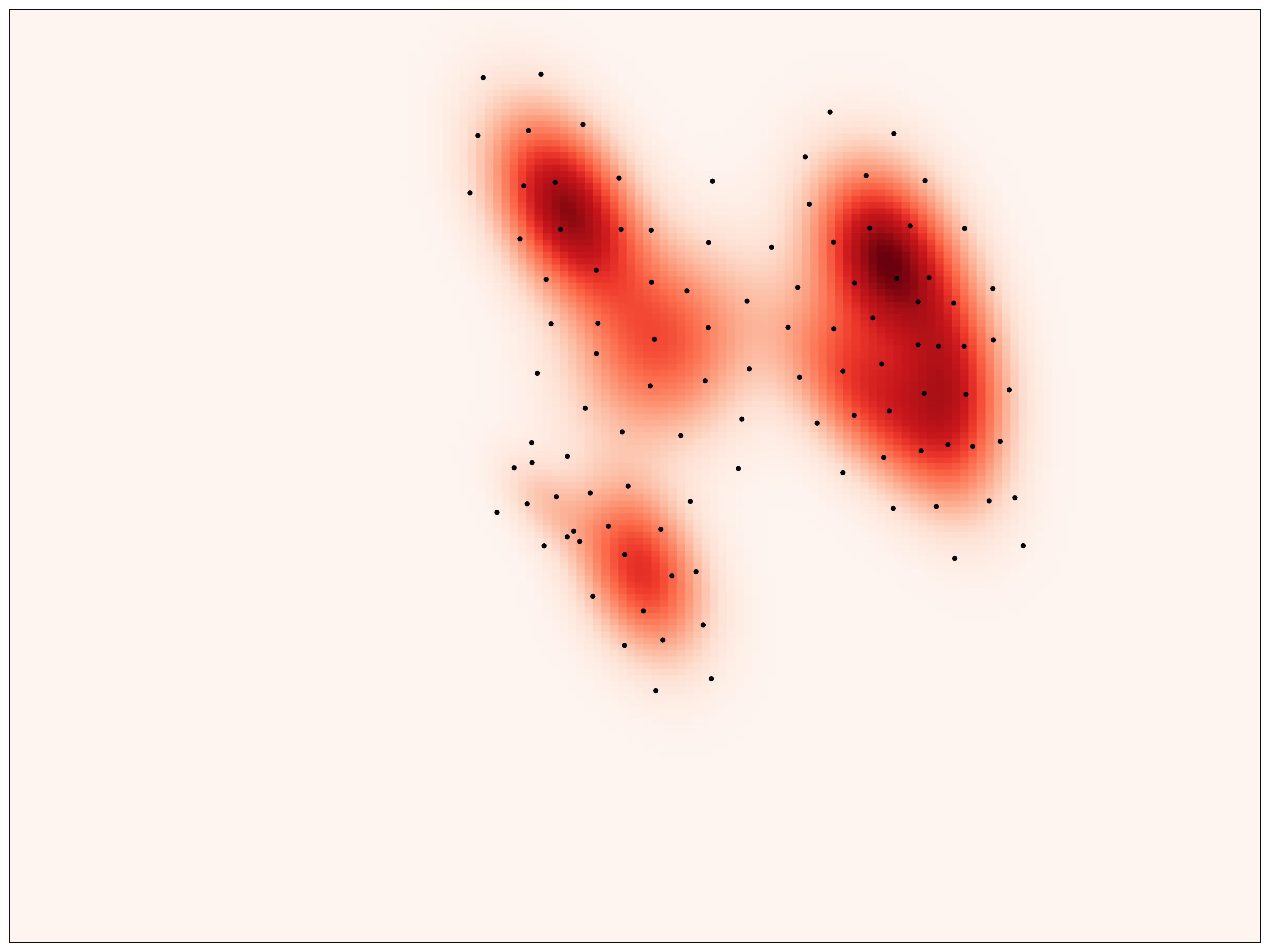}
  }
  \caption{
    The degrees of freedom of the nonlinear decoder evolve in time to preserve stability and near-optimality pointwise in time (\cref{sec:nonlinear dynamical approximation,,sec:dyn-approx-obs}). Around the black dots, local information is collected to best approximate the solution, and the dots move together with the numerical solution thanks to a data-driven strategy (\cref{sec:dyn-sampling}). The images show our dynamical approximation of a Fokker-Planck equation in 2D where the exact solution consists of a mixture of 3 Gaussians (further details in \cref{sec:FP_results}).}
  \label{fig:teaser}
\end{figure}

At the very practical level, our framework leads to a strategy where one needs to estimate local information about the numerical solution, which comes expressed in the form of integrals over spatially localized regions. Our analysis reveals that the location of these regions crucially affects stability. So one needs to dynamically update the locations in time, and we provide a data-driven approach to do this (see \cref{fig:teaser} for a macroscopic idea of how the method works). The accurate estimation of the local information requires sampling techniques for integrals with a certain amount of degrees of freedom which, in turn,  determines the numerical complexity of the whole strategy. It is therefore important to wisely choose the type of local information to extract in order to keep complexity low, while maintaining stability. It is worth mentioning that the complexity to evaluate certain integrals is also present is more classical formulations of dynamical approximation, and the numerical behavior of certain dynamical sampling strategies has been explored in \cite{wen2023couplingparameterparticledynamics, Bruna2024}. In our work, we propose a different dynamical sampling strategy which is obtained from a rigorous theoretical framework.

The question of dynamically, and optimally allocating degrees of freedom is a general question which is also present in approaches that are not based on dynamical approximation. One salient example worth discussing concerns finite element solvers. There, solid theoretical foundations on both accuracy and stability lead to solution strategies where stability is connected to inf-sup conditions on approximation spaces which are involved in existence results such as the Banach–Nečas–Babuška Theorem (see \cite[Theorem 2.6 and Theorem 2.22]{EG2013}). Information on accuracy can be retrieved through a posteriori error estimations (see, e.g., \cite[Chapter 10]{EG2013}), and this leads to space-time adaptive strategies to allocate degrees of freedom on finite element meshes. One important limitation of this approach is that dealing with the grid in high dimensions becomes very difficult in practice. To mitigate the curse of dimension, pruning steps are usually required, and this adds implementation difficulties. One important advantage of our dynamical approach is that it relies on a decoder function which does not involve any underlying grid, and which can easily be evaluated in the spatial domain at any point.

Given these elements of context, the main contributions of this paper are:
\begin{enumerate}
  \item A general, and principled way of understanding and analyzing stability of nonlinear dynamical approximations.
  \item The connection between the actual accuracy of the scheme and theoretical concepts of best approximation.
  \item A dynamical optimal sampling strategy to guarantee stability at low complexity.
  \item Numerical evidence of the performance of the whole strategy in a broad range of PDEs in moderate dimension (our examples go until dimension 10, and further fine-tunning of our code would allow to deal with higher dimensions). We also release of an open source code to reproduce the results, and explore other dynamics.
\end{enumerate}

The remainder of the paper is organised as follows. In \cref{sec:nonlinear dynamical approximation} we introduce an abstract framework for dynamical approximation of PDEs defined on Hilbert spaces. In \cref{sec:dyn-approx-obs} we then reformulate this abstract approximation to one that is implementable and perform an error analysis on this scheme. In \cref{sec:dyn-sampling} we detail how the time-dependent samples evolve. In \cref{sec:summary} we give an overview of the complete scheme and in \cref{sec:numerical results} we show numerical results.

\section{Nonlinear Dynamical Approximation}
\label{sec:nonlinear dynamical approximation}
Our goal is to solve Partial Differential Equations (PDEs) of the form
\begin{equation}
  \label{eq:PDE}
  \begin{cases}
     & \frac{\d}{\d t}(\cJ u(t)) = F(u(t)) \qquad \text{for a.e.~}t \in (0,T), \\
     & u(0) = u_0,
  \end{cases}
\end{equation}
where $u: [0,T] \to V$, and $(V, \inner{\cdot,\cdot}_V)$ is a Hilbert space of functions on a spatial domain $\Omega \subseteq\bR^d$ with $d\in \bN$, with inner product $\inner{\cdot,\cdot}_V$ and induced norm $\norm{v}_V\coloneqq \inner{v,v}_V^{1/2}$. We assume a Gelfand triple $V\hookrightarrow H\hookrightarrow V'$ and denote by $\cJ:H\to V'$ the embedding defined by $\cJ(v)(w)=\inner{v,w}_H$ for $w\in V$. The function $F:\dom(F)\subseteq V\to V'$ takes values in the dual $V'\coloneqq\Lin(V;\bR)$. The time derivative in \cref{eq:PDE} is understood in the distributional sense, with values in $V'$.

The main idea of dynamical approximation is to approximate the solution $u(t)$ with a parametric class of functions where the parameters evolve in time. The parametric functions are usually expressed in the form of decoder mappings as we introduce in \cref{sec:decoders}. We then give an abstract framework for  dynamical approximation of \cref{eq:PDE} in \cref{sec:dynapprox}. Although the paradigm is not new, the generality in which the setting is expressed constitutes a novelty with respect to prior works. In \cref{sec:L2-setting}, we explain why our abstract framework needs further manipulations to make it fully implementable in practice. This section will motivate the dynamical scheme which we develop in \cref{sec:dyn-approx-obs}.

\subsection{Approximation classes generated by decoders}
\label{sec:decoders}

We start by introducing some notation.
For the ambient Hilbert space $V$ and for any $v\in V$ and any $n$-dimensional linear subspace $H_n\subset V$,
$$
  P_{H_n} v\coloneqq {\argmin_{h\in H_n}} \norm{v-h}_V
$$
denotes the orthogonal projection of $v$ onto $H_n$. We denote $\bG_V:V\to V'$ the isomorphism that connects any given linear functional $\ell\in V'$ with its corresponding unique Riesz representer $\bG_V^{-1}(\ell)\in V$. The representer is such that
$$
  \ell(v) = \inner{\bG_V^{-1}(\ell), v}_V, \quad \forall v \in V, \,\,\,\text{and}\,\,\, \norm{\ell}_{V'} = \sup_{v\in V} \frac{\abs{\ell(v)}}{\norm{v}_V} = \norm{\bG^{-1}_V(\ell)}_V.
$$

Our approximation scheme is based on decoder mappings of the form
\begin{equation}
  \label{eq:decoder}
  {\varphi: \Theta \to V},
\end{equation}
{where $\Theta$ is a Riemannian manifold of dimension $n$. By far, the most common choice is just the flat space $\Theta=\bR^n$, but we keep this level of generality to allow working in more general settings that are useful (for example, the manifold of covariance matrices or the manifold of roto-translations). The image set}
\begin{equation}
  \Vn \coloneqq \varphi(\Theta)=\{ \varphi(\theta) \cond \theta \in \Theta \} \subseteq V
\end{equation}
{can be viewed as a subset of $V$ parametrized by the elements $\theta\in \Theta$. Two typical examples of decoder mappings are the following.}
{
\begin{itemize}
  \item \textbf{Neural networks:} This type of decoder will be our choice for the numerical experiments presented in \cref{sec:numerical results}. The most basic instance of this type of decoder is a Shallow Neural Network (SNN), which is of the form
        \begin{equation}
          \label{eq:SNN decoder}
          \varphi^{\SNN}(\theta) = \sum_{i=1}^p c_i \sigma(a_i, b_i),
        \end{equation}
        where $\theta = \prt*{(a_i,b_i,c_i)}_{i=1}^p$ and the spaces where $(a_i,b_i,c_i)$ take values that may vary depending on the architecture. For example, if $V=L^2(\Omega)$, we typically have $(a_i,b_i,c_i)\in \bR^d\times \bR\times \bR$ so that $\Theta= \bR^{n}$ with $n=p(d+2)$. The function $\sigma:\bR^d\times \bR\to \cF(\Omega, \bR)$ is the so-called activation function, where $\cF(\Omega, \bR)$ denotes the family of functions from $\Omega$ to $\bR$.
  \item \textbf{Polynomial-type decoders:} Another important class of decoders are those of polynomial type. Choosing $\Theta=\bR^n$, we consider a polynomial ansatz
        \begin{equation}
          \label{eq:poly decoder}
          \varphi^{\poly}(\theta) = \sum_{\nu\in E } \theta^\nu v_\nu,
        \end{equation}
        where $\nu = (\nu_1,\dots, \nu_n)\in \bN_0^n$ is a multi-index of order $\abs{\nu}=\nu_1+\dots+\nu_n$ and
        $$
          \theta^\nu =\theta_1^{\nu_1}\dots \theta_n^{\nu_n},
        $$
        and the set $E$ is a finite set of multi-indices of $\bN_0^n$. The index set and the functions $v_\nu$ are chosen a priori. This construction generates the approximation space
        $$
          \Vn^{\poly} = \varphi^\poly(\Theta) = \set*{\varphi^{\poly}(\theta) \cond \theta\in \Theta },
        $$
        which is of dimension $\dim(\Vn^{\poly})=n$. We remark that this decoder class includes the important case of linear combinations of basis functions, since we can choose $E = \set{e_i}_{i=1}^n$, where the $e_i$ are the canonical vectors of $\bR^n$.
\end{itemize}}

In general, the mapping $\varphi$ is differentiable but not injective, and the differential of $\varphi$ at a point $\theta\in \Theta$ is a linear mapping $(D\varphi)_\theta \in \Lin(T_\theta \Theta, V)$ which maps elements from the tangent space $T_\theta \Theta$ to $V$. In the following, we write the image set of this mapping as
$$
  T_{\varphi(\theta)} \Vn \coloneqq (D\varphi)_\theta(T_\theta\Theta).
$$
The notation suggests to interpret the image as a tangent space of the set $\Vn$ at $\varphi(\theta)$ when we view $\Vn$ as an $n$-dimensional manifold embedded in $V$. By the rank theorem, $\dim T_{\varphi(\theta)} \Vn \leq n$. In addition, when $\varphi$ is not injective, the dimension of $T_{\varphi(\theta)} \Vn$ will vary depending on the base point $\theta$. In the following, we denote by
$$
  n_{\eff}(\theta) \coloneqq \dim T_{\varphi(\theta)} \Vn \leq n
$$
the actual, effective dimension of the space. For subsequent developments, we also introduce the adjoint $\prt*{D\varphi}_\theta^*: T_{\varphi(\theta)}\Vn \to T_\theta \Theta$ defined as
\begin{equation}
  \label{eq:adjoint}
  \inner{ \prt*{D\varphi}_\theta(\dot \theta), \dot v }_V
  =
  \inner*{ \dot \theta, \prt*{D\varphi}_\theta^*(\dot v) }_{T_\theta \Theta},
  \quad \forall  (\dot \theta, \dot v) \in T_\theta \Theta \times T_{\varphi(\theta)} \Vn,
\end{equation}
and emphasize that its definition depends on the inner products in $T_\theta \Theta$ and $T_{\varphi(\theta)} \Vn$. Note that for the inner product on $T_{\varphi(\theta)} \Vn$ we take the inner product from $V$, as we explicitly consider this tangent space as a subset of the ambient space $V$.

\subsection{Dynamical Approximation Scheme: Abstract Framework}
\label{sec:dynapprox}

\subsubsection{Main strategy}
The main strategy in dynamical approximation schemes is to construct a
parameter curve $\theta:[0,T]\to\Theta$ such that the \emph{velocity}
of the numerical trajectory $\varphi(\theta(t))$ approximates that of
the exact solution $u(t)$. Since the approximation is performed at the
level of time derivatives, the space in which these derivatives exist
matters. For an evolution such as \eqref{eq:PDE}, the exact solution
is assumed to have only a $V'$-valued derivative, in the sense expressed
in \eqref{eq:PDE}. We distinguish below a strong formulation, requiring
additional regularity of the exact solution, and a fully weak formulation.
In both numerical schemes, we use a decoder
$\varphi\in C^1(\Theta;V)$ and a parameter curve
$\theta\in C^1([0,T];\Theta)$ for simplicity of exposition.

\subsubsection{Strong Formulation}
\label{sec:riesz-formulation}
We assume that $u\in W^{1,1}(0,T;V)$ and that
$F(u(t))\in\cJ(V)$ for almost every $t\in(0,T)$. Then $u$ has a
$V$-valued time derivative $\dot{u}\in V$ satisfying
\[
  \frac{\d}{\d t}(\cJ u(t)) = \cJ\dot u(t)
  \qquad\text{for a.e.\ }t\in(0,T).
\]
Consequently, the weak evolution equation \cref{eq:PDE} is equivalent
to the strong formulation
\begin{equation}
  \label{eq:PDE-riesz}
  \begin{cases}
    \dot u(t) = f(u(t)) & \text{for a.e.\ }t\in(0,T), \\
    u(0) = u_0,
  \end{cases}
\end{equation}
where
\[
  f\coloneqq(\cJ\vert_V)^{-1}\circ F
  \quad\text{on}\quad
  \dom(f)\coloneqq\{v\in \dom(F):F(v)\in\cJ(V)\}.
\]
Here $(\cJ\vert_V)^{-1}$ is defined only on $\cJ(V)$, and we also
assume that $\varphi(\Theta)\subseteq \dom(f)$.

In this formulation, the numerical trajectory
$z(t)\coloneqq\varphi(\theta(t))$ must have a $V$-valued time
derivative. For example, it suffices to require
$z\in W^{1,1}(0,T;V)$, which follows if
$\theta\in W^{1,1}(0,T;\Theta)$ and $\varphi\in C^1(\Theta;V)$.
Under the $C^1$ assumptions adopted here, the chain rule gives
\begin{equation}
  \frac{\d}{\d t}\varphi(\theta(t))
  =
  \prt*{D\varphi}_{\theta(t)}\big(\dot{\theta}(t)\big)
  \in \Tn(t)\coloneqq T_{\varphi(\theta(t))}\Vn.
\end{equation}
The strategy (sometimes called the Dirac-Frenkel principle) is then
to find the best approximation of $f(\varphi(\theta(t)))$ in $\Tn(t)$.
This translates into searching for the unique minimizer of
\begin{equation}
  \label{eq:dirac-frenkel}
  \prt*{D\varphi}_{\theta(t)}\big(\dot{\theta}(t)\big)
  =
  \argmin_{\dot v\in\Tn(t)}
  \norm{\dot v-f(\varphi(\theta(t)))}_V^2,
  \quad \forall t\in(0,T],
\end{equation}
which means that
\begin{equation}
  \prt*{D\varphi}_{\theta(t)}\big(\dot{\theta}(t)\big)
  =
  P_{\Tn(t)}f(\varphi(\theta(t))),
\end{equation}
and for $t=0$,
\begin{equation}
  \label{eq:approx-t0}
  \varphi(\theta(0))
  \in
  \argmin_{v\in\Vn}\norm{v-u(0)}_V^2.
\end{equation}
The necessary optimality conditions of \cref{eq:dirac-frenkel} read
\begin{equation}
  \label{equ:gradient_flow_Hilbert_DF_tested}
  \left\langle
  \prt*{D\varphi}_{\theta(t)}\big(\dot{\theta}(t)\big),
  \dot v
  \right\rangle_V
  =
  \left\langle f(\varphi(\theta(t))),\dot v\right\rangle_V,
  \quad \forall\dot v\in\Tn(t).
\end{equation}
Using that
$\Tn(t)=\text{span}\{
  \prt*{D\varphi}_{\theta(t)}(e_1),\dots,
  \prt*{D\varphi}_{\theta(t)}(e_n)\}$,
where $\{e_i\}_{i=1}^n$ is a basis of $T_{\theta(t)}\Theta$,
and using the property from \cref{eq:adjoint} of the adjoint,
\cref{equ:gradient_flow_Hilbert_DF_tested} is equivalent to
\begin{equation*}
  \inner*{
    \prt*{
      \prt*{D\varphi}_{\theta(t)}^*
      \circ
      \prt*{D\varphi}_{\theta(t)}
    }\big(\dot{\theta}(t)\big),
    e_i
  }_{T_{\theta(t)}\Theta}
  =
  \inner*{
    \prt*{D\varphi}_{\theta(t)}^*
    \prt*{f\big(\varphi(\theta(t))\big)},
    e_i
  }_{T_{\theta(t)}\Theta},
  \quad \forall i=1,\dots,n,
\end{equation*}
which is itself equivalent to
\begin{equation}
  \label{equ:gradient_flow_Hilbert_normal}
  \prt*{
    \prt*{D\varphi}_{\theta(t)}^*
    \circ
    \prt*{D\varphi}_{\theta(t)}
  }\big(\dot{\theta}(t)\big)
  =
  \prt*{D\varphi}_{\theta(t)}^*
  \Big(f\big(\varphi(\theta(t))\big)\Big),
  \quad \forall t\in(0,T].
\end{equation}
This formulation is interesting because it leads to an evolution
in the finite-dimensional parameter space $\Theta$.

\subsubsection{Weak formulation and connections with the strong formulation}
\label{sec:fully-weak-formulation}
In this approach, we do not require the exact solution to have a
$V$-valued time derivative. We retain the evolution equation in its
weak form and compare the derivatives in $V'$, namely
$\frac{d}{dt}(\cJ u(t))$ and
$\frac{d}{dt}\bigl(\cJ\varphi(\theta(t))\bigr)$.
In full generality, one should differentiate the $V'$-embedded decoder
$\widetilde{\varphi}=\cJ\circ\varphi: \Theta \to V'$ and minimize over its image
tangent space in $V'$, without requiring a $V$-valued numerical
velocity. We do not develop this broad construction here.
Instead, we retain the $C^1$ assumptions on $\varphi$ and $\theta$,
so that the numerical trajectory has a $V$-valued derivative.
This additional regularity of the approximation imposes no
corresponding requirement on the exact solution.

The chain rule gives
\[
  \frac{d}{dt}\bigl(\cJ\varphi(\theta(t))\bigr)
  =
  \cJ\vert_V\bigl((D\varphi)_{\theta(t)}\dot\theta(t)\bigr).
\]
Thus, minimizing the residual in $V'$ and assuming that $\varphi(\Theta)\subseteq \dom(F)$, we are led to
\begin{align}
  \prt*{D\varphi}_{\theta(t)}\big(\dot{\theta}(t)\big)
   & =
  \argmin_{\dot v\in\Tn(t)}
  \norm{\cJ\vert_V(\dot v)-F(\varphi(\theta(t)))}_{V'}^2
  \nonumber \\
   & =
  \argmin_{\dot v\in\Tn(t)}
  \norm{\bG_V^{-1}\big(
    \cJ\vert_V(\dot v)-F(\varphi(\theta(t)))
    \big)}_V^2.
  \label{eq:dirac-frenkel-weak}
\end{align}

The strong evolution equation is a special case of the weak one when $u\in W^{1,1}(0,T;V)$ and $F(u(t))\in\cJ(V)$ almost everywhere in time. The two schemes may nevertheless yield different numerical approximations because they impose different approximation conditions. Their relationship is encoded by the bounded, positive, self-adjoint operator $A\coloneqq\bG_V^{-1}\cJ\vert_V:V\to V$, which satisfies
\[
  \inner{A(u),v}_V=\inner{u,v}_H,
  \qquad \forall u,v\in V.
\]
Whenever $\varphi(\theta(t))\in \dom(f)$, we have
$F(\varphi(\theta(t)))=\cJ f(\varphi(\theta(t)))$,
and the minimization problem can be rewritten as
\[
  \prt*{D\varphi}_{\theta(t)}\big(\dot{\theta}(t)\big)
  =
  \argmin_{\dot v\in\Tn(t)}
  \norm{A\big(\dot v-f(\varphi(\theta(t)))\big)}_V^2.
\]
Thus, the strong formulation measures the velocity residual directly
in $V$, whereas the weak formulation applies $A$ to the entire
residual before measuring it in $V$. In particular, both approaches
coincide whenever $A=\id$, i.e., when the inner products of $V$ and
the pivot space $H$ agree on $V$ (in particular, when
$V=H=L^2(\Omega)$).

\subsection{The setting \texorpdfstring{$\Theta=\bR^n$}{Theta=Rn} and \texorpdfstring{$V=L^2(\Omega)$}{V=L2} and challenges for numerical implementation}
\label{sec:L2-setting}
Unfortunately, the dynamical scheme from \cref{eq:dirac-frenkel} or \cref{eq:dirac-frenkel-weak} is not implementable in practice. To highlight the difficulties, we next detail one of the most common settings which occurs when $\Theta=\bR^n$ and $V=L^2(\Omega)$. Taking $H=V$, both approaches coincide. Also, we can take the canonical basis $\{e_i\}_{i=1}^n$ of $\bR^n$ as a basis of $T_\theta \Theta$ for every $\theta \in \bR^n$, and
\begin{equation}
  \Tn(t) = \vspan\{ \partial_{\theta_i} \varphi(\theta) \}_{i=1}^n,
\end{equation}
Defining the $n\times n$ matrix
\begin{equation}
  \label{eq:M-continuous}
  M(\theta) \coloneqq
  \prt*{ \inner{\partial_{\theta_i} \varphi(\theta), \partial_{\theta_j}\varphi(\theta) }_{L^2(\Omega)} }_{1\leq i,j\leq n},
\end{equation}
and the right-hand side vector
\begin{equation}
  \label{eq:q-continuous}
  q(\theta) \coloneqq
  \prt*{\inner{\partial_{\theta_i} \varphi(\theta), f(\varphi(\theta(t)))}_{L^2(\Omega)}}_{1\leq i\leq n},
\end{equation}
\cref{equ:gradient_flow_Hilbert_normal} leads to
\begin{equation}
  M(\theta(t)) \dot\theta = q(\theta(t)). \label{eq:ode}
\end{equation}
This is an ODE system for the dynamics of the parameters $\theta(t)$ which opens the door to a practical numerical implementation by using time-integration schemes. However, there are still two main caveats to address:
\begin{enumerate}
  \item To assemble the matrix $M(\theta(t))$, we need to evaluate inner products in $L^2(\Omega)$. This requires sampling quadrature points. If selected naively, for example, in a uniform grid, the amount of points may lead to the curse of dimension. Strong instabilities may also arise. The main focus of our work is to address this challenge with the dynamical sampling strategy that we present in \cref{sec:dyn-approx-obs,sec:dyn-sampling}.
  \item An even deeper challenge is the following: since in general $\varphi$ is not injective, the effective dimension of $\Tn(t)$ is
        \begin{equation}
          \label{eq:effective-n}
          n_\eff(t) \coloneqq \dim\Tn(t) \leq n,
        \end{equation}
        and it can happen that $M(\theta(t))$ does not have full rank at certain times. In that case, the evolution in \cref{equ:gradient_flow_Hilbert_normal} becomes a so-called system of Differential Algebraic Equations (DAEs) (see \cite{hairer1996solving}), and the question of existence of solutions to the ODE arises. For very specific cases it is possible to guarantee existence as in for example \cite{bachmayr2021existence}. However one requires much more structure in the approximation class and the type of problem to solve. For very general situations, current suggestions on how to circumvent this issue are based on regularization (see \cite{feischl2024regularized, berman2023randomized}). In our practical implementation, we follow this approach, and replace $M(\theta(t))$ by
        \begin{equation}
          \label{eq:regulatization}
          M(\theta(t))+\eps \id,
        \end{equation}
        with $\eps>0$ when $M(\theta(t))$ becomes rank deficient.
\end{enumerate}

\begin{remark}
  \label{rem:generalization-ode}
  In the general Hilbert space setting, the ODE in \cref{eq:ode} remains valid with only minor modifications. In the strong formulation, it suffices to replace the inner products by those of $V$. In the weak formulation, one additionally replaces the partial derivatives $\partial_{\theta_i}\varphi(\theta)$ by $A\partial_{\theta_i}\varphi(\theta)$ and the right-hand side $f(\varphi(\theta))$ by $\bG_V^{-1}F(\varphi(\theta))$. The same challenges arise in this more general setting.
\end{remark}

\section{Dynamical Approximation with linear observations}
\label{sec:dyn-approx-obs}
In the following, we use the strong formulation from \cref{sec:riesz-formulation}, with $u\in W^{1,1}(0,T;V)$ solving \cref{eq:PDE-riesz} and $\varphi(\Theta)\subseteq \dom(f)$.

\subsection{Formulation}
\label{sec:dyn-approx-formulation}
To make the above framework implementable and guarantee stability, we take \cref{eq:dirac-frenkel} as a starting point, namely,
\begin{equation}
  \label{eq:dirac-frenkel-reminder}
  \prt*{D\varphi}_{\theta(t)} \big(\dot{\theta}(t)\big) \in \argmin_{\dot{v}\in \Tn(t)} \norm{ \dot{v} - f(\varphi(\theta(t))) }^2_V, \quad \forall t \in (0,T].
\end{equation}
Our goal is to replace the continuous norm $\norm{\cdot}_V$ by a surrogate which can be evaluated in practice. To account for this, we proceed as follows. Let $\{\ell_i\}_{i=1}^m$ be a set of $m$ independent linear functionals from $V'$ with Riesz representers $\{\omega_i\}_{i=1}^m\in V$. Note that $\omega_i \coloneq \bG_V^{-1}(\ell_i)$. In the following, we denote
$$
  \ell(v)=(\ell_i(v))_{i=1}^m \in \bR^m, \quad \forall v\in V,
$$
and the linear subspace spanned by the representers
\begin{equation}
  \label{eq:observation-space}
  W_m = \vspan\{ \omega_i \}_{i=1}^m,
\end{equation}
which is of dimension $m$ since we are assuming that the $\{\ell_i\}_{i=1}^m$ are independent. As an alternative to \cref{eq:dirac-frenkel}, we consider
\begin{equation}
  \label{eq:dirac-frenkel-sampled}
  \prt*{D\varphi}_{\theta(t)} \big(\dot{\theta}(t)\big) \in \argmin_{\dot{v}\in \Tn(t)} \abs{ \ell(\dot{v}) - \ell\prt*{f(\varphi(\theta(t)))} }^2_Z, \quad \forall t \in (0,T],
\end{equation}
where $\abs{\cdot}_Z$ is a norm in $\bR^m$. Next, we obtain the approximate solution $\varphi(\theta(t))$ by time integration of \cref{eq:dirac-frenkel-sampled}. In the formulation, one could choose any norm for $\abs{\cdot}_Z$ since they are all equivalent. However, there is a notion of optimality in relation to the approximation error bounds as we explain later on.

By switching the point of view from \cref{eq:dirac-frenkel-reminder} to \cref{eq:dirac-frenkel-sampled}, we are expressing the fact that we are dealing with partial information about the dynamics: we now ``observe'' the evolution through the lens of linear functional evaluations. This is why we call $W_m$ the \emph{observation space}.

\paragraph{Example} To illustrate the quantities that are involved in formulation \cref{eq:dirac-frenkel-sampled}, suppose we are in the setting from \cref{sec:L2-setting} where $\Theta=\bR^n$ and $V=L^2(\Omega)$. {Defining
    \begin{equation}
      \label{eq:parametrized exp}
      g(\Sigma, \mu) : \Omega \to \bR, \,\, x \mapsto (2\pi)^{-d/2} \det(\Sigma)^{-1/2} \exp \prt*{-\frac 1 2 (x-\mu)^\top \Sigma^{-1} (x-\mu)},
    \end{equation}
    we choose $\omega_i=g(\sigma^2 \id, x_i)$. Our linear functionals are thus}
\begin{equation}
  \label{eq:Gaussian observations}
  \ell_i(v)=\int_\Omega g(\sigma^2 \id, x_i)(y)v(y)\d y=v_\sigma(x_i),
\end{equation}
{which} is a local {(Gaussian)} average of $v\in V$ around the point $x_i$. When $v\in \cC(\Omega)$ is continuous, $\ell_i(v)=v_\sigma(x_i)\to v(x_i)$ as $\sigma\to 0$ so when $\sigma$ is small, $\ell_i(v)$ can be seen as an approximation of $v$ at $x_i$. If $\abs{\cdot}_Z$ is chosen as a weighted Euclidean norm of $\bR^m$ of the form ${\abs{p}_Z^2 \coloneqq \sum_{i, j=1}^m \kappa_{i,j} p_i p_j}$ with weights {$K = \prt*{\kappa_{i,j}}_{i,j=1}^m\in \bR^{m\times m}$ such that $K$ is symmetric positive definite}, then \cref{eq:dirac-frenkel-sampled} becomes
\begin{equation}
  \label{equ:dirac-frenkel-pointwise}
  \prt*{D\varphi}_{\theta(t)} \big(\dot{\theta}(t)\big) \in \argmin_{\dot{v}\in \Tn(t)} {\sum_{i, j=1}^m \kappa_{i,j} \prt*{\dot{v}_\sigma(x_i) - f_\sigma(\varphi(\theta(t)))(x_i)}\prt*{\dot{v}_\sigma(x_j) - f_\sigma(\varphi(\theta(t)))(x_j)}}, \quad \forall t \in (0,T].
\end{equation}
Consequently, the loss function in \cref{equ:dirac-frenkel-pointwise} can be understood as a weighted and sampled version of an $L^2(\Omega)$ norm. The optimality conditions of \cref{equ:dirac-frenkel-pointwise} lead to the ODE system
\begin{equation}
  \label{eq:sampled-M}
  \widehat{M}_{\sigma, m}(\theta(t)) \dot\theta = \hat q_{\sigma, m} (\theta(t)),
\end{equation}
where
$$
  \widehat{M}_{\sigma, m}(\theta) \coloneqq \prt*{ {\sum_{k, \ell=1}^m \kappa_{k, \ell} \prt*{\partial_{\theta_i} \varphi(\theta)}_\sigma (x_k) \prt*{\partial_{\theta_j}\varphi(\theta)}_\sigma (x_\ell)} }_{1\leq i,j\leq n}
$$
and
$$
  \hat q_{\sigma, m}(\theta) = \prt*{ {\sum_{k, \ell=1}^m \kappa_{k, \ell} \prt*{\partial_{\theta_i} \varphi(\theta)}_\sigma(x_k) f_\sigma(\varphi(\theta(t)))(x_\ell)} }_{1\leq i\leq n}.
$$
When $\nabla_\theta \varphi(\theta)$ and $f$ are continuous in the spatial variable,
$$
  \widehat{M}_{\sigma, m}(\theta) \to M(\theta), \quad \text{and} \quad
  \hat{q}_{\sigma, m}(\theta) \to q(\theta),
  \text{ as } \sigma\to 0 \text{ and }m\to+\infty .
$$
Therefore, we are approximating the integrals of $M(\theta)$ and $q(\theta)$ by a sampled version of the $L^2(\Omega)$ norm. The sampled norm comes expressed with a mollification of point evaluations. This is due to the fact that Dirac delta functions are not in $V'$, so we cannot choose $\ell_i = \delta_{x_i}$. Pointwise evaluation requires that the Hilbert space $V$ is a reproducing kernel Hilbert space (RKHS), that is, a Hilbert space that continuously embeds in $\cC(\Omega)$. Examples of such spaces are the Sobolev spaces $H^s(\Omega)$ for $s>d/2$, possibly with additional boundary conditions.

\paragraph{Complexity}
Let us discuss the numerical complexity in the example. In a practical implementation, one needs to compute the mollified terms
\begin{equation}
  \label{eq:mollified}
  \prt{\partial_{\theta_i} \varphi(\theta)}_\sigma (x_k)
  =
  \int_\Omega \partial_{\theta_i} \varphi(\theta)(y) g(\sigma^2 \id, x_k)(y)\d y, \quad 1\leq i\leq n, \ 1\leq k \leq m,
\end{equation}
and
$$
  f_\sigma (\varphi(\theta))(x_k) = \int_\Omega f(y) \varphi(\theta) g(\sigma^2 \id, x_k)(y)\d y.
$$
In general, such terms can be approximated accurately with efficient quadrature methods involving Gaussian sampling (see, e.g., \cite{chen2018sparse}). Therefore, since in our approach we end up needing sampling points, we may then wonder what the advantage is when compared to approximating the original inner product integrals. There are three main reasons:
\begin{enumerate}
  \item First, since we can view the integrals \cref{eq:mollified} as expectations involing a Gaussian measure, the question about which sampling/quadrature points to choose is much better explored than the sampling for the original integrals of $M(\theta)$. Deterministic quadrature points based on tensorization of the Gauss-Hermite quadrature points give spectral accuracy (provided that the integrand is smooth). This option can be leveraged for moderate dimension $d$ but it suffers from the curse of dimension. For very high dimension, efficiently sampling from a Gaussian distribution is a topic of very active research and state of the art methods can {be} leveraged in the current content  (see, e.g., \cite{chen2018sparse}).
  \item Since, in practice, it is possible to work with very concentrated Gaussians, we can finely steer the location of the samples by choosing the regions where the Gaussian concentrates.
  \item There are notable cases in which \cref{eq:mollified} can be computed analytically:
        \begin{itemize}
          \item If $\Omega=\bR^d$, and we use polynomial decoders $\varphi^\poly(\theta)=\sum_{\nu\in E} \theta^\nu v_\nu$ with polynomial $\theta_\nu\in \bP[\Omega]$ (cf. \cref{eq:poly decoder}). In this case, \cref{eq:mollified} are moments of the Gaussian distributions which can be computed explicitly.
          \item If $\Omega=\bR^d$, and we use a shallow neural network {(cf. \cref{eq:SNN decoder}) with activation function
                    \begin{equation}
                      \label{eq:exponential_activation_function}
                      \sigma: \bR^{d \times d} \times \bR^d \rightarrow \cF(\Omega, \bR), \,\, \sigma(S, \mu)(x) = \exp \prt*{-\frac 1 2 (x-\mu)^\top S^{-1} (x-\mu)}.
                    \end{equation}}
                This is due to the fact that for this case $\varphi(\theta) \propto g({S_i},\mu_i)$ and consequently,
                \begin{equation*}
                  \begin{split}
                    &\partial_{c_i}      \varphi(\theta) \propto g({S_i}, \mu_i), \\
                    &\partial_{\Sigma_i} \varphi(\theta) \propto (\cdot-\mu_i)^\top{S_i}^{-1}(\cdot-\mu_i) g({S_i}, \mu_i), \\
                    &\partial_{\mu_i}    \varphi(\theta) \propto -{S_i^{-1}}(\cdot-\mu_i)g({S_i}, \mu_i).
                  \end{split}
                \end{equation*} 
                In this case, \cref{eq:mollified} involves moments of order 0, 1 and 2 of the product of two Gaussians. So one can derive an analytic expression for the integral.
        \end{itemize}
\end{enumerate}

From this discussion, it follows that to reduce complexity it is crucial to work with as few observations $m$ as possible while guaranteeing that problem \cref{eq:dirac-frenkel-sampled} does not deviate too much from the original problem \cref{eq:dirac-frenkel}. For this, we leverage two ideas:
\begin{enumerate}
  \item Problems \cref{eq:dirac-frenkel-sampled} and \cref{eq:dirac-frenkel} are least squares optimization problems over tangent spaces $\Tn(t)$. Therefore one should adapt the choice of the observation operators $\ell_i$ (and therefore the space $W_m$) to the tangent spaces. The error analysis that we present in the next section reveals that $W_m$ and $\Tn(t)$ are connected through a stability quantity (defined in \cref{eq:beta} below) which is crucial for the approximation quality.
  \item Since the tangent spaces evolve in time, it is natural to ask that the linear functionals $\ell_i$ (and therefore $W_m$ as well) evolve also with the dynamics. This is why we will consider them time-dependent and write $\ell_i(t)$ and $W_m(t)$. In fact, if $W_m$ is kept constant, one can easily construct counter-examples of dynamics transporting spatially localized solutions which would require a very fine sampling of the whole domain $\Omega$, which makes $m$ become very large.
\end{enumerate}

\paragraph{Refined Formulation}
In the light of these remarks, we update formulation \cref{eq:dirac-frenkel-sampled} to make the observations time dependent. We denote the minimizer as
\begin{equation}
  \label{eq:dirac-frenkel-dynamically-sampled}
  \prt*{D\varphi}_{\theta(t)} \big(\dot{\theta}(t)\big) = \dot{v}^*_Z(t) \in \argmin_{\dot{v}\in \Tn(t)} \abs{ \ell(t)(\dot{v}) - \ell(t)\prt*{f(\varphi(\theta(t)))} }^2_Z, \quad \forall t \in (0,T].
\end{equation}
We discuss how the observations are evolving in \cref{sec:dyn-sampling}, and present an error analysis of the scheme in the next section.

\begin{remark}
  \label[remark]{rem:PBDW}
  {Our least squares strategy \eqref{sec:dyn-approx-formulation} and \eqref{eq:dirac-frenkel-dynamically-sampled} takes its roots from the so-called Parametrized-Background Data Weak algorithm (PBDW, see \cite{MPPY2015} for the original paper and \cite{BCDDPW2017, CDDFMN2020, CDMN2022} for further analysis and improvements of the method), and on the Generalized Empirical Interpolation Method (GEIM, see \cite{MM2013, MMPY2015, MMT2016}). Both methods were introduced to solve inverse state estimation tasks using reduced order models. In the language of the current paper, they both work with the optimal norm $Z_\circ$ which we introduce in \cref{eq:optimal-norm-Z}. Even for $Z=Z_\circ$, our current formulation is not exactly the one from PBDW, though. Indeed, directly applying PBDW to our current setting is not possible since the method would provide an approximation of $\prt*{D\varphi}_{\theta(t)(\dot{\theta}(t))} \in T_n(t)$ with an element $\dot{u}^*_Z(t)$ which lies in a larger space, $T_n(t)\oplus (T_n(t)^\perp\cap W_m)$. One can prove that our approximation $\dot{v}^*_{Z_\circ}(t)$ is actually equal to $P_{T_n(t)} \dot{u}_Z^*(t)$. Three other distictive points of the current formulation are the fact that (i) the $\ell(t)$ and $T_n(t)$ are time-dependent, (ii) we are dealing with a forward problem, and (iii) we approximate velocities (instead of directly work with the state itself).}
\end{remark}

\subsection{Error Analysis}
\label{sec:error-analysis}

Our analysis requires defining the quantity
\begin{equation}
  \label{eq:beta}
  \beta_Z(\Hn,\ell) \coloneqq {\inf_{v \in \Hn}}\frac{\abs{\ell(v)}_Z}{\norm{v}_V},
\end{equation}
where $\ell=\{\ell_i\}_{i=1}^m$ is a set of linear functions and $\Hn$ is an $n$-dimensional linear subspace of $V$.
For our scheme, accuracy and stability are driven by an interplay between the observations $\ell(t)$ and the reconstruction space $\Tn(t)$. To reflect on this, we introduce an abstract notion of admissibility for reconstruction in the $Z$ norm using a pair $\{\ell,\Hn\}_Z$ of linear functionals $\ell$ and a generic subspace $\Hn$.

\begin{definition}[Admissibility of $\{\ell, \Hn\}_Z$ for reconstruction]
  Let $\ell=\{\ell_i\}_{i=1}^m$ be a set of $m$ linear functions from $V'$, and $\Hn$ be an $n$-dimensional subspace of $V$.
  We say that the pair $\{\ell, \Hn\}_Z$ is admissible for reconstruction in $Z$ if it satisfies the structural assumptions:
  \begin{itemize}
    \item[(A1)] The observation operator $\ell:(V, \bG_V) \to (\bR^m, \abs{\cdot}_Z)$ is Lipschitz continuous, namely
      \begin{equation}
        \label{eq:ell-Lip}
        \Lip_Z(\ell)= {\sup_{v\in V}}\frac{\abs{\ell(v)}_Z}{\norm{v}_V} <+\infty.
      \end{equation}
    \item[(A2)] \emph{Inverse stability property:} $\beta_Z(\Hn,\ell)>0$.
  \end{itemize}
\end{definition}
The next Proposition shows that a necessary (but not sufficient) condition to satisfy (A2) is that $m\geq n$. {The result holds for all norms $Z$, but its proof relies on a particular choice $Z_\circ$  defined as follows. Since $\ell(t):V\to \bR^m$ is continuous and surjective, we can define a norm in $\bR^m$ through
    \begin{equation}
      \label{eq:optimal-norm-Z}
      \abs{z}_{Z_\circ(t)} \coloneqq \min\{ \norm{v}_V \cond \ell(t)(v)=z \}, \quad \forall z\in \bR^m.
    \end{equation}
    This norm is such that
    $$
      \abs{\ell(t)(v)}_{Z_\circ(t)} = \norm{P_{W_m(t)} v}_V, \quad \forall v \in V,
    $$
    and in this case \cref{eq:beta} reads
    \begin{equation}
      \label{eq:stability-PBDW}
      \beta_{Z_\circ}(\Hn,\ell) = \beta(\Hn,\Wm)\coloneqq \inf_{v\in \Hn} \frac{\norm{P_{W_m} v}_V}{\norm{v}_V}.
    \end{equation}
  }

\begin{proposition}
  \label{prop:n-m-condition}
  If $n> m$, then $\beta_Z(\Hn,\ell)=0$ for all norms $Z$ in $\bR^m$.
\end{proposition}
\begin{proof}
  If $n>m$, then there exists an element $v\in \Hn$ which is orthogonal to $\Wm$. Thus, for this element, $\norm{P_{\Wm}v}_V=\abs{\ell(v)}_{Z_\circ}=0$ and therefore $\beta_{Z_\circ}(\Hn,\ell)=0$. Since all norms in $\bR^m$ are equivalent, $\abs{\ell(v)}_{Z_\circ}=0$ implies that $\abs{\ell(v)}_{Z}=0$ for all norms $Z$. Therefore, $\beta_Z(\Hn, \ell)=0$.
\end{proof}

We are now in position to derive a bound for the error between the exact solution $u(t)$ and the approximation $\varphi(\theta(t))$ given by time integration of \cref{eq:dirac-frenkel-dynamically-sampled}.
As an intermediate step, we derive an upper bound for the approximation error between $f(\varphi(\theta(t)))$ and its approximation $\prt*{D\varphi}_{\theta(t)} \big(\dot{\theta}(t)\big)$ from \cref{eq:dirac-frenkel-dynamically-sampled}. {The result also reveals that working with the norm $Z_\circ$ provides optimal stability.}

\begin{theorem}[Error bound on the time derivative]
  \label{lem:error-bound-f}
  Let $t\geq0$. If $\{\ell(t), \Tn(t)\}_Z$ is admissible for reconstruction in $Z$, then
  \begin{align}
    \norm{f(\varphi(\theta(t))) - \prt*{D\varphi}_{\theta(t)} \big(\dot{\theta}(t)\big) }_V & \leq \kappa(t)\coloneqq
    C_Z(t) \min_{\dot v \in \Tn(t)} \norm{\dot v - f(\varphi(\theta(t)))}_V,
    \label{eq:error-bound-f}
  \end{align}
  {with
  \begin{equation}
    \label{eq:CZ}
    C_Z(t) \coloneqq
    1 + 2\Lip_Z(\ell(t)) \beta^{-1}_Z(\Tn(t), \ell(t)),
  \end{equation}
  and $Z_\circ$ is a norm providing optimal stability in the sense that for any norm $Z$,
  \begin{equation}
    \label{eq:optimal-norm}
    C_{Z_\circ}(t)\leq C_Z(t), \quad \forall t\geq0.
  \end{equation}
  }
\end{theorem}

\begin{proof}
  Throughout the proof, we use the shorthand notation $f=f(\varphi(\theta(t)))$,  $z(t)=\ell(t)(f)$. Let $\dot v\in \Tn(t)$, and apply the triangle inequality to obtain
  \begin{align}
    \norm*{f- \dot{v}^*_Z(t)}_V & \leq \norm*{f- \dot v}_V + \norm*{\dot v- \dot{v}^*_Z(t)}_V                                                                                                       \\
                                & \leq \norm*{f- \dot v}_V + \beta_Z^{-1}(\Tn(t), \ell(t)) \abs{\ell(t)(\dot v)-\ell(t)(\dot{v}^*_Z(t))}_Z\quad \text{(by (A2)).} \label{eq:err-bound-intermediate}
  \end{align}
  Applying the triangle inequality again and chaining the minimizing property of \cref{eq:dirac-frenkel-dynamically-sampled} with \cref{eq:ell-Lip},
  \begin{equation*}
    \begin{aligned}
      \abs{\ell(t)(\dot v)-\ell(t)(\dot{v}^*_Z(t))}_Z
       & \leq \abs{z(t)-\ell(t)(\dot v)}_Z + \abs{z(t)-\ell(t)(\dot{v}^*_Z(t))}_Z                         \\
       & \leq 2 \abs{z(t)-\ell(t)(\dot v)}_Z\quad \text{(by \cref{eq:dirac-frenkel-dynamically-sampled})} \\
       & \leq 2\Lip_Z(\ell(t)) \norm*{f- \dot v}_V\quad \text{(by \cref{eq:ell-Lip}).}
    \end{aligned}
  \end{equation*}
  Inserting the last inequality into \cref{eq:err-bound-intermediate}, we get
  \begin{equation*}
    \begin{aligned}
      \norm*{f- \dot{v}^*_Z(t)}_V & \leq (1 + 2\Lip_Z(\ell(t)) \beta_Z^{-1}(\Tn(t), \ell(t))) \norm*{f- \dot v}_V.
    \end{aligned}
  \end{equation*}
  and \cref{eq:error-bound-f} follows by minimizing over $\dot v\in \Tn(t)$.

  Finally, we prove \cref{eq:optimal-norm}, which will show that $Z_\circ$ provides optimal stability. This result is a straightforward application of \cite[Theorem 3.1]{CDMS2022}. We recall the proof for self-completeness and ease of reading. We start by remarking that $\Lip_{Z_\circ}(\ell)=1$. To show that $C_{Z_\circ}(t)\leq C_Z(t)$, we are going to prove that $ \beta_{Z_\circ}^{-1}(T_n(t), \ell(t))\leq \Lip_{Z}(\ell) \beta_{Z}^{-1}(T_n(t), \ell(t))$. For this, we bound,
  \begin{align}
    \beta_{Z_\circ}^{-1}(T_n(t), \ell(t))
     & = \max_{v\in \Tn(t)} \frac{\norm{v}_V}{\norm{P_{\Wm(t)}v}_V} \\
     & \leq
    \max_{v\in \Tn(t)} \frac{\abs{\ell(t)(v)}_Z}{\abs{\ell(t)(v)}_{Z_\circ}} \max_{v\in \Tn(t)} \frac{\norm{v}_V}{\abs{\ell(t)(v)}_Z}
    =
    \max_{v\in \Tn(t)} \frac{\abs{\ell(t)(v)}_Z}{\abs{\ell(t)(v)}_{Z_\circ}} \beta^{-1}_Z(\Tn(t), \ell(t)). \label{eq:ineq1}
  \end{align}
  From the definition of the $Z_\circ$ norm,
  \begin{equation}
    \max_{v\in \Tn(t)} \frac{\abs{\ell(t)(v)}_Z}{\abs{\ell(t)(v)}_{Z_\circ}} \leq \max_{z\in \bR^m} \frac{\abs{z}_Z}{\abs{z}_{Z_\circ}} = \max_{z\in \bR^m} \max_{\ell(t)(v)=z} \frac{\abs{z}_Z}{\norm{v}_V} = \Lip_{Z}(\ell(t))
    \label{eq:ineq2}
  \end{equation}
  Chaining \cref{eq:ineq1} with \cref{eq:ineq2} yields  $\beta_{Z_\circ}^{-1}(T_n(t), \ell(t))\leq \Lip_{Z}(\ell(t))\beta_{Z}^{-1}(T_n(t), \ell(t))$.
\end{proof}

We next leverage \cref{lem:error-bound-f} to derive the main result of this section which is an error bound of the approximation of $u(t)$ with $\varphi(\theta(t))$. In the following, we denote
$$
  e(t) = \norm{u(t) - \varphi(\theta(t))}_V,\quad \forall t\geq0.
$$
The convolution in time
$$
  (\Phi_a*\kappa)(t) \coloneqq \int_0^t \Phi_a({t-s})\kappa({s})\d s, \quad \forall t\geq0
$$
of the error upper bound $\kappa$ from \cref{eq:error-bound-f} with exponential filters of the form
$$
  \Phi_a(t)\coloneqq \exp\prt*{{a} t}, \quad \forall (a,t)\in \bR\times\bR_+
$$
will play a central role in the analysis. The result is obtained under the assumption that $f$ is one-sided Lipschitz continuous, which means that there exists ${L \in \bR}$ such that
\begin{equation}
  \label{eq:one-sided-Lip}
  \inner{f(v)-f(w), v-w}_V \leq L \norm{v-w}^2_V, \quad \forall (v,w)\in \dom(f)\times \dom(f).
\end{equation}

\begin{theorem}[Error bound between $u(t)$ and $\varphi(\theta(t))$] \label{theorem:best-fit-estimator}
  Let $t\geq0$. If $f$ is one-sided $L$-Lipschitz continuous, $t \to \kappa(t)$ is continuous, and $\{\ell(s), \Tn(s)\}_Z$ are admissible for reconstruction for all $0\leq s\leq t$, then
  \begin{equation}
    \label{eq:error-bound}
    e(t)\leq e(0) \Phi_L(t) + (\Phi_L * \kappa)(t),
  \end{equation}
  which, in a more expanded notation, reads:
  \begin{align}\label{eq:error-bound-Z}
    e(t) & \leq
    e(0) \exp\prt*{{L} t}  +
    \int_0^t \exp\prt*{{L} (t-s)} C_Z(s) \inf_{\dot v \in T_n(s)} \norm{\dot v - f(\varphi(\theta(s)))}_V \d s.
  \end{align}
\end{theorem}

\begin{proof}
  For almost every $t$, we have
  \begin{align}
    {\frac{1}{2}} \frac{\d e^2}{\d t} (t)
     & =  {e(t) \frac{\d e}{\d t}} =  \inner*{\frac{\d}{\d t} (u - \varphi(\theta(t))), u - \varphi(\theta(t)) }_V                                                                    \\
     & = \inner*{f(u(t)) -f(\varphi(\theta(t)))+f(\varphi(\theta(t))) -\frac{\d}{\d t} \varphi(\theta(t)), u - \varphi(\theta(t)) }_V, \quad \text{(by \cref{eq:PDE-riesz})}\nonumber \\
     & \leq L e^2(t) +
    \norm*{f(\varphi(\theta(t))) -\frac{\d}{\d t} \varphi(\theta(t))}_V e(t), \quad \text{(by \cref{eq:one-sided-Lip})}                                                               \\
     & \leq L e^2(t) +
    \kappa(t) e(t), \quad \text{(by \cref{eq:error-bound-f})}.
  \end{align}
  {Since $e$ is absolutely continuous, we may divide by $e(t)$ where $e(t)>0$; moreover, $e'(t)=0$ for almost every $t$ where $e(t)=0$. Thus, almost everywhere,
  \begin{equation*}
    \frac{\d e(t)}{\d t} \leq L e(t) + \kappa(t),
  \end{equation*}
  to which we can apply the Grönwall lemma to get the desired result.}
\end{proof}

\cref{theorem:best-fit-estimator} is the central result that gives theoretical guarantees of the approximation quality. Since it holds under the condition that the mapping $f$ is one-sided Lipschitz continuous, it is important to discuss the PDEs for which this condition holds. For parabolic PDEs, this condition must be verified for the strong realization $f$ in the chosen $V$-inner product; see \cref{sec:example-PDEs} for examples. Furthermore, by adopting a slightly different point of view in our main setting and conditions on $f$, we show that the same result holds with $L<0$ for gradient flows with $\lambda$-convex potentials, $\lambda>0$.
This is a particularly strong result because, while errors accumulate in time for a general evolution, in the gradient flow case, the effect of previous errors decays exponentially fast, and the error at time $t$ is essentially due to the approximation quality at that time. {The result is valuable both for physical systems that have a gradient flow structure, and for the minimization of $\lambda$-convex of functionals. We briefly explain the main setting here, and provide full details in \cref{sec:example-PDEs}. Taking $H=V$ with the same inner product, so that $\cJ=\bG_V$, consider a proper, lower semicontinuous, $\lambda$-convex functional $\cF:V\to (-\infty, +\infty]$ with $\lambda>0$. The associated gradient flow evolution is given by
    \begin{equation}
      \dot u(t) \in - \grad_V \cF(u(t)), \quad \text{for a.e.~} t>0,
    \end{equation}
    with $u(0)=u_0$ for a given initial condition $u_0\in\dom(\cF)$, and where $\grad_V \cF(u(t))$ is the subgradient of $\cF$ at $u(t)$ in the sense that we define in \cref{sec:example-PDEs}. This dynamics converges as $t\to\infty$ to the unique minimizer $u_\infty$ of the optimization problem
    \begin{equation}
      u_\infty \in \argmin_{v \in V} \cF(v).
    \end{equation}
    Therefore, approximately solving, the gradient flow dynamics will give an approximation to $u_\infty$. For differentiable energies, applying the exact projection scheme \cref{eq:dirac-frenkel} to this optimization context gives the continuous-time natural gradient method originally introduced in \cite{naturalGD}. Our contribution should be understood as a way to dynamically choose the samples to use in the course of the natural gradient steps. We summarize our main result here, and provide the proof in  \cref{lemma:best-fit-estimator-gf}.
    \begin{lemma}
      If:
      \begin{itemize}
        \item $\cF$ is proper, lower semicontinuous and $\lambda$-convex with $\lambda>0$,
        \item $u_0 \in \dom(\cF)$ and $\kappa$ is continuous on $[0,t]$,
        \item For all $0\leq s\leq t$, $\varphi(\theta(s))\in \dom(D^F \cF)$ and $\{\ell(s), \Tn(s)\}_Z$ are admissible for reconstruction,
      \end{itemize}
      then the approximation error bound \eqref{eq:error-bound} of the dynamical approximation scheme holds with $L=-\lambda <0$.
    \end{lemma}
  }

\subsection{Inexact Observations}
In our dynamical formulation \cref{eq:dirac-frenkel-dynamically-sampled}, and in our previous analysis, we have assumed that the observations $\ell(t)(f(\varphi(\theta(t))))$ can exactly be computed. As already discussed, this happens in a few relevant cases such as, for example, when $V=L^2(\Omega)$, and we use polynomial of Gaussian decoders for $\Vn$, and Gaussians for the $\omega_i(t)$. Another important case is when $V\subset \cC(\Omega)$: in this case, one can choose $\ell_i = \delta_{x_i}$, which leads to point evaluations $\ell_i(t)(f(\varphi(\theta(t))))=f(\varphi(\theta(t)))(x_i)$. In general however, the evaluation of $\ell(t)(f(\varphi(\theta(t))))$ will require computing integrals which do not come in closed form, and there will be an additional computational error in the approximation. To account for this in the error analysis, suppose that we can compute $\ell(t)(v)$ at accuracy $\eta(t)$ for any $v\in V$, in the sense that we have a control in the $p$-Euclidean norm of the form
$$
  \abs{z(t)-\ell(t)(f(\varphi(\theta(t))))}_p \leq \eta(t).
$$
The error bound from \cref{lem:error-bound-f} becomes:
\begin{theorem}[Error bound on the time derivative with inexact observations]
  \label{lem:error-bound-f-inexact}
  Let $t\geq0$. If $\{\ell(t), \Tn(t)\}_Z$ is admissible for reconstruction in $Z$, then
  \begin{align}
     & \norm{f(\varphi(\theta(t))) - \prt*{D\varphi}_{\theta(t)} \big(\dot{\theta}(t)\big) }_V \leq \kappa(t)+ 4 \nu_p\beta_Z^{-1}(\Tn(t), \ell(t)) \eta(t)
  \end{align}
  with $\nu_p \coloneqq \max_{z\in \bR^m} \frac{\abs{z}_Z}{\abs{z}_p}$.
\end{theorem}
\begin{proof}
  The result is obtained by adding perturbation arguments in the proof of \cref{lem:error-bound-f}. We omit it for the sake of brevity.
\end{proof}
With the new error bound of \cref{lem:error-bound-f-inexact}, by directly applying \cref{theorem:best-fit-estimator} we obtain an error bound where, as expected, the approximation error of the linear functionals $\ell(t)$ accumulates in time.

\section{Dynamical Sampling}
\label{sec:dyn-sampling}
Our numerical scheme \cref{eq:dirac-frenkel-dynamically-sampled} involves a dynamical approximation space $T_n(t)$ and dynamical observation functionals $\ell_i(t)$ which generate a dynamical observation space $\Wm(t)$. While $T_n(t)$ is automatically updated by time integration, we have not yet specified a concrete strategy on how to make the observations $\ell(t)=\{\ell_i(t)\}_{i=1}^m$ evolve in time.  Our strategy follows the approach of \cite{MPV2025} which is based on the fact that $\ell(t)$ interacts with $\Tn(t)$, and this interplay crucially impacts on the stability of the method through the constant $\beta_Z(\Tn(t), \ell(t))$. In turn, this affects the reconstruction accuracy as is reflected by the presence of $\beta_Z(\Tn(t), \ell(t))$ in the upper bound \cref{eq:error-bound}. We therefore evolve $\ell(t)$ in a way to minimize the value of the upper bound. In turn, this means that we need to maximize $\beta_Z(\Tn(t), \ell(t))$.

Since, by \cref{eq:optimal-norm}, the optimal $Z$-norm with respect to the error bound is $Z_\circ$, we confine the discussion to this case, and work with $\beta(\Tn(t), \Wm(t))$ in the following (see \cref{eq:stability-PBDW}). Also, to ease the presentation, we start by assuming that the linear functionals are picked from a dictionary $\cD$ of $V'$ generated by convolutions,
\begin{equation}
  \label{eq:dictionary}
  \cD = \{ \ell_x(\cdot) =  (k \circledast \cdot )(x)\cond x\in \Omega \},
\end{equation}
where $k:\Omega\to\bR$ is a kernel function, and
\begin{equation}
  \label{eq:dictionary-element}
  \ell_{x}(v) =  (k \circledast \cdot )(x) \coloneqq \int_\Omega k(x- y) v(y) \d y, \quad \forall v\in V.
\end{equation}
Note that the Gaussian observations presented earlier (\cref{eq:Gaussian observations}) are an example of such a dictionary. Denoting by $\omega_{x_i}$ the Riesz representer of $\ell_{x_i}$, an observation space generated from $m$ linear functionals from $\cD$ is of the form
$$
  \Wm(X) \coloneqq \vspan\{\omega_{x_i}\}_{i=1}^m, \quad \text{ with }X=\{x_i\}_{i=1}^m\in (\bR^d)^m.
$$

Such a space becomes time-dependent as soon as the $x_i$ follow a certain dynamics. We thus seek for an evolution of the locations $X(t)$ in a time-interval $[0,T]$ by solving
\begin{equation}\label{eq:lagrangian}
  X = \mathop{\arg \max}\limits_{\substack{Y : [0, T] \to (\bR^d)^m \\ Y(0) = X_0 \\ \dot{Y}(0)= \dot{X}_0}}\int_0^T \left( \beta^2\big(\Tn(\tau), \Wm(Y(\tau))\big) - \frac{\lambda}{2} \abs{\dot{Y}(\tau)}^2 \right) \,d\tau,
\end{equation}
where $X_0$ and $\dot{X}_0$ are suitably prescribed initial positions and velocities, and $\lambda\in\bR_+$ is a regularization coefficient that penalizes large velocities. The Euler-Lagrange equations for \cref{eq:lagrangian} read
\begin{equation}\label{eq:euler_lagrange}
  \begin{cases}
    \lambda\ddot{X}(t) = -\nabla_X \beta^2(\Tn(t),\Wm(X(t))), \\
    X(0) = X_0, \quad \dot{X}(0)=\dot{X}_0.
  \end{cases}
\end{equation}
We may note that without regularization of the velocity ($\lambda=0$), problem \cref{eq:lagrangian} boils down to finding
\begin{equation}
  \label{eq:beta-no-reg}
  X(t) = \max_{Y\in (\bR^d)^m} \beta^2(\Tn(t), \Wm(Y)), \quad \forall t\in (0, T].
\end{equation}
In practice, for every $t$, one can solve problem \cref{eq:beta-no-reg} with any optimization scheme. In our numerical experiments, we used gradient descent methods relying on a time-discrete version of the gradient flow
\begin{equation}
  \label{eq:grad-descent-continuous}
  \begin{cases}
    \frac{\d}{\d \tau} \widetilde{X}(\tau)
                     & = - \nabla_X \beta^2(\Tn(t), \Wm(\widetilde{X}(\tau))), \quad \forall \tau >0, \\
    \widetilde{X}(0) & \text{ given initial guess.}
  \end{cases}
\end{equation}
The version without regularization is the one that was used in this work. Independently of the value of $\lambda$, any numerical scheme employed for the solution of \cref{eq:euler_lagrange} requires computing the value of $\beta^2$, and also its gradient  with respect to the position of the observations. As explained in \cite[App.~2]{Mula2023}, the evaluation of $\beta^2$ is rather straightforward, and boils down to the computation of the smallest eigenvalue of a $n\times n$ matrix. For the practical computation of $\nabla_X\beta^2$, we refer to \cite[Section 3]{MPV2025}. {We note that the complexity of evaluating $\nabla_X\beta^2$ is $\cO(n^3)$ (see \cite[Section 3.1]{MPV2025}), which implies that this operation can become expensive for large $n$. In our numerical experiments, the numerical cost was not significant  since we could obtain very satisfactory accuracies with moderate values of $n$.}

\begin{remark}
  \label[remark]{rem:general-dictionaries}
  {
  It is possible to consider more general versions for the dictionary $\cD$ than the one given in \eqref{eq:dictionary}. One particularly interesting example is the one involving parametric Gaussian kernels
  \begin{equation}
    \cD = \left\lbrace \ell_{x, \Sigma}(v) =   \int_\Omega g(\Sigma, x)(y) v(y) \d y,\; \forall v\in V\;\cond (x, \Sigma)\in \Omega \times \bS_{++}^d \right\rbrace,
  \end{equation}
  where $\bS_{++}^d$ is the set of symmetric positive definite matrices.  In this case, the optimal control problem \eqref{eq:lagrangian} needs to be rephrased as the search in $(\bR^d)^m \times (\bS_{++}^d)^m$ for an evolution of the locations $X(t)=\{x_i(t)\}_{i=1}^m$, and covariances $Q(t) = \{\Sigma_i(t)\}_{i=1}^m$. Like in the main text, the Euler-Lagrange equations would provide the evolution of these quantities. The added difficulty here is that we would need to differentiate in $\bS_{++}^d$ (by viewing it as a Riemannian manifold). Such an extension was not considered in our numerical experiments, but will be investigated in future works. In our examples from \cref{sec:numerical results}, dictionaries $\cD$ with fixed covariance like in \cref{eq:dictionary} were enough to provide satisfactory accuracy and stability.}
\end{remark}

\section{Summary of the scheme}
\label{sec:summary}
Before presenting our numerical results, we summarize how the whole scheme is implemented in practice for the example discussed in \cref{sec:dyn-approx-formulation} where $V=L^2(\Omega)$ and $\omega_x = g(\sigma^2\id, x)$.
\begin{itemize}
  \item At $t=0$, we compute $\theta(0)$ and $\varphi(\theta(0))$ by solving \cref{eq:approx-t0}. We find appropriate locations $X(0)$ and define $\Wm(X(0))$.
  \item For $t>0$, we need to integrate \cref{eq:sampled-M} in time.
        Let us assume that we are at an intermediate time $t_k$, and we have computed $\theta_k$ and the locations $X_k$ so that $\varphi(\theta_k)$ and $\Wm(X_k)$ are known. If we do an explicit Euler scheme in $(t_{k}, t_{k+1}]$, we obtain $\theta_{k+1}$ by solving the linear system
        \begin{equation}
          \widehat{M}_{\sigma, m, X_k}(\theta_k) \prt*{\frac{\theta_{k+1}-\theta_k}{\d t}} = \hat q_{\sigma, m, X_k} (\theta_k).
        \end{equation}
        Crucially, since we have maximized the value of $\beta(\Tn(t_k), \Wm(X_k))$ in the previous time-step, the matrix $\widehat{M}_{\sigma, m, X_k}(\theta_k)$ is well conditioned.
        Once $\theta_{k+1}$ is obtained, we can compute the reconstruction $\varphi(\theta_{k+1})$ and the approximation space $\Tn(t_{k+1})$ at time $t_{k+1}$. Finally, given $\Tn(t_{k+1})$, we compute $X_{k+1}$ by solving \cref{eq:beta-no-reg} for $t=t_{k+1}$. In turn, this yields $\Wm(X_{k+1})$. Alternatively to \cref{eq:beta-no-reg}, if we work with regularization we can compute $X_{k+1}$ by integrating \cref{eq:euler_lagrange} with an explicit Euler scheme in $[t_k, t_{k+1}]$.
\end{itemize}
Alternatively to explicit Euler schemes, one can of course very easily consider any higher-order explicit time-integration methods. Working with an implicit time scheme turns the task much more challenging, because one is led to a system of coupled equations for the unknowns $\theta_{k+1}$ and $X_{k+1}$. One would then need to work with fixed-point iterative strategies. We have left this aspect for future works, and have worked with explicit Runge-Kutta schemes in our numerical tests. Also, for an analysis of the behavior of time integration schemes in regularized nonlinear dynamical approximation, we refer to \cite{feischl2024regularized, LN2025}. The results presented in these works could easily be imported to the current setting to take into account the additional error coming from time discretization.

\begin{remark}
  \label{rem:dynamical-m}
  In the present setting, the value $\beta(\Tn(t), \Wm(X(t)))$ is maximized pointwise in time, but we do not guarantee that it stays above a given threshold $\underline{\beta}\in(0,1]$. We will see in our numerical results that it may happen that $\beta(\Tn(t), \Wm(X(t)))$ decreases in time, especially when $m$ is chosen to be close to the value $n_\eff(t)$ (see, e.g.,~\cref{fig:AC beta looped}). To ensure that $\beta(\Tn(t), \Wm(X(t)))\geq\underline{\beta}$ for all $t>0$, one could add a further refinement in our strategy, and increase the amount of observations $m$ dynamically in time: for a given time $t$, if $\beta(\Tn(t), W_{m(t)}(X(t)))<\underline{\beta}$, we could start adding observations until we reach the desired threshold. One way of doing this could be with a greedy algorithm that would find the optimal locations for the new observations to add. We refer to \cite{BCMN2018} for details on the strategy, and an analysis on the behavior of $\beta(\Tn(t), W_{m}(X(t)))$ as a function of $m$. Incorporating this aspect in the current method would require adapting the optimal control formulation \cref{eq:lagrangian} for the dynamical sampling strategy, and will be investigated in future works.
\end{remark}

\section{Numerical Results}
\label{sec:numerical results}
In this section, we show the behavior of the scheme for a broad range of PDEs. Our implementation\footnote{\url{https://gitlab.tue.nl/data-driven/stable_dyn_approx}} closely follows the form given in \cref{sec:summary}. It is based on computing $\dot{v}_{Z_\circ}^*(t)$ from \cref{eq:dirac-frenkel-dynamically-sampled}  with dynamical sampling as described in \cref{sec:dyn-sampling}. The ambient Hilbert space is taken as $V=L^2(\Omega)$ for all examples, and we work with observation spaces $\Wm(t)=\vspan\{\omega_i(t)\}_{i=1}^m$ spanned by Gaussians $\omega_i(t)=g(\Sigma, x_i(t))$ with a fixed covariance matrix {(see \cref{rem:general-dictionaries} for more general choices)}. We use an explicit RK4 time-integrator. As guiding examples, we discuss the Korteweg-De Vries and Allen-Cahn equation in $d=1$, a Fokker-Planck equation in $d=2$ and $d=6$, and a pure transport equation in $d=10$. An additional test case on the viscid Burger's equation can be found in \cref{app:additional numerical results}, and reveal similar features as the ones outlined in this section. For each of the experiments below except for the Allen-Cahn equation, we use the exponential decoder, denoted by  $\varphi = \varphi^\textnormal{exp}$, which is a shallow neural network defined by \cref{eq:SNN decoder} with activation function $\sigma \coloneq \exp$ defined in \cref{eq:parametrized exp}. The parameters are $\theta = (\theta_i)_{i=1}^p$ with $\theta_i=(c_i, S_i,\mu_i)\in \bR \times \bR^{d\times d} \times \bR^d$.

\subsection{Korteweg-de Vries Equation (KdV 1D)}
As a first test for our scheme, we consider the one-dimensional Korteweg-De Vries (KdV) equation
\begin{align}
  \dot{u} & = -6u\partial_x u - \partial_x^3 u, \label{eq:kdv}                                                              \\
  u(0, x) & = 2 \partial_x^2\ln\left(1 + e^x + e^{\sqrt{5}x + 10.73} + Ce^{(1+\sqrt{5})x + 10.73}\right), \label{eq:KdV IC}
\end{align}
with $C = \left(\frac{1-\sqrt{5}}{1+\sqrt{5}}\right)^2$, and for all $(t, x) \in [0, 4]\times \bR$. This PDE has an explicit solution (see \cite{Taha1984}) which we use to evaluate the approximation error of our scheme. As \cref{fig:KdV results} illustrates, the solution is composed of two solitons travelling at different speed, and one overtakes the other. The solution is thus compactly supported, and intuitively one expects that the location $x_i(t)$ of our Gaussians $\omega_i(t)=g(\sigma, x_i(t))$ follows the support to retrieve as much information as possible, and preserve stability.

We take {for $\varphi$ a SNN \cref{eq:SNN decoder} with activation function \cref{eq:exponential_activation_function} and $p=10$}. {We have $\S_i \equiv s_i \in \bR_+$} since $d=1$. Therefore, the parameters are $ {\theta_i = (c_i, s_i, \mu_i)} \in \bR\times \bR_+ \times\bR$ for $i=1,\dots,10$. This results in a dynamical approximation space $\Vn$ with $n=30$. In the course of the evolution, the effective dimension $n_\eff(t)$ of the tangent space $\Tn(t)$ changes in such a way that $n_{\eff}(t)\leq n=30$ for all $t\geq0$ (see \cref{eq:effective-n}), but $n_\eff(t)$ is unknown a priori. Thus, by \cref{prop:n-m-condition}, we need $m\geq n=30$ observations for guaranteed stability at all times. For the observations $\omega_i$, we take variance $\sigma = 0.1$. We choose time step $\dt = 10^{-3}$ and regularization parameter $\varepsilon = 10^{-2}$.

\Cref{fig:KdV results} shows the shape of the approximated solution for $m=40$, and it also shows the locations $x_i(t)$ of the mean of the observation Gaussians. We see that the numerical solution captures the peaks of the exact solution very accurately. Also, as expected, the $x_i(t)$ follow the support of the exact solution, and they are spread around it.

\Cref{fig:KdV error looped,,fig:KdV beta looped} show the behavior of the method with respect to the $L^2(\Omega)$ error $e(t)$ and the stability constant $\beta(t)$ when we work with $m\in \{25, 30, 35, 40\}$ observations. Lowering the number of observations to as little as 25 is motivated by an a posteriori analysis of $n_\eff(t)$ for the initial test case with $m=40$: \cref{fig:KdV effective dimension} shows that, in this setting, $n_\eff(t) \leq 24$. We observe that the error $e(t)$ increases in time for all values of $m$ (as predicted by \cref{theorem:best-fit-estimator}). However, the value of $e(t)$ strongly depends on $m$: as $m$ decreases and gets closer to $n_\eff(t)$, the value of $e(t)$ increases. This effect is particularly striking for $m=25$. \Cref{fig:KdV beta looped} reveals that, in this case, $\beta(t)$ is very close to 0, which results in very poor stability, and explains the large error values. For $m\in\{30, 35, 40\}$, $\beta(t)=\ord(10^{-1})$ so we are in a very stable regime. We may observe that stability starts to degrade at later times for $m=30$. This could be mitigated by dynamically increasing the amount of observations.

\begin{figure}[H]
  \centering
  \subfloat[$t=0$]{
    \includegraphics[width=0.3\linewidth]{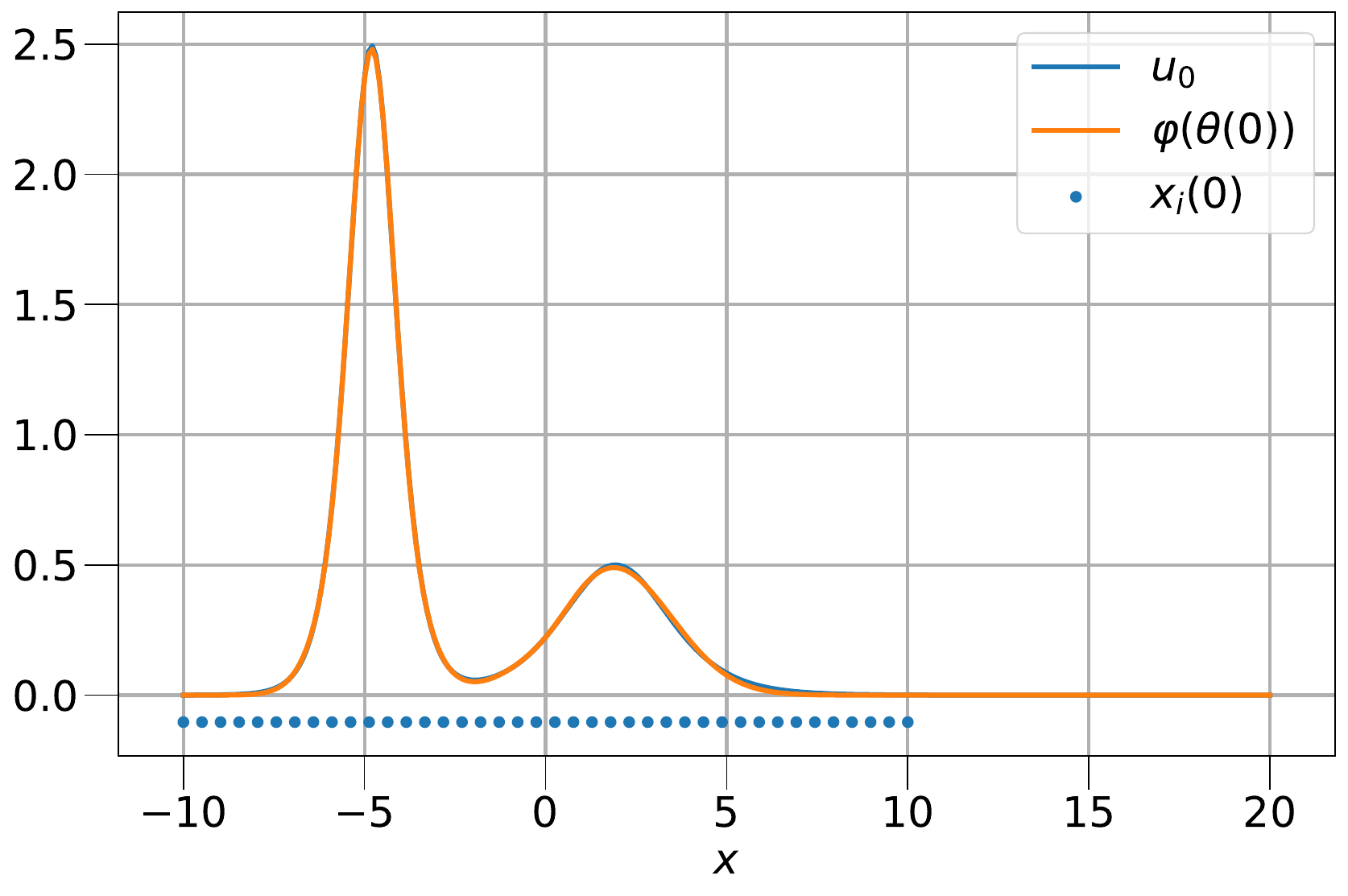}
    \label{fig:KdV initial fit}
  }
  \subfloat[$t=2$]{
    \includegraphics[width=0.3\linewidth]{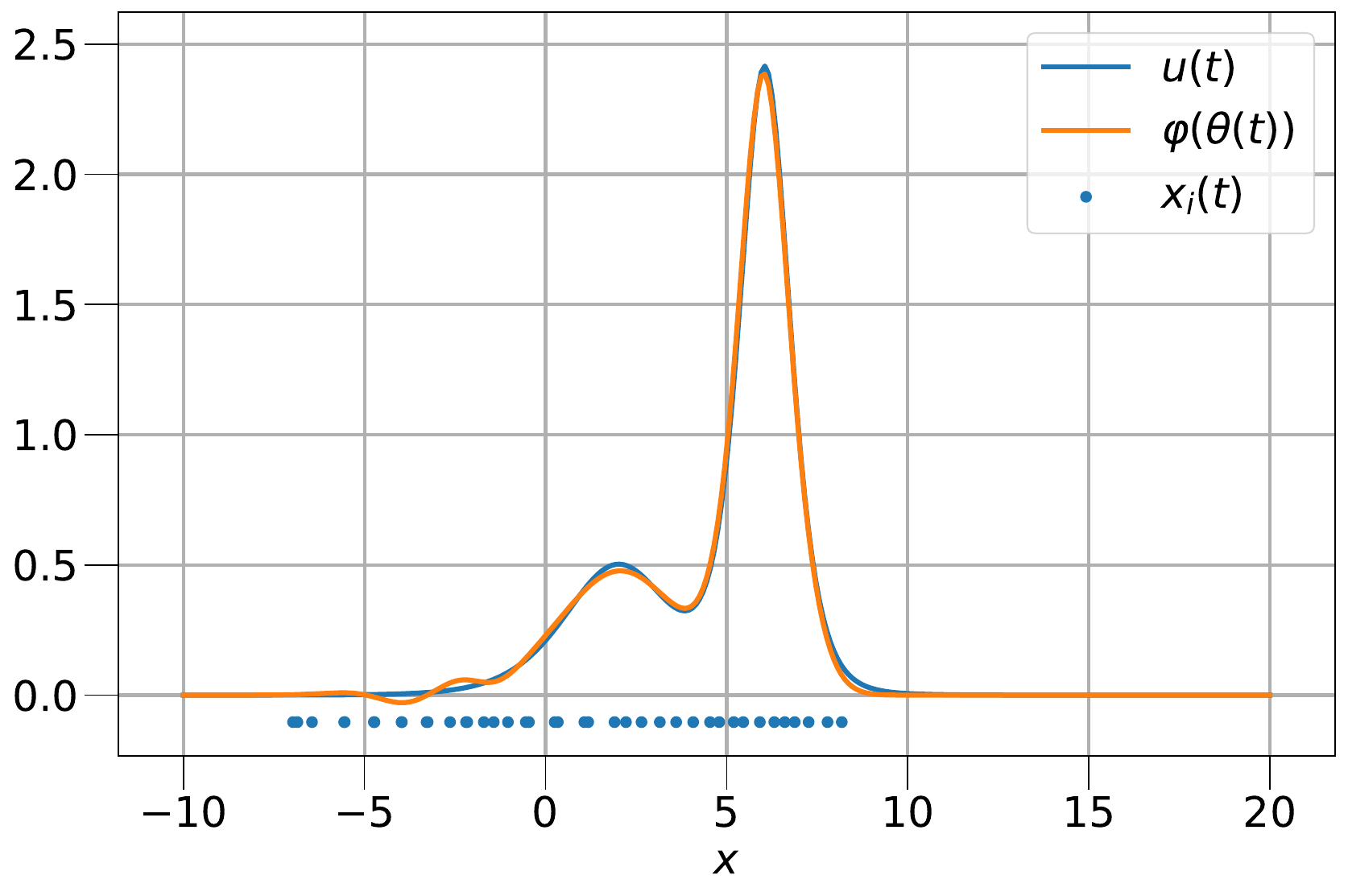}
    \label{fig:KdV snapshot t=2}
  }
  \hspace{0mm}
  \subfloat[$t=T=4$]{
    \includegraphics[width=0.3\linewidth]{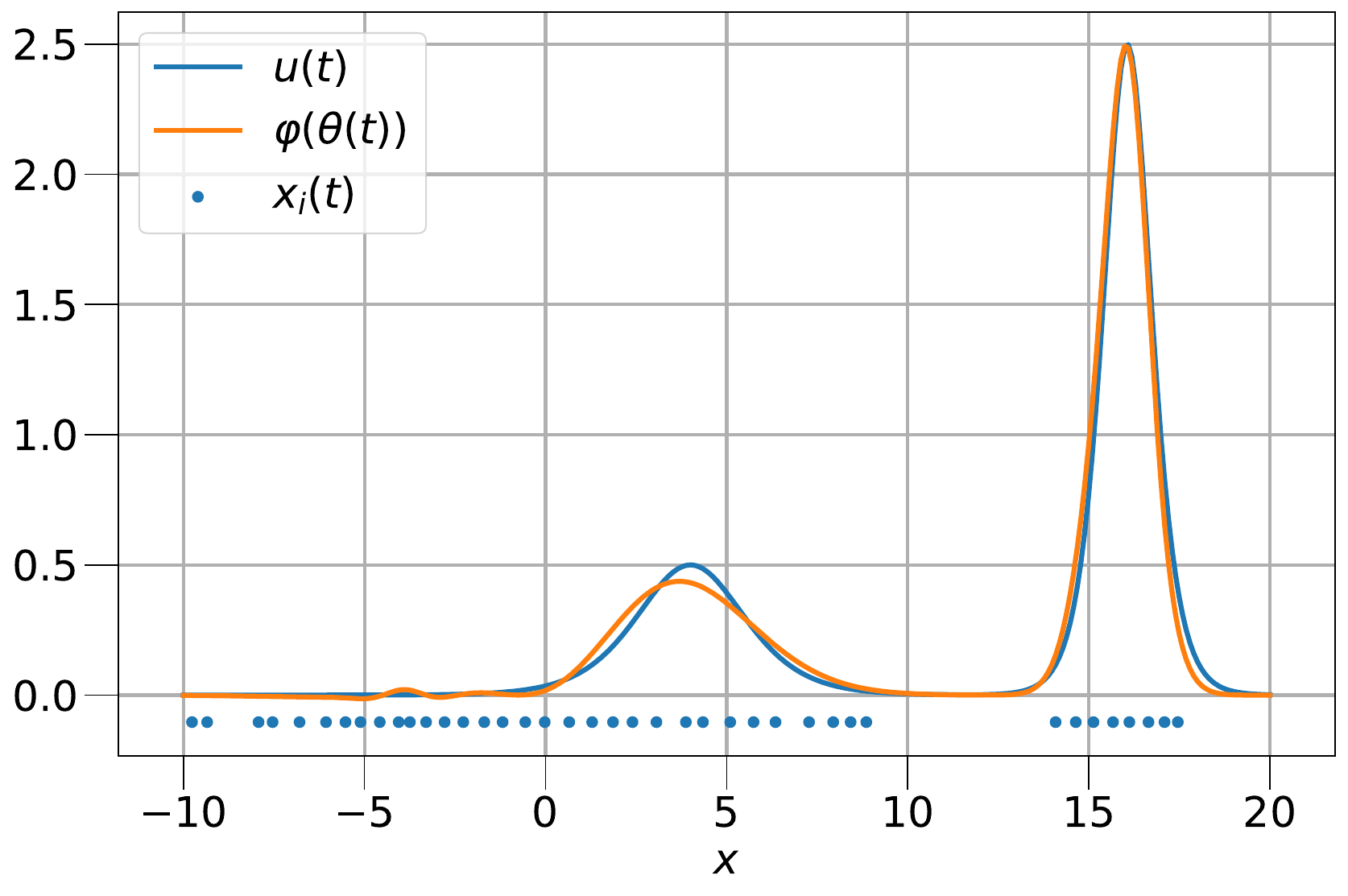}
    \label{fig:KdV snapshot t=4}
  }
  \caption{KdV 1D: Exact solution and its approximation ($n=30,\, m=40,\, \sigma=0.1$).}
  \label{fig:KdV results}
\end{figure}

\begin{figure}[H]
  \centering
  \subfloat[$e(t)$]{
    \includegraphics[width=0.3\linewidth]{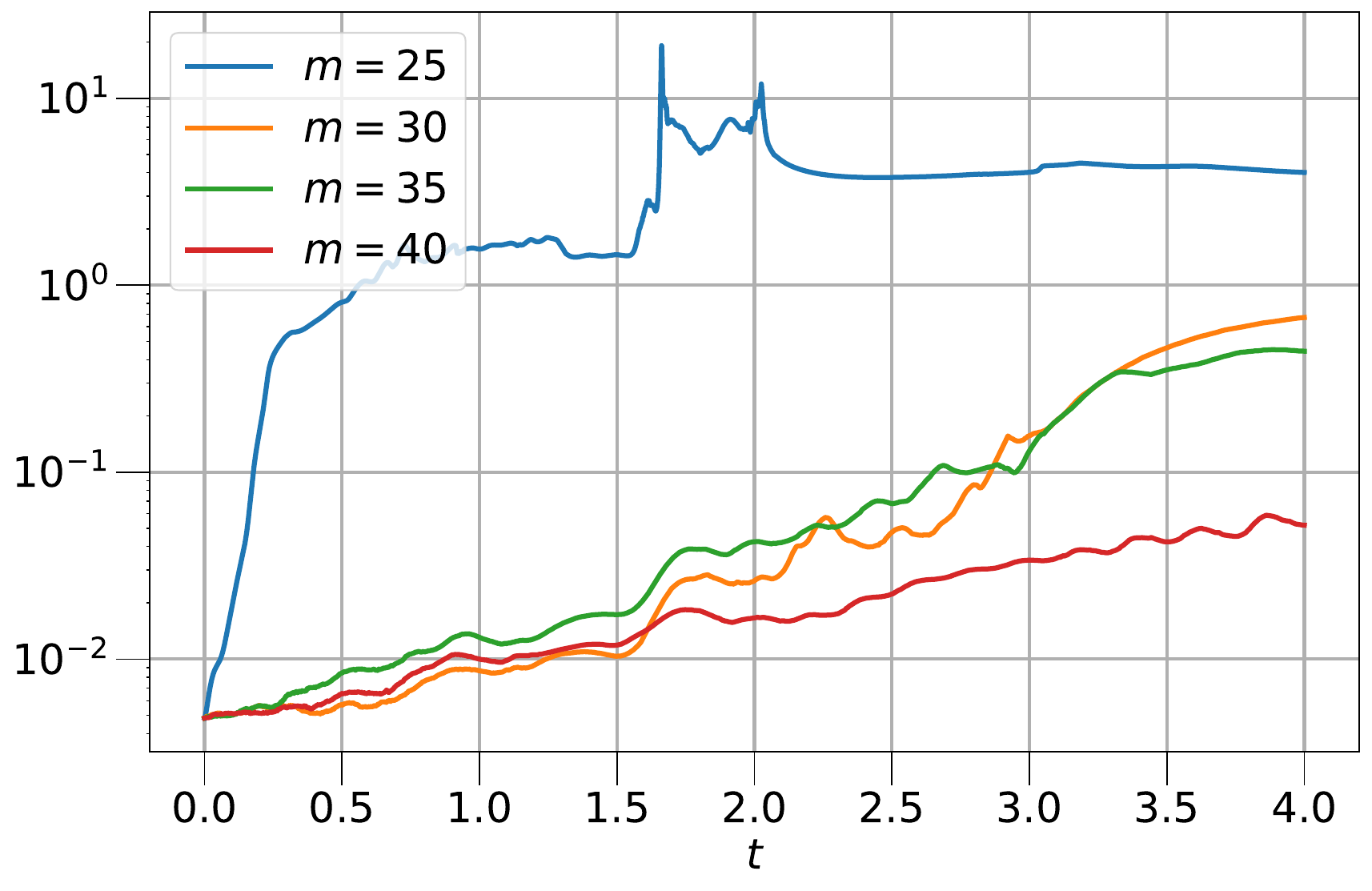}
    \label{fig:KdV error looped}
  }
  \subfloat[$\beta(t)$]{
    \includegraphics[width=0.3\linewidth]{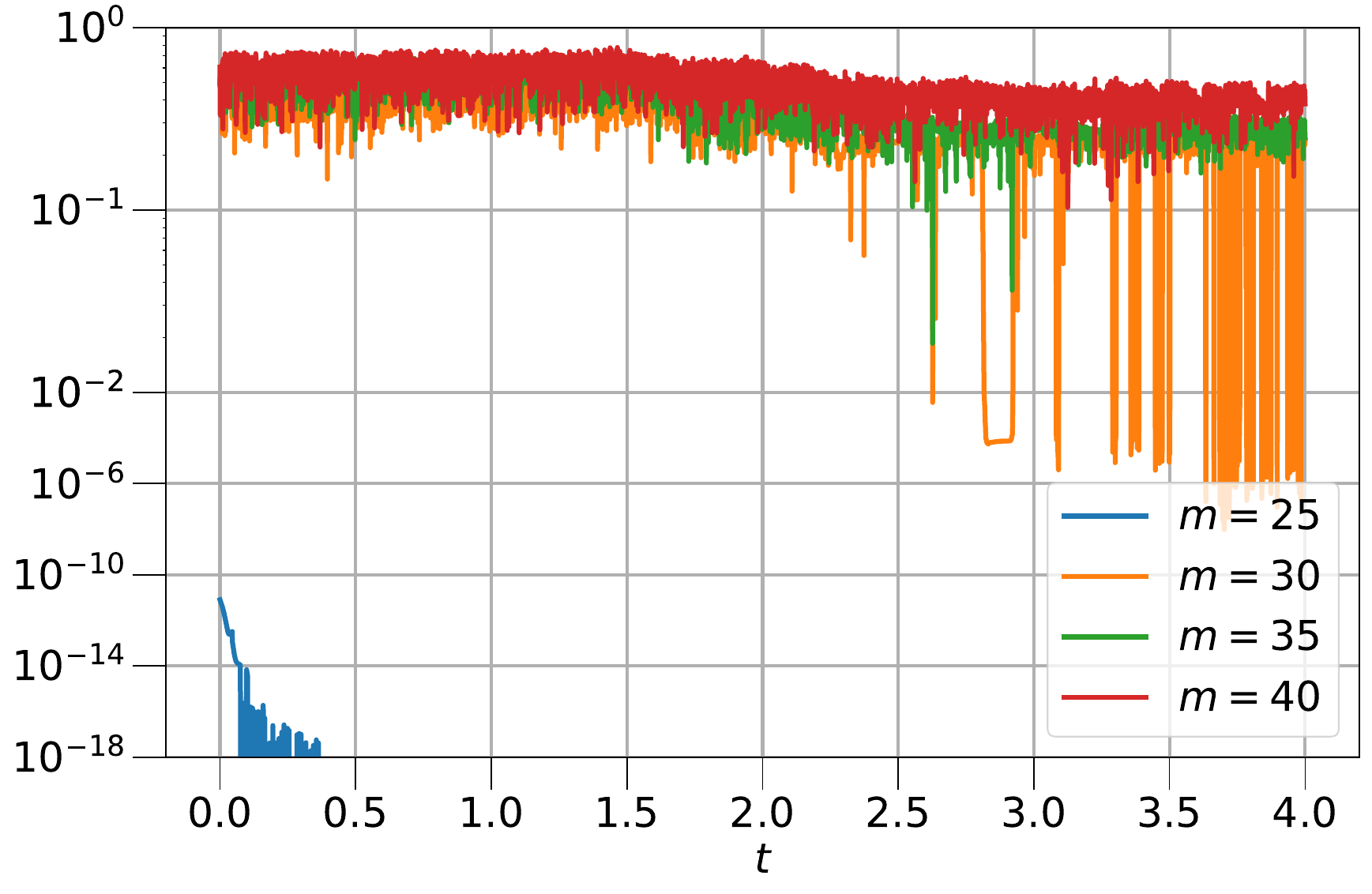}
    \label{fig:KdV beta looped}
  }
  \hspace{0mm}
  \subfloat[$n_\eff(t)$ ($m=40$)]{
    \includegraphics[width=0.3\linewidth]{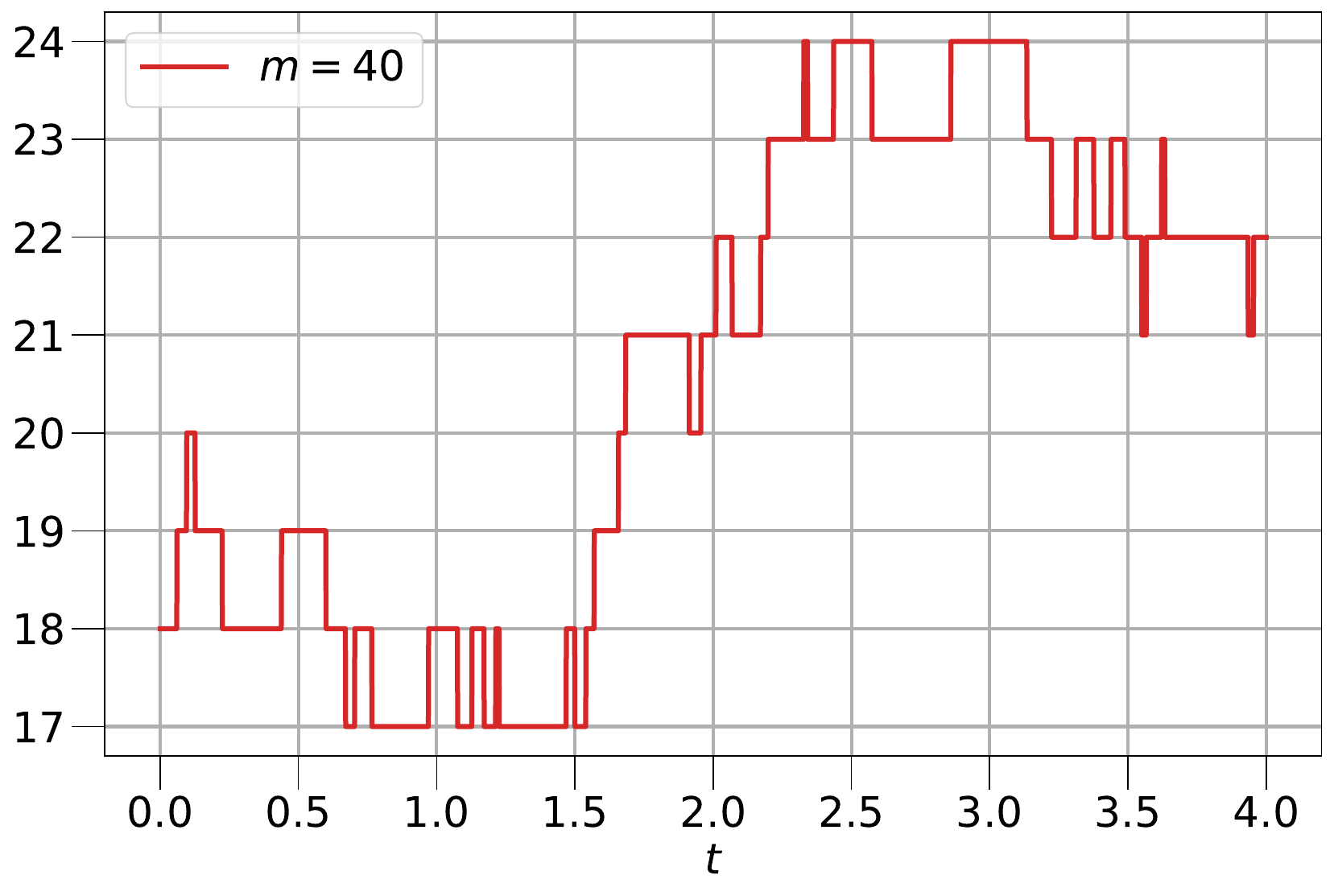}
    \label{fig:KdV effective dimension}
  }
  \caption{KdV 1D: Behaviour of the solution under a varying number of measurements $m$ ($\sigma=0.1$).}
  \label{fig:KdV looped}
\end{figure}

\subsection{Allen-Cahn Equation (AC 1D)}
We next consider the Allen-Cahn equation from \cref{eq:allen-cahn-sf} with $\Omega=(0, 2\pi)$, $T=4$, and parameters $(a,b)=(5\cdot10^{-3},-1)$, and with initial condition
\begin{equation*}
  u_0(x) = \exp\left(-\left[\sin\left(\frac{x-1.5}{2}\right)\right]^2\right) - \exp\left(-\left[\sin\left(\frac{x-2\pi+1.5}{2}\right)\right]^2\right), \qquad \forall x \in \Omega.
\end{equation*}
As explained in \cref{sec:example-PDEs}, this is a gradient flow in $L^2$, for which we have an extended theoretical framework.

Following \cite{Bruna2024}, as a reference solution, we use a numerical solution computed with an implicit-explicit scheme, in which we take $2\cdot10^3$ equidistant grid points and $\d t=10^{-5}$. For the decoder, we take a two-layer feed forward neural network, where each layer is 5 neurons wide, and every neuron uses the $\tanh$ activation function. This yields an approximation space $\Vn$ with $n=45$. For our scheme, we use a time step $\d t = 10^{-3}$ and regularization parameter $\eps = 10^{-4}$. For the observations, we set $\sigma=0.1$.

The results are structured like in the KdV example. \Cref{fig:AC results} shows the shape of the approximated solution for $m=45$, and it also shows the locations $x_i(t)$ of the mean of the observation Gaussians. The numerical solution is visually perfect except at the boundaries. We see that at $t=T=4$, the measurements are slightly more concentrated around the steepest parts of the solution, which is indicative of our method's tendency to focus on the parts of the solution which need particular focus.

\Cref{fig:AC error looped,,fig:AC beta looped} show the behavior of the method with respect to the $L^2(\Omega)$ error $e(t)$, the stability constant $\beta(t)$, and $n_\eff(t)$ when we fix $\sigma=0.1$, and we work with an increasing amount of observations $m\in \{18, 20, 30, 40, 45\}$. Again, the lower value $m=18$ is based on observing $n_\eff(t)$ for the initial experiment with $m=45$, as shown in \cref{fig:AC effective dimension}. Like before, $e(t)$ increases in time, and its value strongly depends on how close $m$ gets to $n$. When $m$ is taken large enough ($m\geq 30$), we observe that the stability factor is very close to 1 at all times. This means that the method is extremely stable, and delivers a near-optimal approximation at all times. We conjecture that the fact that stability does not degrade over time is connected to the gradient flow nature of the dynamics, and the convergence of the flow to a stationary solution. When $m$ gets close to $n_\eff(t)$, we observe that $\beta(\Tn(t), \Wm(X(t)))$ degrades as time increases. As explained in \cref{rem:dynamical-m}, this is in connection to the fact that we are not explicitly enforcing that $\beta(\Tn(t), \Wm(X(t)))$ stays above a given threshold $\underline{\beta}$, and a fix for this would be to consider an dynamical adaptive strategy where $m$ would depend on time. This point will be investigated in future works.

In \cref{fig:AC sigma looped} we illustrate the impact of the value of $\sigma$ by fixing $m=40$, and varying  \\$\sigma\in \{0.001, 0.01, 0.1, 0.5\}$. The main observation is that there is a ``sweet spot'' for the value of $\sigma$ for which $\beta(t)$ behaves best. In our case, this happens around $\sigma=10^{-1}$. The underlying intuition is that if $\sigma$ is too small, then the information which is retrieved is too concentrated, and this causes instabilities in $\beta$. When $\sigma$ is too large, Gaussian observations that are close to each other retrieve almost the same type of information. The associated $\omega_{x_i}$ become linearly dependent, and lead to instabilities as well.

  \begin{figure}[H]
    \centering
    \subfloat[$t=0$]{
      \includegraphics[width=0.3\linewidth]{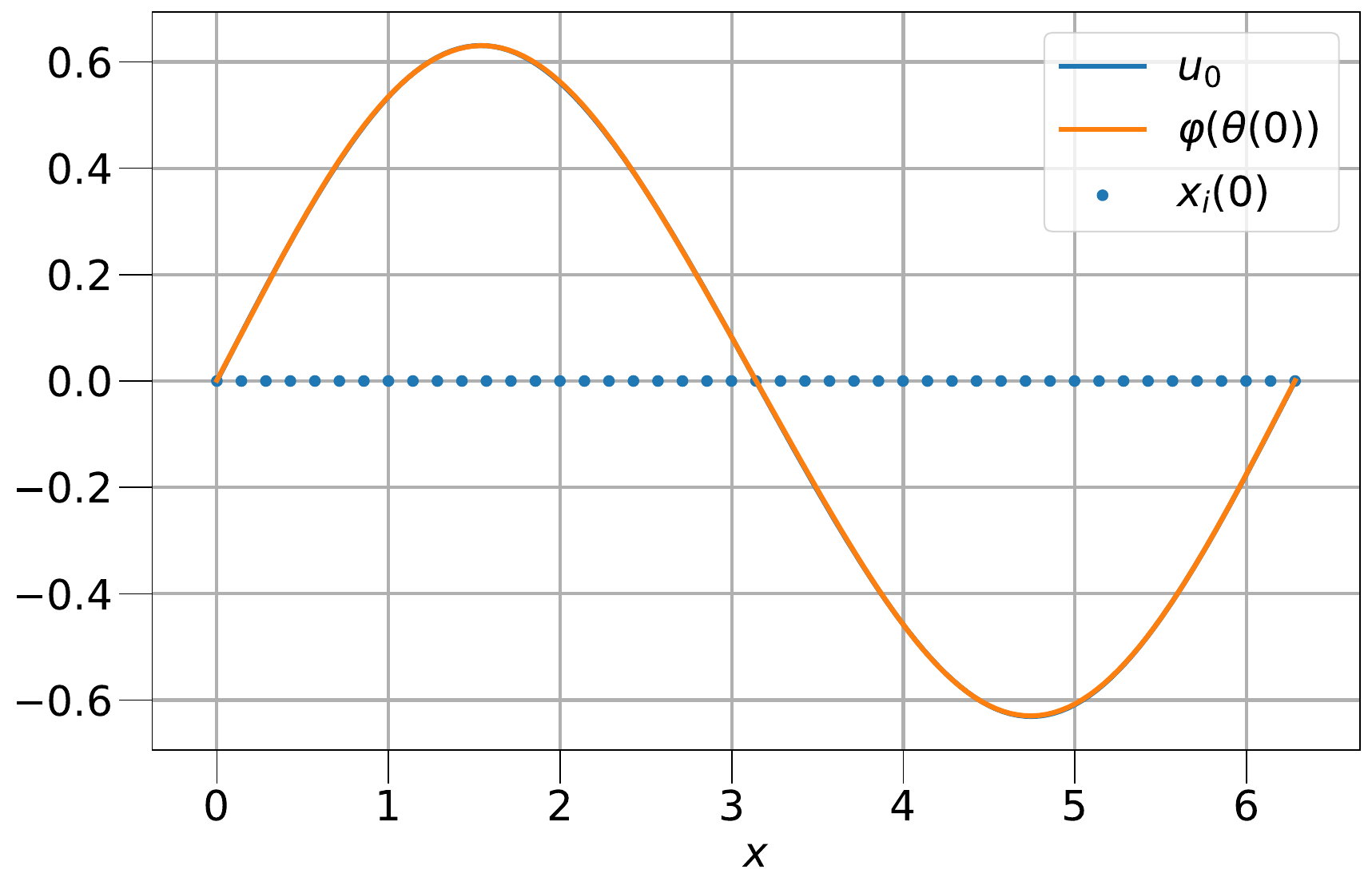}
    }
    \subfloat[$t=1$]{
      \includegraphics[width=0.3\linewidth]{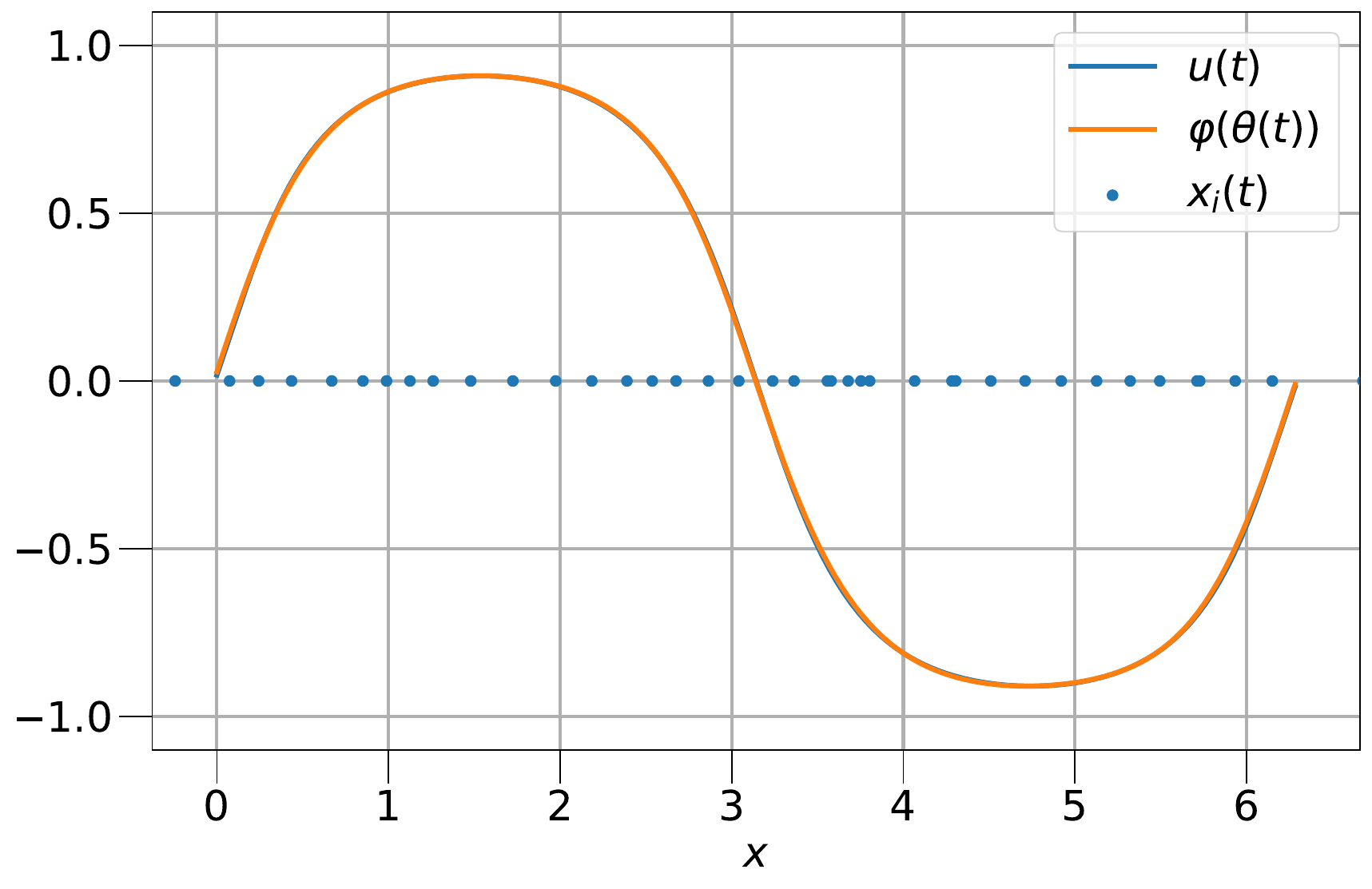}
      \label{fig:AC snapshot t=1}
    }
    \hspace{0mm}
    \subfloat[$t=T=4$]{
      \includegraphics[width=0.3\linewidth]{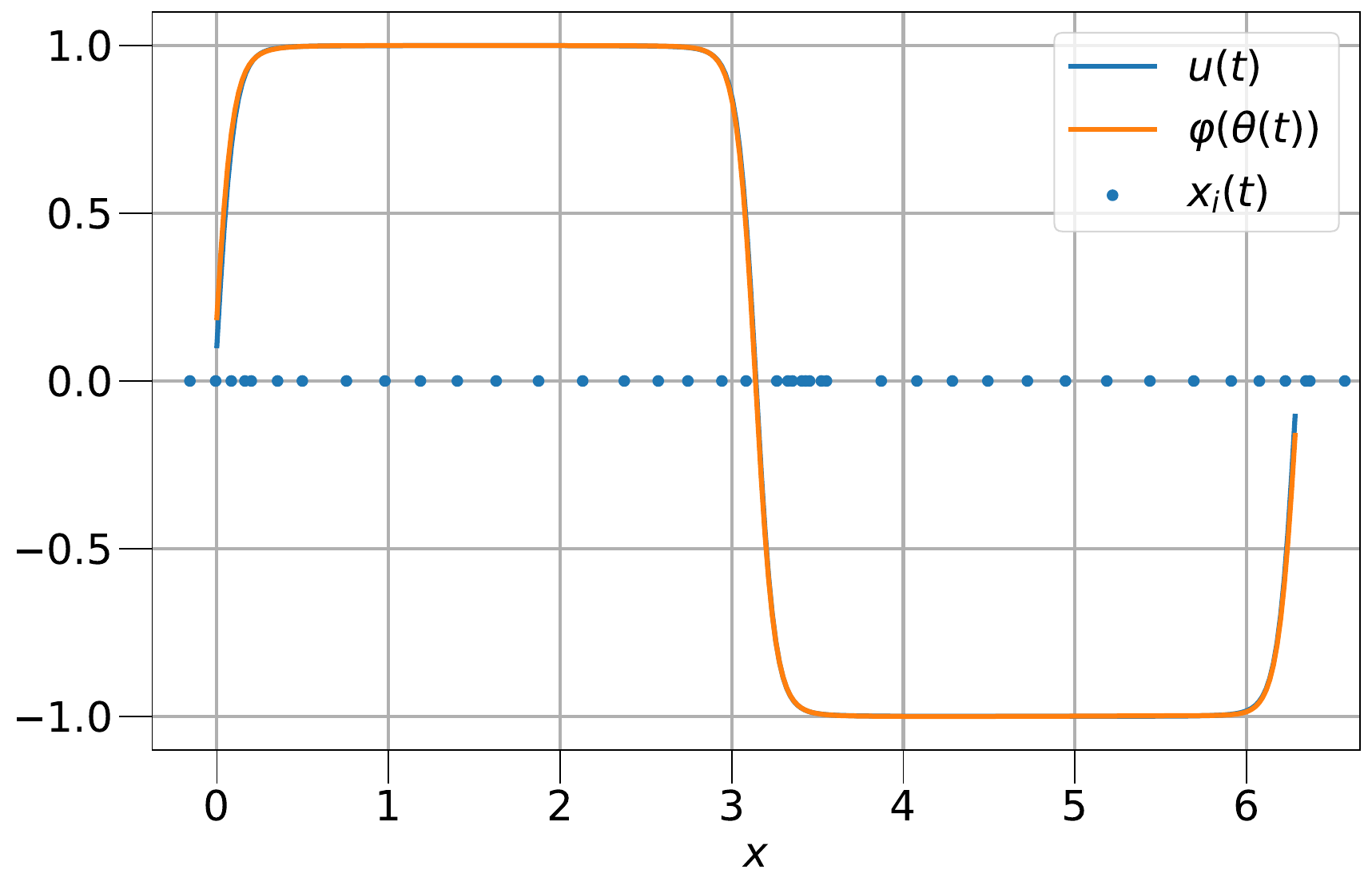}
      \label{fig:AC snapshot t=4}
    }
    \caption{AC 1D: Exact solution and its approximation ($n=m=45$,\, $\sigma=0.1$).}
    \label{fig:AC results}
  \end{figure}

  \begin{figure}[H]
    \centering
    \subfloat[$e(t)$]{
      \includegraphics[width=0.3\linewidth]{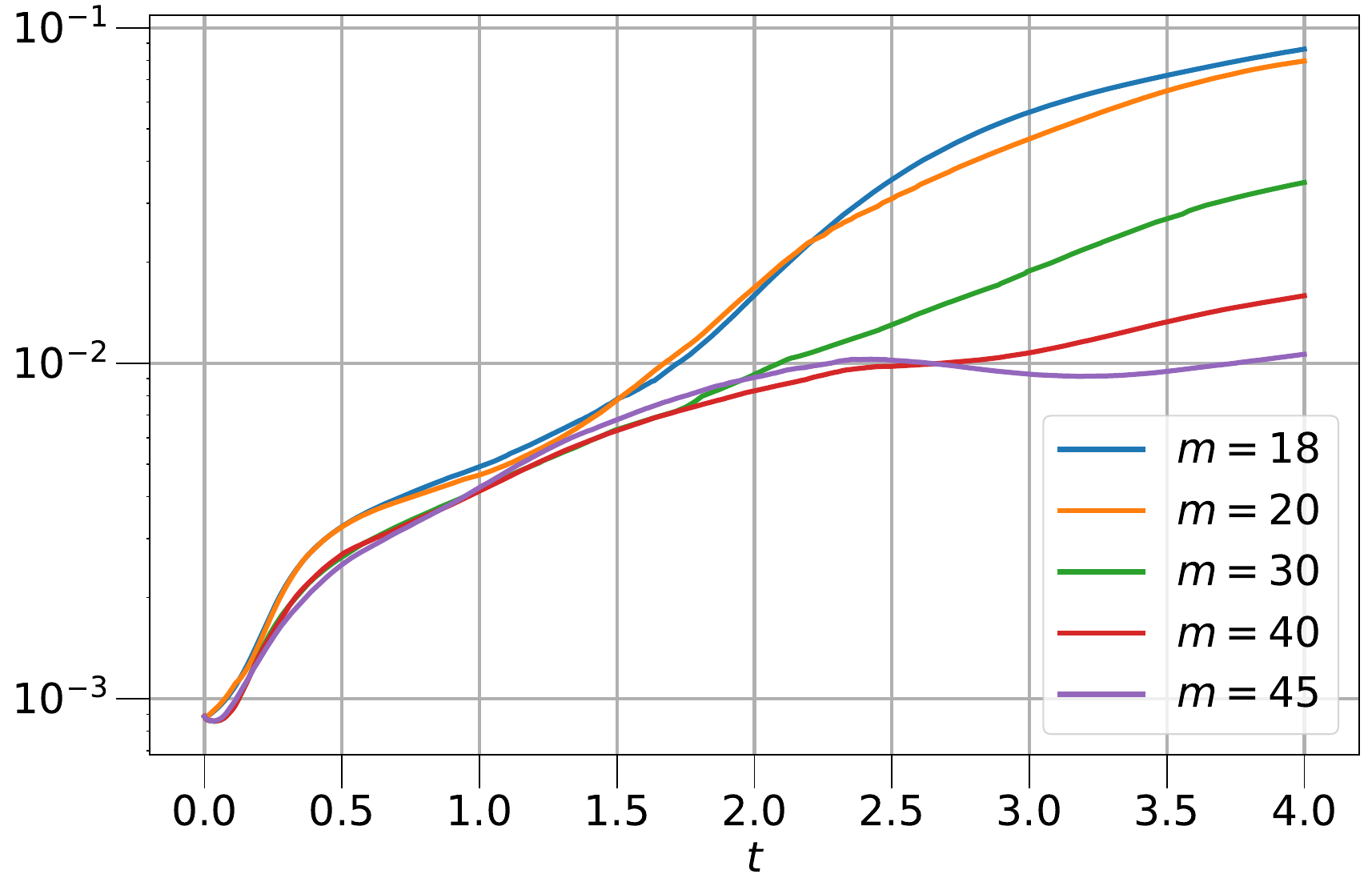}
      \label{fig:AC error looped}
    }
    \subfloat[$\beta(t)$]{
      \includegraphics[width=0.3\linewidth]{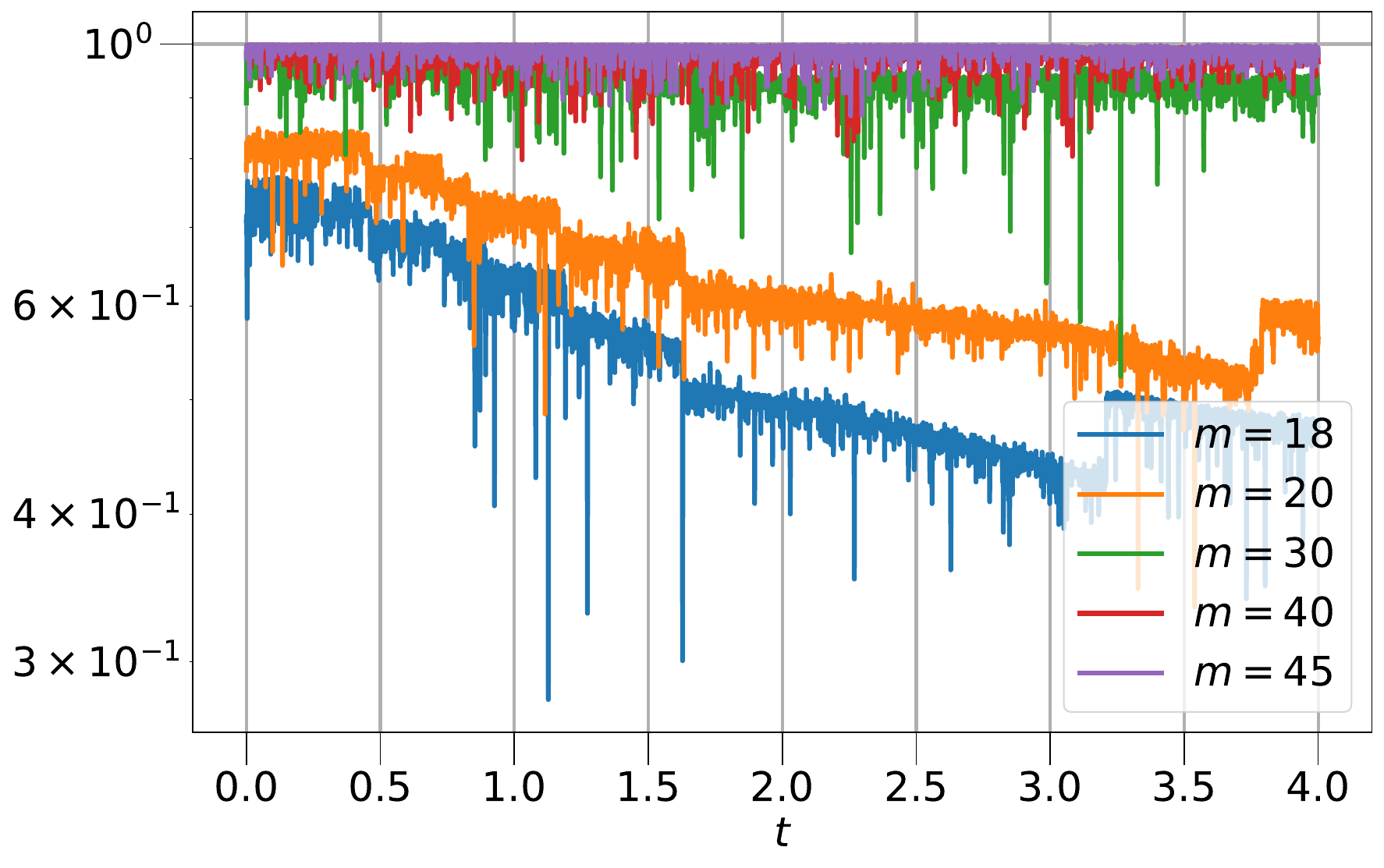}
      \label{fig:AC beta looped}
    }
    \hspace{0mm}
    \subfloat[$n_\eff(t)$ ($m=45$)]{
      \includegraphics[width=0.3\linewidth]{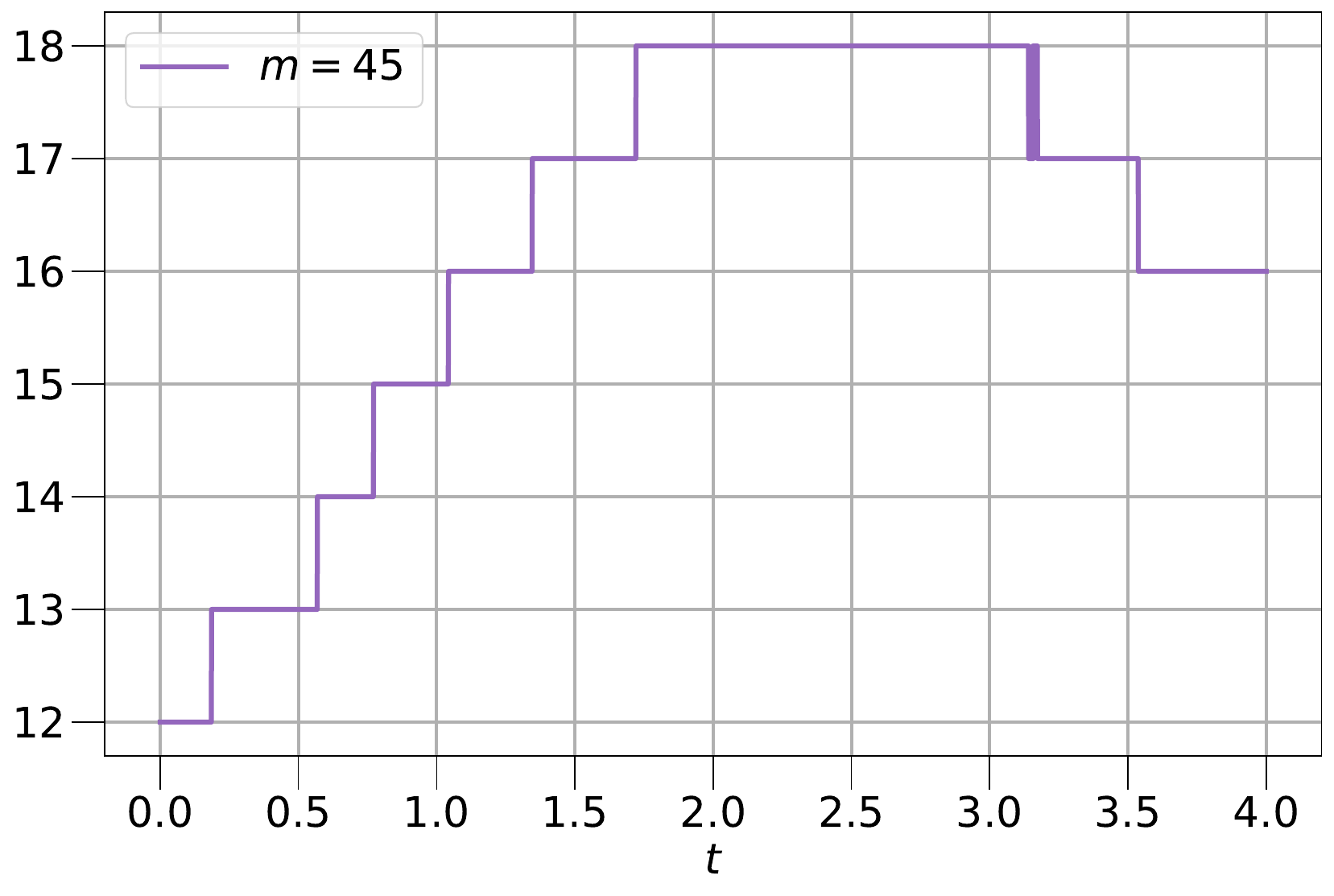}
      \label{fig:AC effective dimension}
    }
    \caption{AC 1D: Behaviour of the solution under a varying number of measurements $m$ ($\sigma=0.1$).}
    \label{fig:AC looped}
  \end{figure}

  \begin{figure}[H]
    \centering
    \subfloat[$e(t)$]{
      \includegraphics[width=0.3\linewidth]{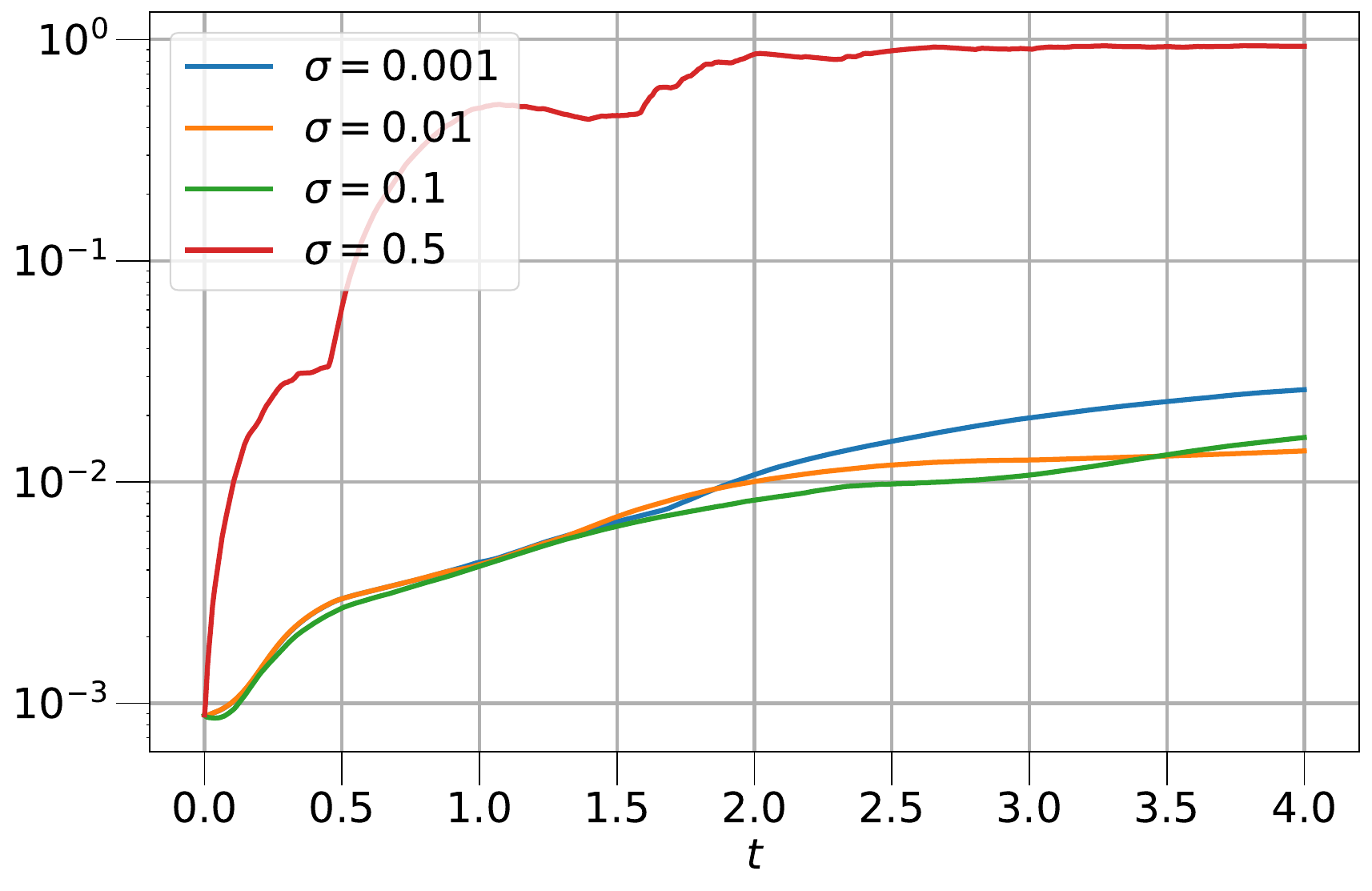}
      \label{fig:AC error sigma looped}
    }
    \subfloat[$\beta(t)$]{
      \includegraphics[width=0.3\linewidth]{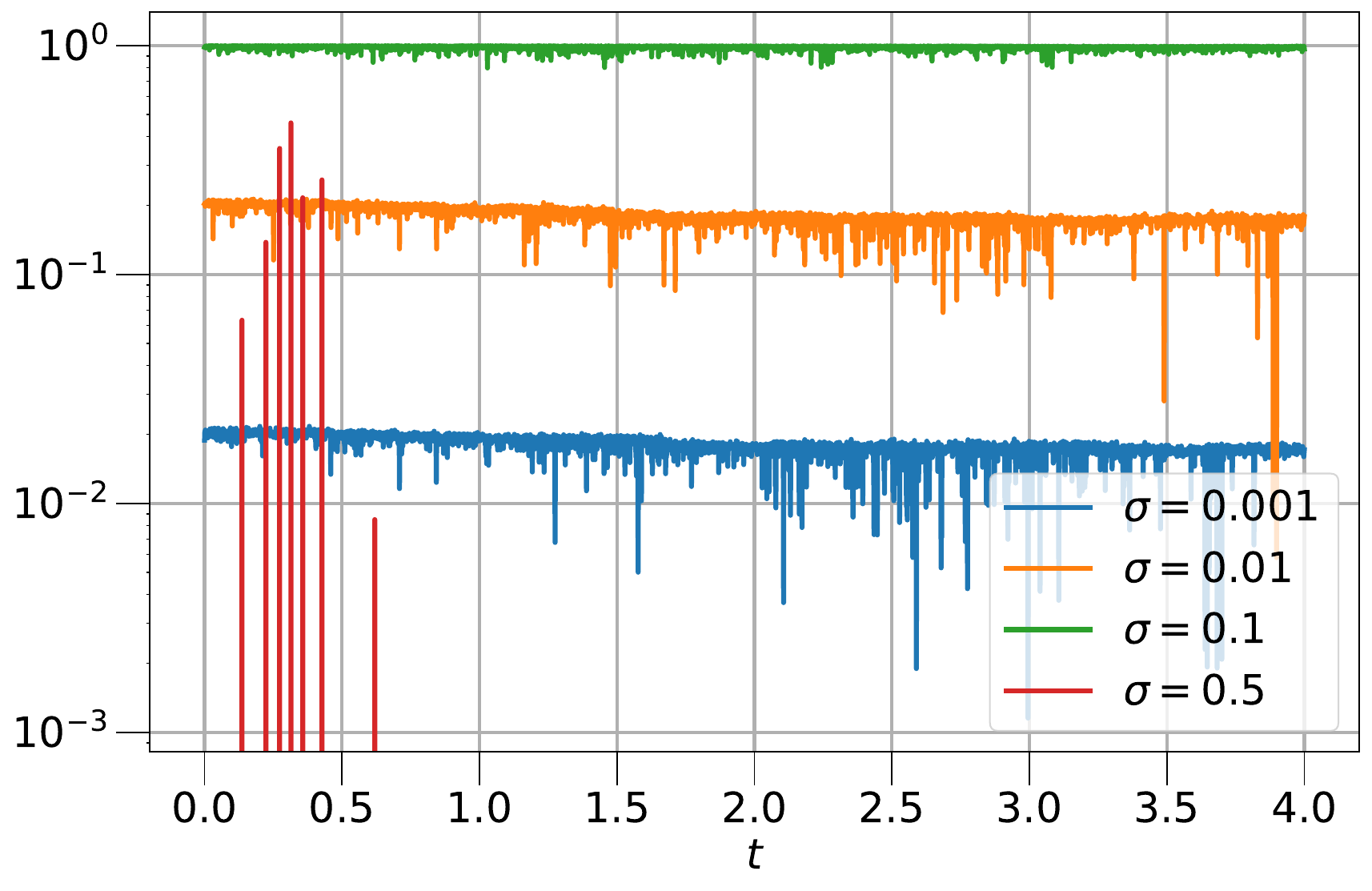}
      \label{fig:AC beta sigma looped}
    }
    \caption{AC 1D: Behaviour of the solution under a varying $\sigma$ ($m=40$).}
    \label{fig:AC sigma looped}
  \end{figure}

  \subsection{Fokker-Planck Equation (FP)}
  \label{sec:FP_results}
  Next, we consider the Fokker-Planck (FP) equation defined for all $(t,x)\in \bR_+\times \bR^d$ as
  \begin{equation*} 
    \label{eq:fp}
    \dot{u}(t,x ) = - \nabla \cdot \left( w(x) u(t, x) \right) + a \Delta u(t,x),\,\,  u(0,x) = \exp(-\abs{x}^2/2),
  \end{equation*}
  where $a > 0$ is the diffusion coefficient, and $w(x) = Ax + b$ is a spatially dependent vector field with $A\in \bR^{d\times d}$ and $b\in \bR^d$.
  The exact solution of the equation is
  \begin{equation}
    \label{eq:FP true solution}
    {u(t, x) = \frac{1}{\abs*{\det\left( S(t) \right)}^{\frac 1 2 }} \exp\prt*{-\frac 1 2 (x-\mu(t))^\top S(t) (x-\mu(t))}}
  \end{equation}
  with the parametrized exponential defined in \cref{eq:parametrized exp}, and
  \begin{equation*}
    \mu(t) = \exp(tA) A^{-1} \big(\id - \exp(-tA) \big) b, \;\;\;\; S(t) = \exp\left( t A\right) \Big( \id + 2 a \int_0^t \exp\left(-s( A + A^T)\right) \d s \Big) \exp\left( t A^T \right).
  \end{equation*}
  We consider $d=2,6$ and choose for the vector field respectively
  \begin{equation*}
    A = \begin{pmatrix}
      0  & 1 \\
      -5 & 0
    \end{pmatrix}
    , \quad
    b = \begin{pmatrix}
      3 \\
      3
    \end{pmatrix}
    \text{   and   }
    A =
    \begin{pmatrix}
      0  & 0 & 0  & 1  & -1 & - 1 \\
      0  & 0 & -1 & -1 & -1 & -1  \\
      0  & 1 & 0  & 0  & 0  & 3   \\
      -1 & 1 & 0  & 0  & -1 & 0   \\
      1  & 1 & 0  & 1  & 0  & 0   \\
      1  & 1 & -3 & 0  & 0  & 0
    \end{pmatrix},
    \;\;
    b = \begin{pmatrix}
      0 \\ 0.4 \\ 0.8 \\ 1.2 \\ 1.6 \\ 2
    \end{pmatrix}.
  \end{equation*} 
  We first discuss $d=2$, where we make use of {a SNN \cref{eq:SNN decoder} with activation function \cref{eq:exponential_activation_function}}. Since the solution is of exponential form, one can in principle capture it exactly with only $p=1$ unit in the decoder. We investigate the effect of redundancy by considering $p\in \{1, 2, 5, 20\}$ in our numerical results, so the dimension of the approximation space $\Vn$ is $n\in\{6, 12, 30, 140\}$. In all cases, we initialize the weights $\theta(0)$ in such a way that $\varphi(\theta(0))=u_0$. We do this by setting $\mu_i(0)=0$, and $c_i(0) = \frac{1}{p}$ for all $i$ (observe that the choice in the $c_i(0)$ is not unique as soon as $p>1$). We use $m=40$ Gaussian observations with $\Sigma = 0.3\id$ for all our choices of $p$. As regularization for $M(\theta(t))$, we take $\eps = 10^{-5}$, and our time step is $\d t = 5\cdot 10^{-4}$.

  The results can be found in \cref{fig:FP 2d performance,fig:FP 2d theta,fig:FP solution p1}. As expected, we can see very good performance of the method, as we mentioned already that the decoder should be able to exactly reconstruct the solution. This is reflected in a low error $e(t)$, as can be seen in \cref{fig:FP 2d error}. We can see in \cref{fig:FP 2d beta} that for $p=2,5,20$, the stability constant $\beta(t)$ drops throughout the computation. In \cref{fig:FP 2d parameters 2,fig:FP 2d parameters 5} we can see a reason for this, namely that the initially identical exponential units ``split''. This corresponds to an increase in the error, and is related to the difficulties sketched in \cref{sec:L2-setting}, {in particular the second point there about the loss of injectivity. Specific to this test case is also the fact that as soon as one of the $c_i$ coefficients in $\varphi^{\exp}$ is zero, it does not contribute anymore to the value of the solution, but in the tangent space it is still present, and so will have an effect on the configuration of the observations.}

  \begin{figure}[H]
    \centering
    \subfloat[$e(t)$ for $p\in \{1, 2, 5, {20}\}$]{
      \includegraphics[width=0.3\linewidth]{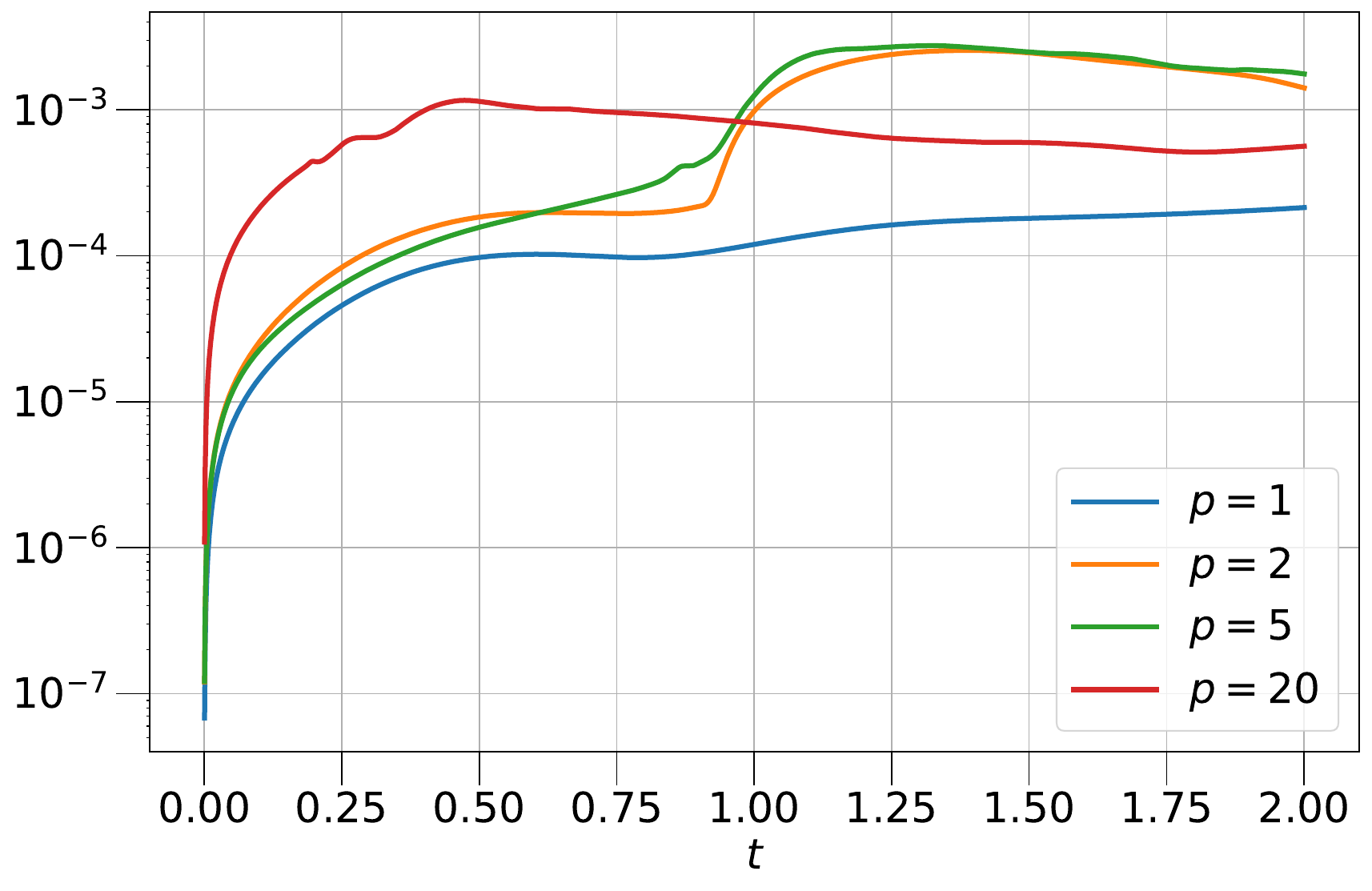}
      \label{fig:FP 2d error}
    }
    \subfloat[$\beta(t)$ for $p\in \{1, 2, 5, {20}\}$]{
      \includegraphics[width=0.3\linewidth]{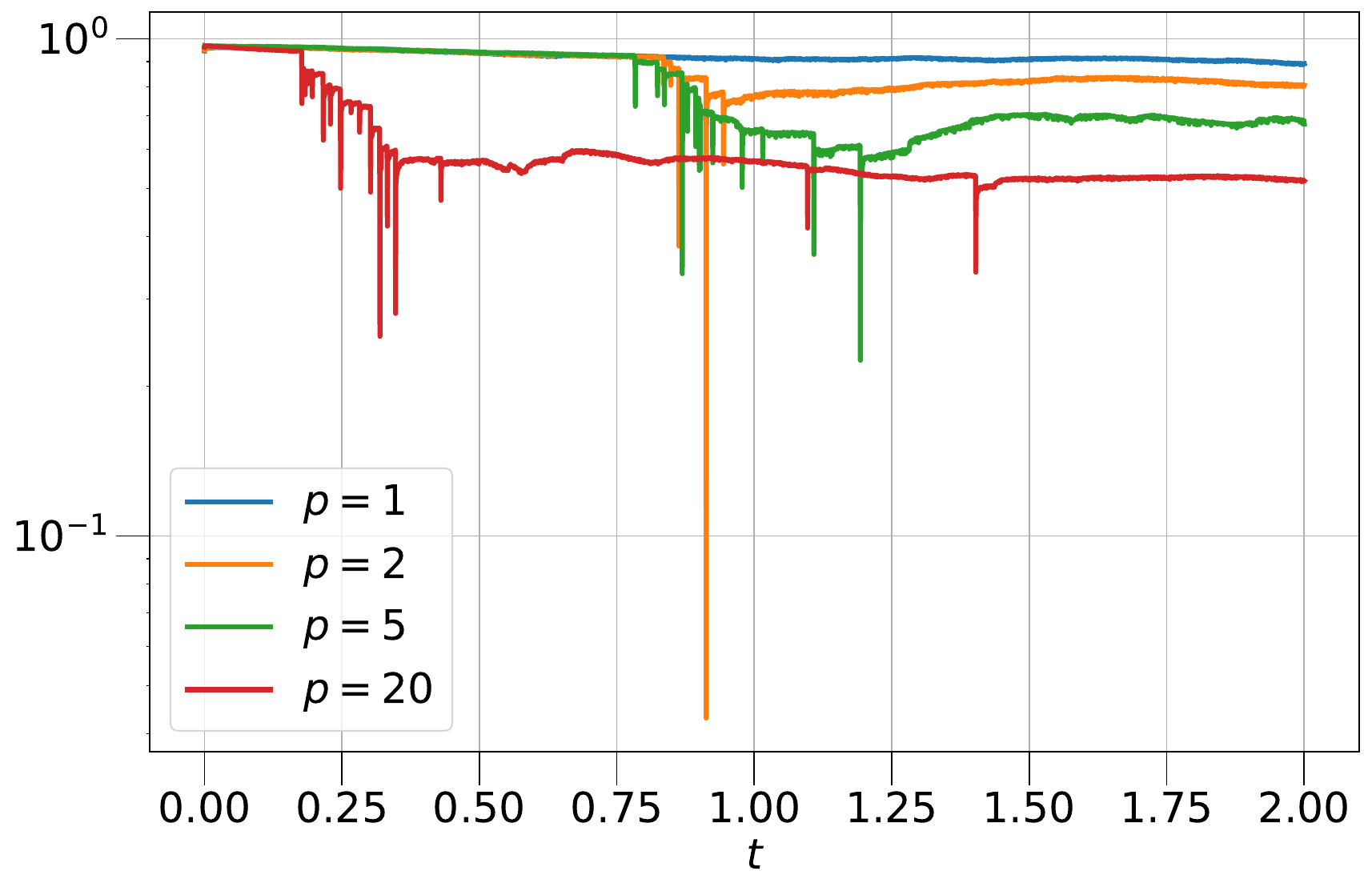}
      \label{fig:FP 2d beta}
    }
    \caption{FP 2D: The approximation error and stability constant ($m=40$).}
    \label{fig:FP 2d performance}
  \end{figure}

  \begin{figure}[H]
    \centering
    \subfloat[$p=1$]{
      \includegraphics[width=0.30\linewidth]{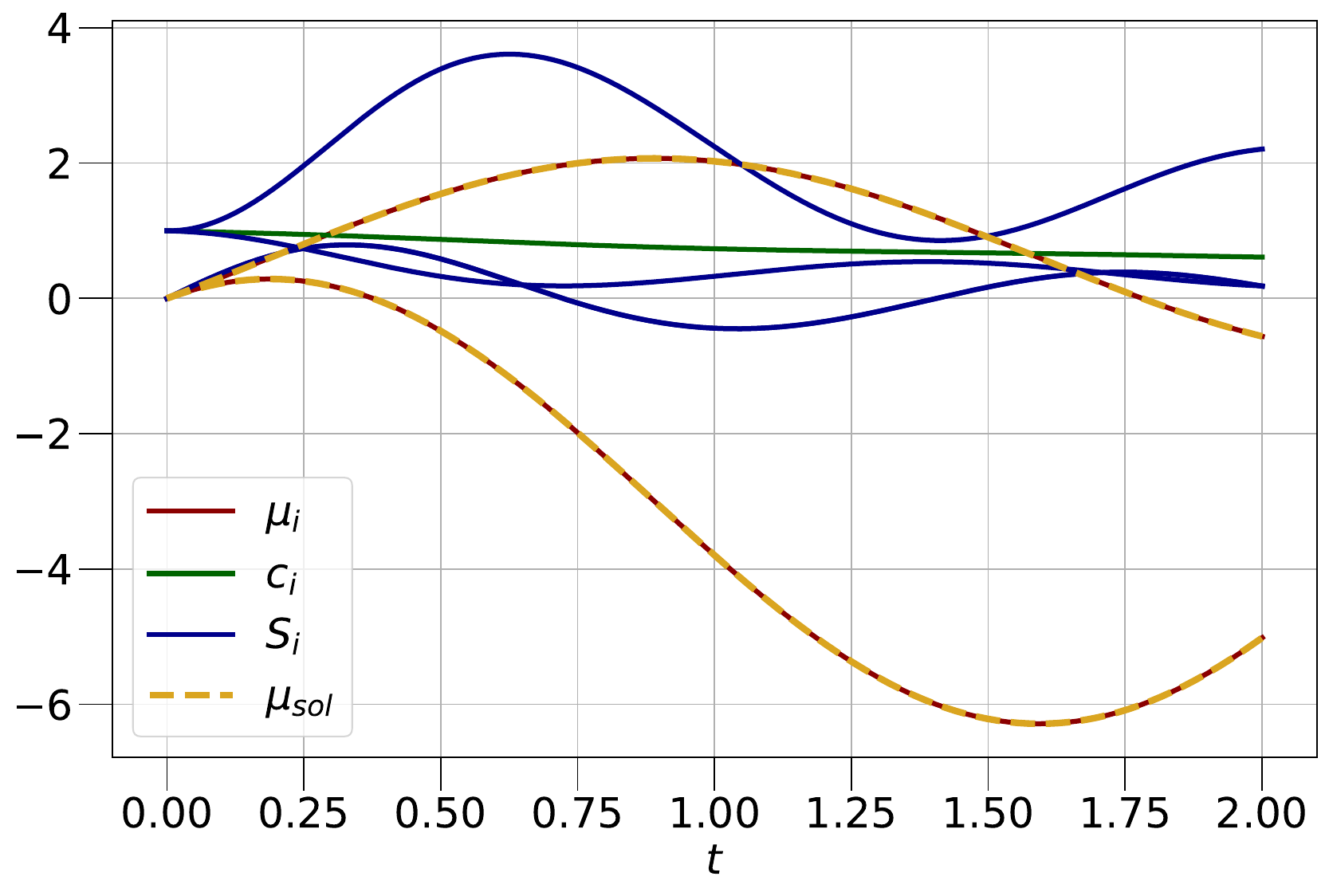}
    }
    \subfloat[$p=2$]{
      \includegraphics[width=0.30\linewidth]{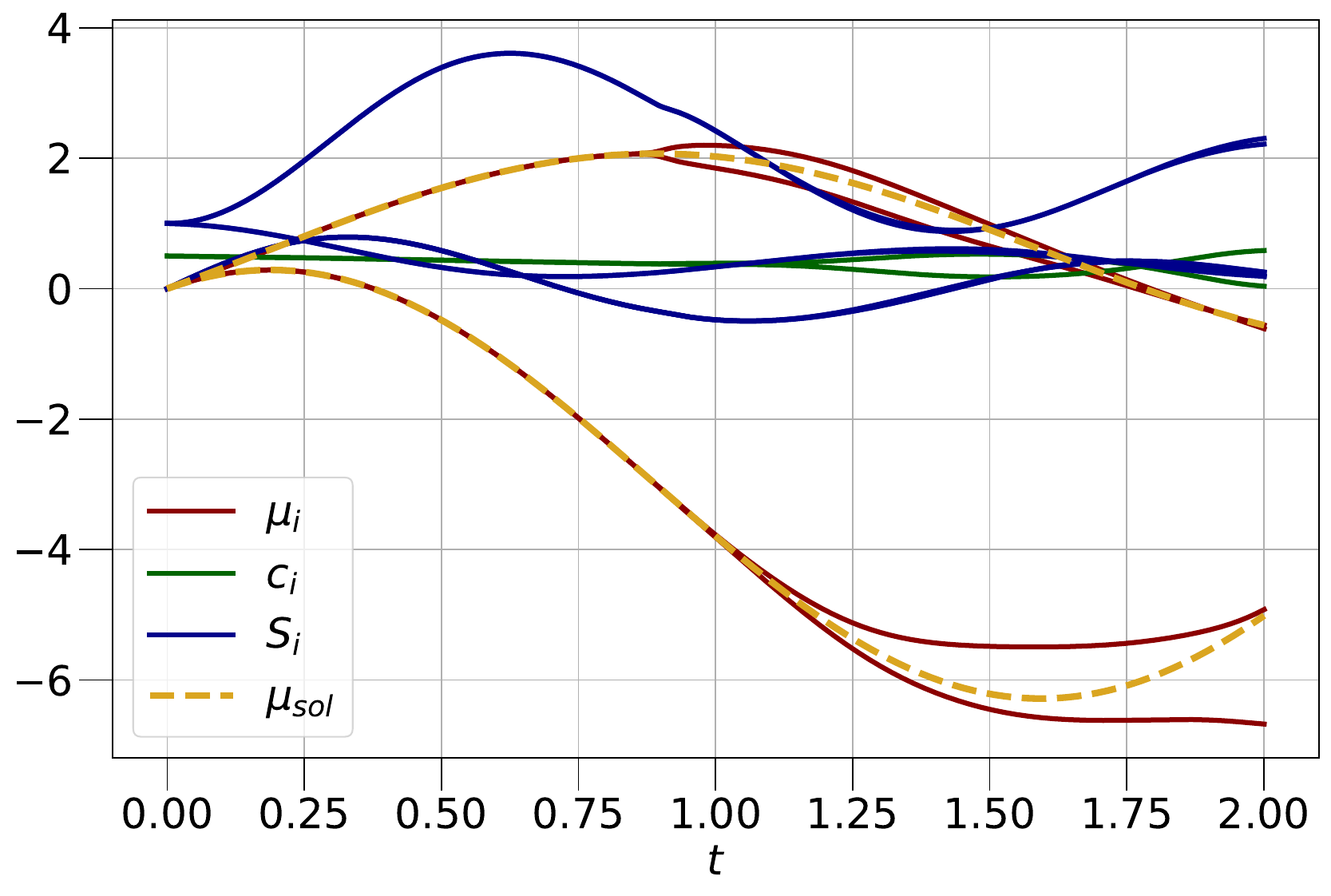}
      \label{fig:FP 2d parameters 2}
    }
    \subfloat[$p=5$]{
      \includegraphics[width=0.30\linewidth]{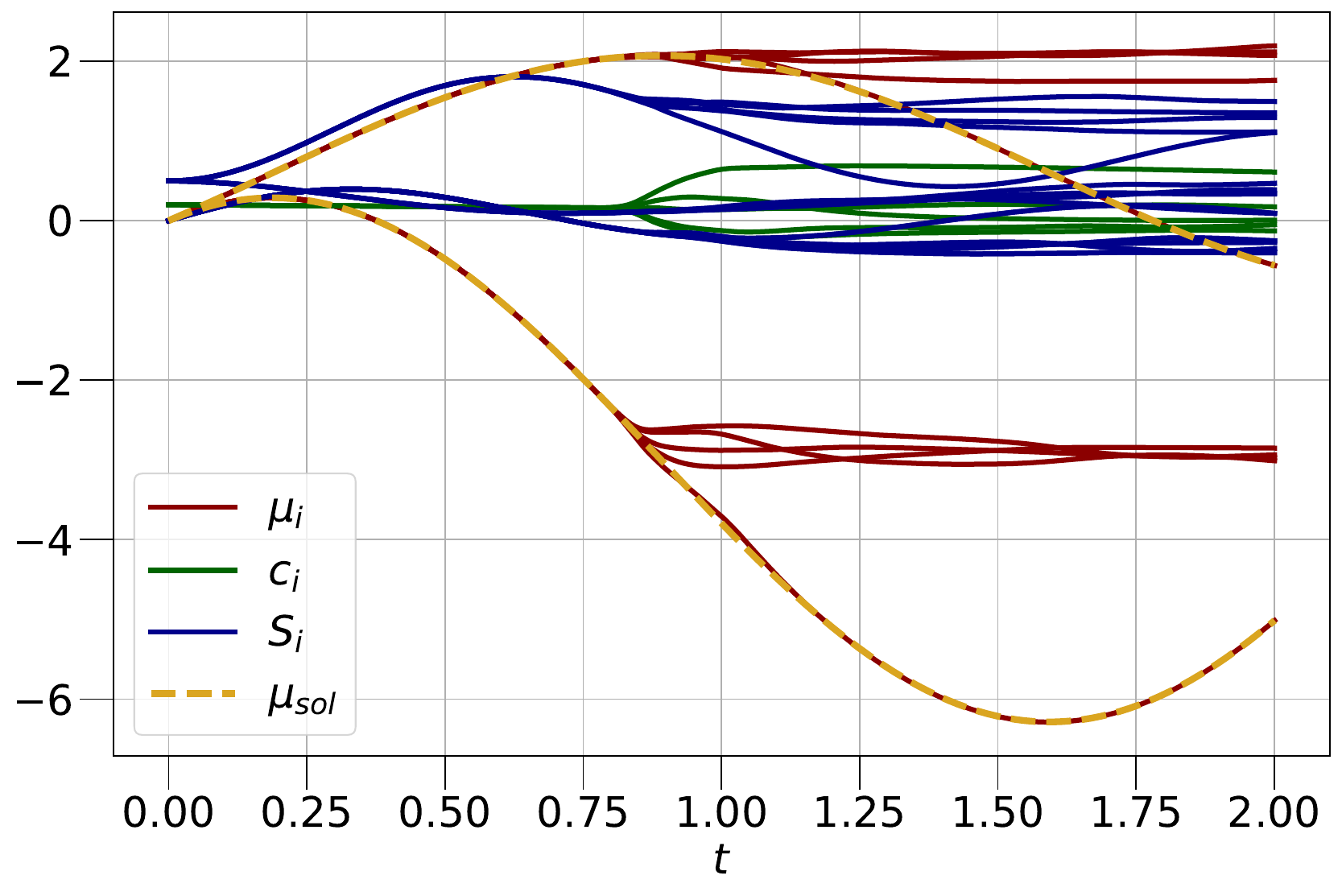}
      \label{fig:FP 2d parameters 5}
    }
    \caption{FP 2D: The evolution of the parameters $\theta(t)$ of the decoder for $p\in \{1, 2, 5\}$.}
    \label{fig:FP 2d theta}
  \end{figure}

  \begin{figure}[H]
    \centering
    \subfloat{
      \includegraphics[width=0.7\linewidth]{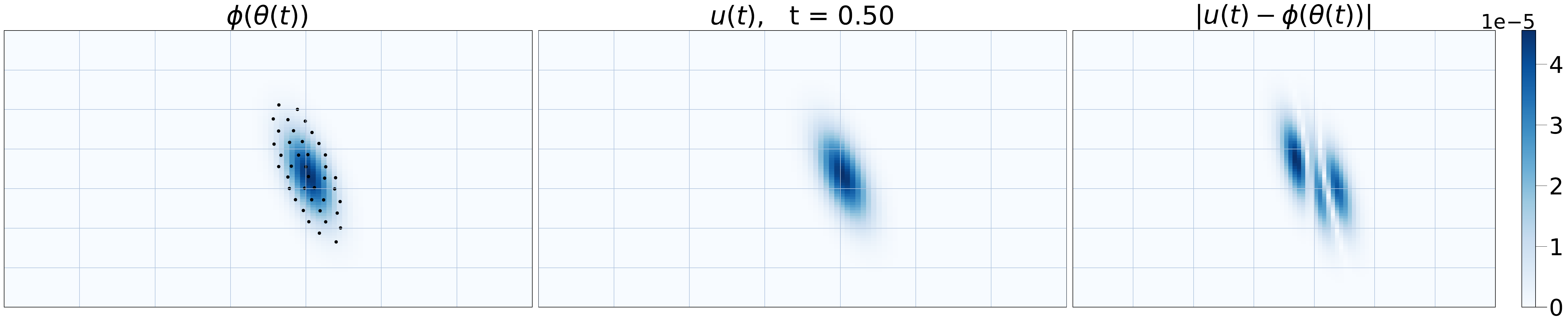}
    }\hspace{0.1cm}
    \subfloat{
      \includegraphics[width=0.7\linewidth]{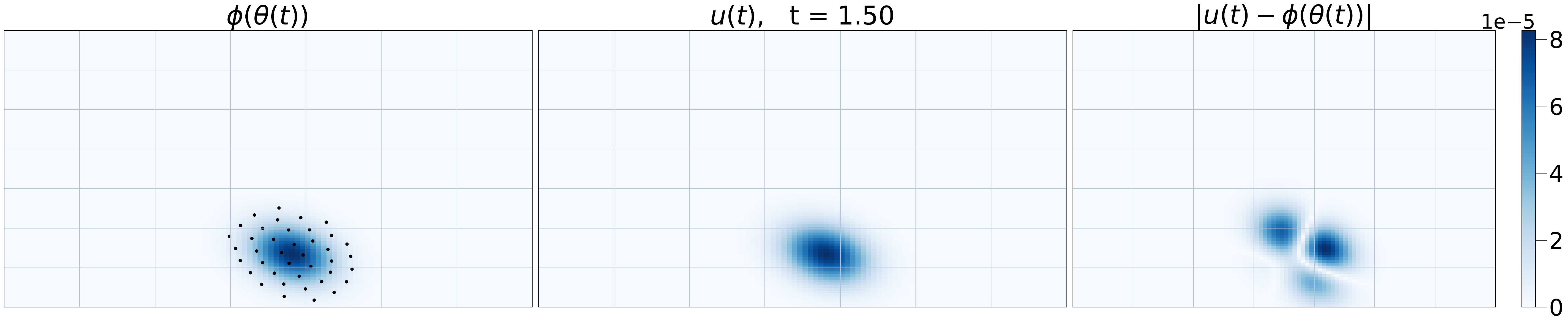}
    }
    \caption{FP 2D: The approximation, exact solution and pointwise difference at $t = 0.5,1.5$ for $p=1$. The locations of the observations can be seen in black in the approximation figures.}
    \label{fig:FP solution p1}
  \end{figure}

  Next, we discuss $d=6$. We again make use of {a decoder of the same form as for $d=2$}, with $p \in \{1,2,5\}$. We now set {$S_i = s_i \id$} and  ${s_i \in \R_{+}}$ for $i=1,\dots, 4$, to limit the number of parameters of the network: $n\in\{8, 16, 40\}$. We use $m=45$ Gaussian observations with $\Sigma = 0.5\id$. We set the regularization parameter $\eps=10^{-5}$, and work with a time step $\d t = 5\cdot 10^{-4}$ on the interval $[0, T]=[0,1]$. In \cref{fig:FP 6d res} we report a sampled error $e(t)$ and the stability constant $\beta(t)$. The error is computed by taking the parameters from the true solution, \cref{eq:FP true solution}, and using those to construct a normal distribution from which we sample $10^{5}$ points. These points are then used to compute the error numerically. We can see a similar behaviour as in the 2D case, where all models perform well but overparametrization with a constant number of observations degrades the performance of the decoder.
  \begin{figure}[H]
    \centering
    \subfloat[$e(t)$ for $p\in\{1,2, 5\}$]{
      \includegraphics[width=0.3\linewidth]{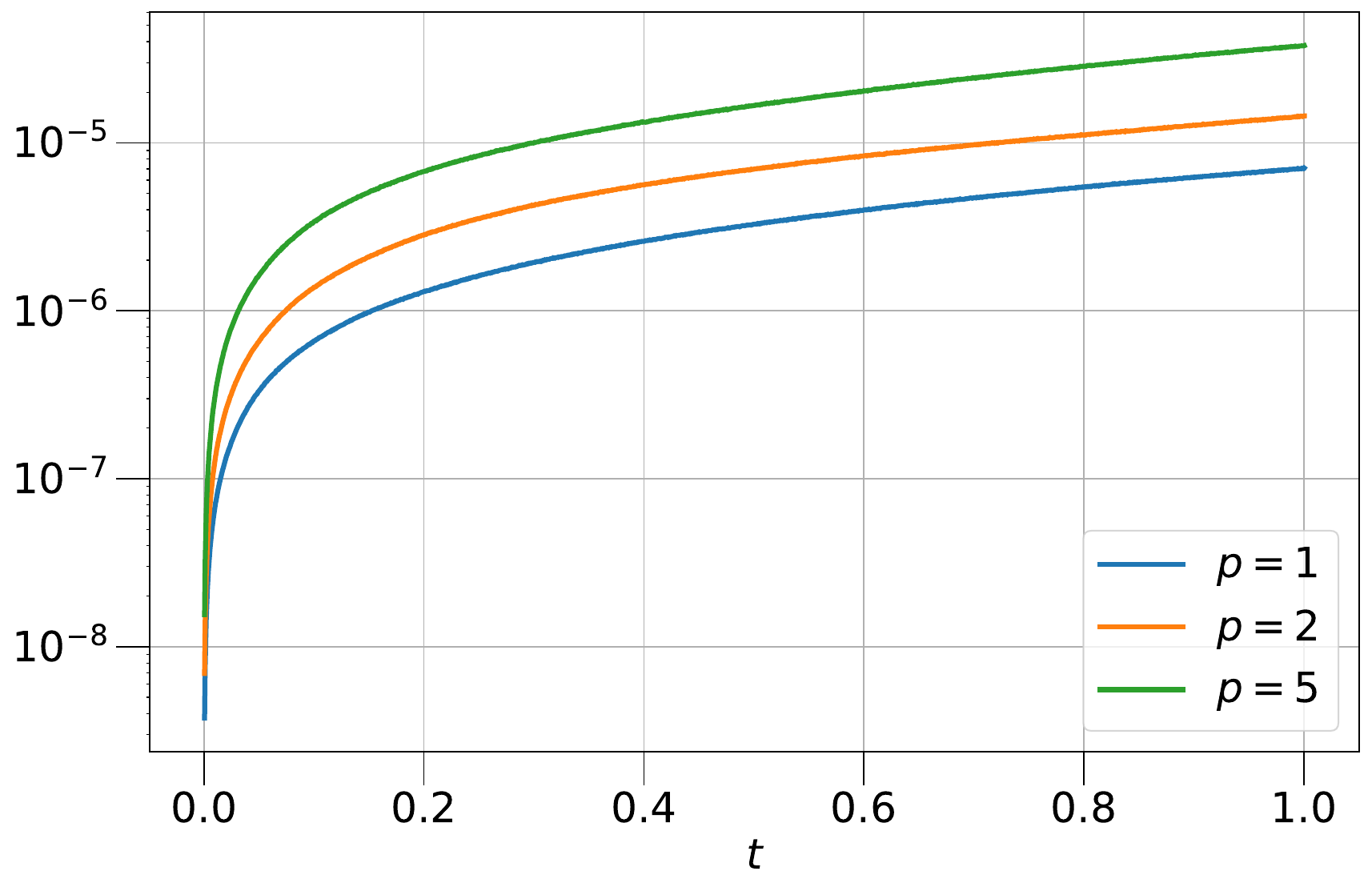}
    }
    \subfloat[$\beta(t)$ for $p\in\{1,2, 5\}$]{
      \includegraphics[width=0.3\linewidth]{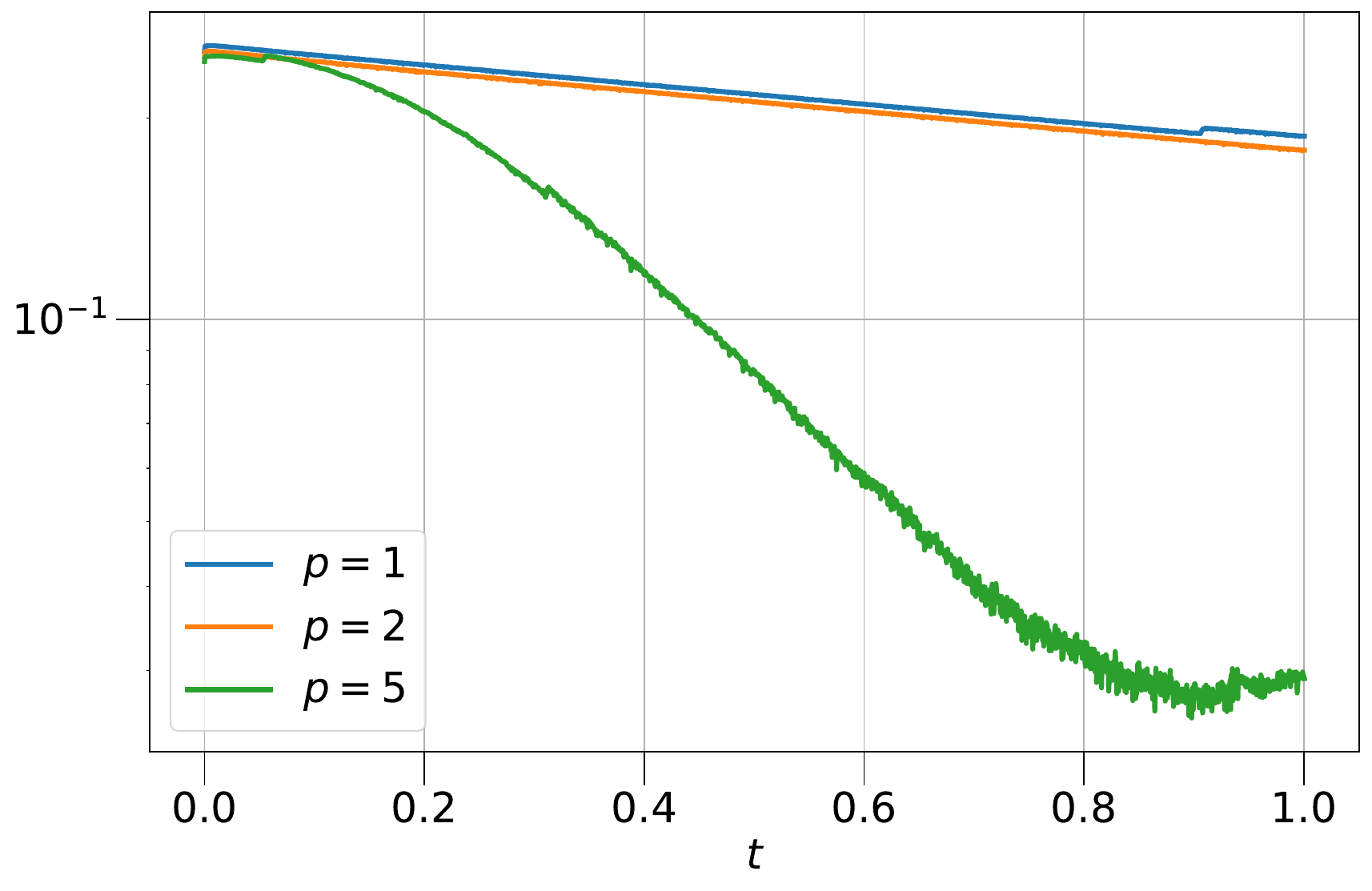}
    }
    \caption{FP 6D: The approximation error and stability constant ($m=45$). The error is computed with $10^{5}$ samples, drawn from the true solution (normalized to a probability measure).}
    \label{fig:FP 6d res}
  \end{figure}

  \subsection{Transport Equation}
  Finally, we discuss a high-dimensional example, to show that one can also use this strategy when the spatial dimension of the PDE is large. To this end, we consider the following transport equation: 
  \begin{equation*}
    \dot{u}(t) = - w(t) \cdot \nabla u(t), \,\,\,  u(0,x) = \exp(-\abs{x}^2/2),
  \end{equation*}
  where $w : [0,T] \rightarrow \R^d$ is a vector field only depending on time. For this example, we take the test case used in \cite[Section 5.3.1]{Bruna2024}, defined by
  \begin{equation}
    w(t) = w_1 \odot \left( \sin(w_2 \pi t) + \frac{5}{4} \right),
  \end{equation}
  where all operations are elementwise, $w_1 = (1,\cdots,d)$ and $w_2 = 2 + \frac{2}{d}(0,\cdots,d-1)$. The exact solution of this equation is given by a translation along the vector field:
  \begin{equation}
    u(t, x) = u\prt*{0, x - \int_0^t w(s) \d s}.
  \end{equation}
  We consider $d=10$, and use {a SNN \cref{eq:SNN decoder} with activation function \cref{eq:exponential_activation_function}, with $p={1,2}$}. The solution is calculated on the interval $[0,T] = [0,1]$, with timesteps $\d t = 10^{-3}$ and a regularization parameter $\eps=10^{-5}$. For $p=1$ we take $m=30$ with $\Sigma=0.7\id$, and for $p=2$ we take $m=60$ with $\Sigma = \id$. The results of this can be seen in \cref{fig:Transport res,fig:Transport parameters}, depicting the evolution of the sampled error $e(t)$, the stability constant $\beta(t)$ and the parameters of the decoders. As we can see, for $p=1$ we have good performance even in $d = 10$, and with only 30 observations we still get a small error. Things are slightly more difficult in the overparametrized case $p=2$, but the performance is reasonable nonetheless.
\begin{figure}[H]
  \centering
  \subfloat[$e(t)$ for $p\in\{1,2\}$]{
    \includegraphics[width=0.3\linewidth]{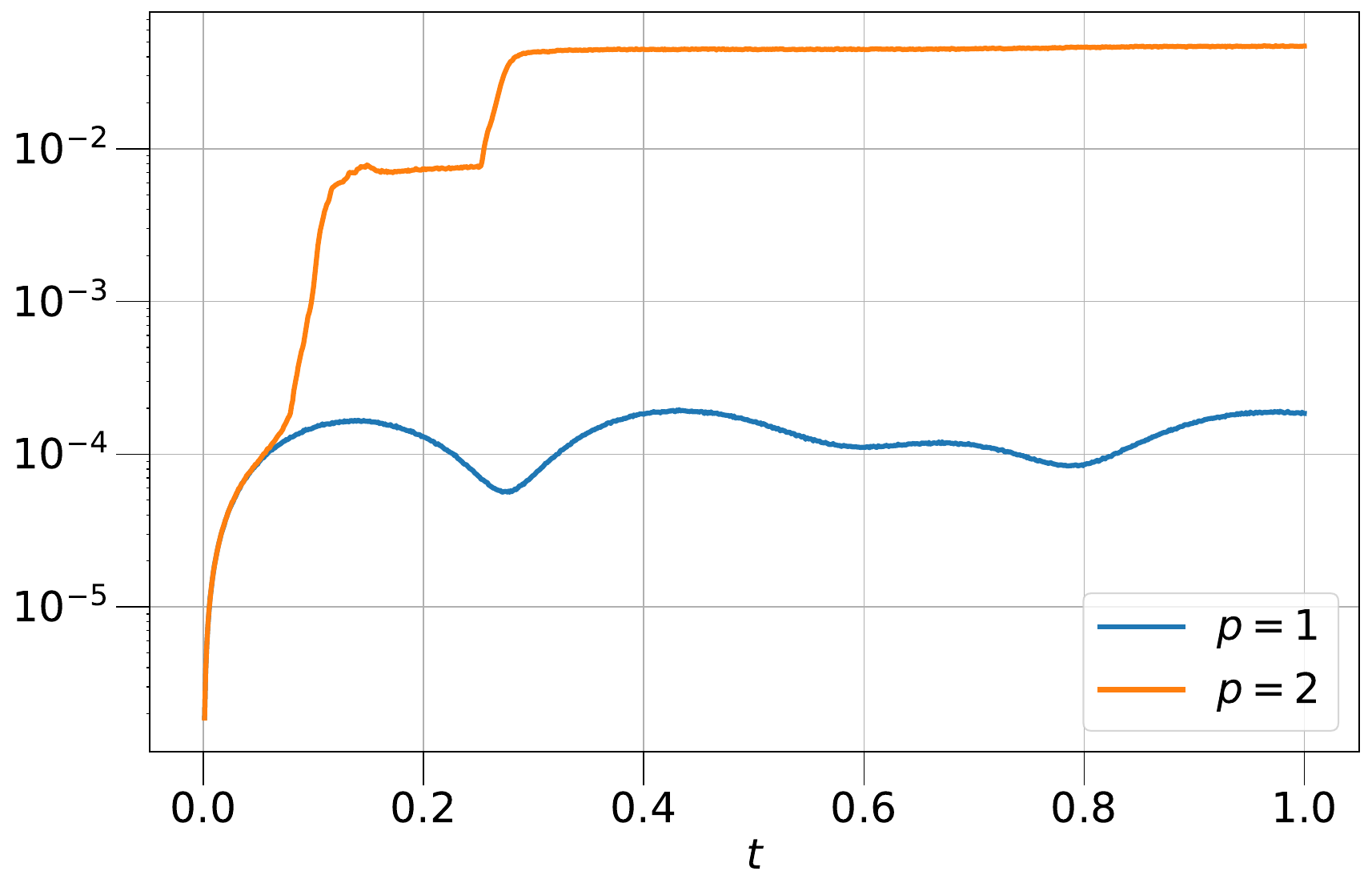}
  }
  \subfloat[$\beta(t)$ for $p\in\{1,2\}$]{
    \includegraphics[width=0.3\linewidth]{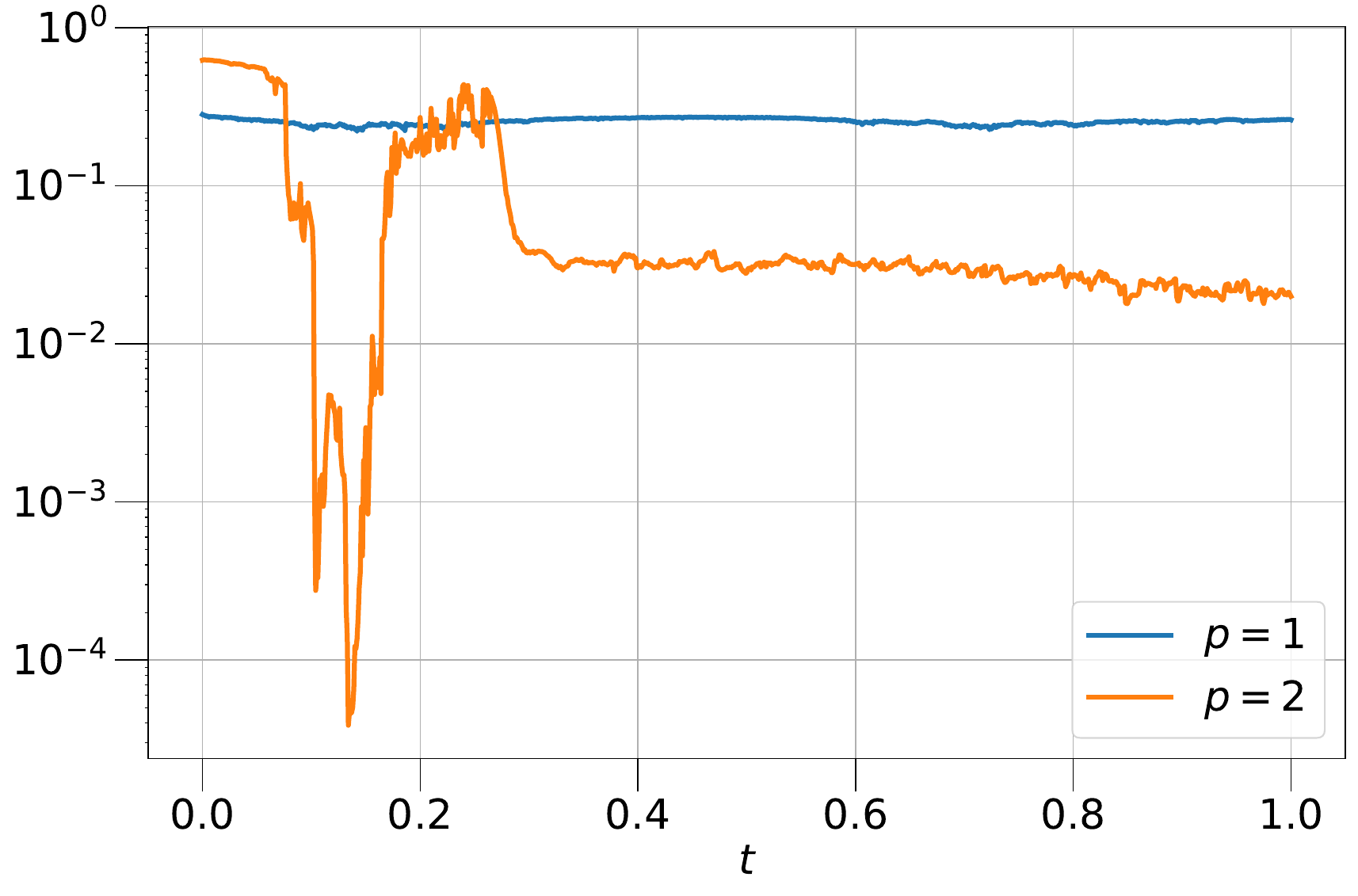}
  }
  \caption{Results of the transport equation for $d=10$. The error is computed with $10^{5}$ samples, drawn from the true solution (normalized to a probability measure).}
  \label{fig:Transport res}
\end{figure}
\begin{figure}[H]
  \centering
  \subfloat[$p=1$]{
    \includegraphics[width=0.3\linewidth]{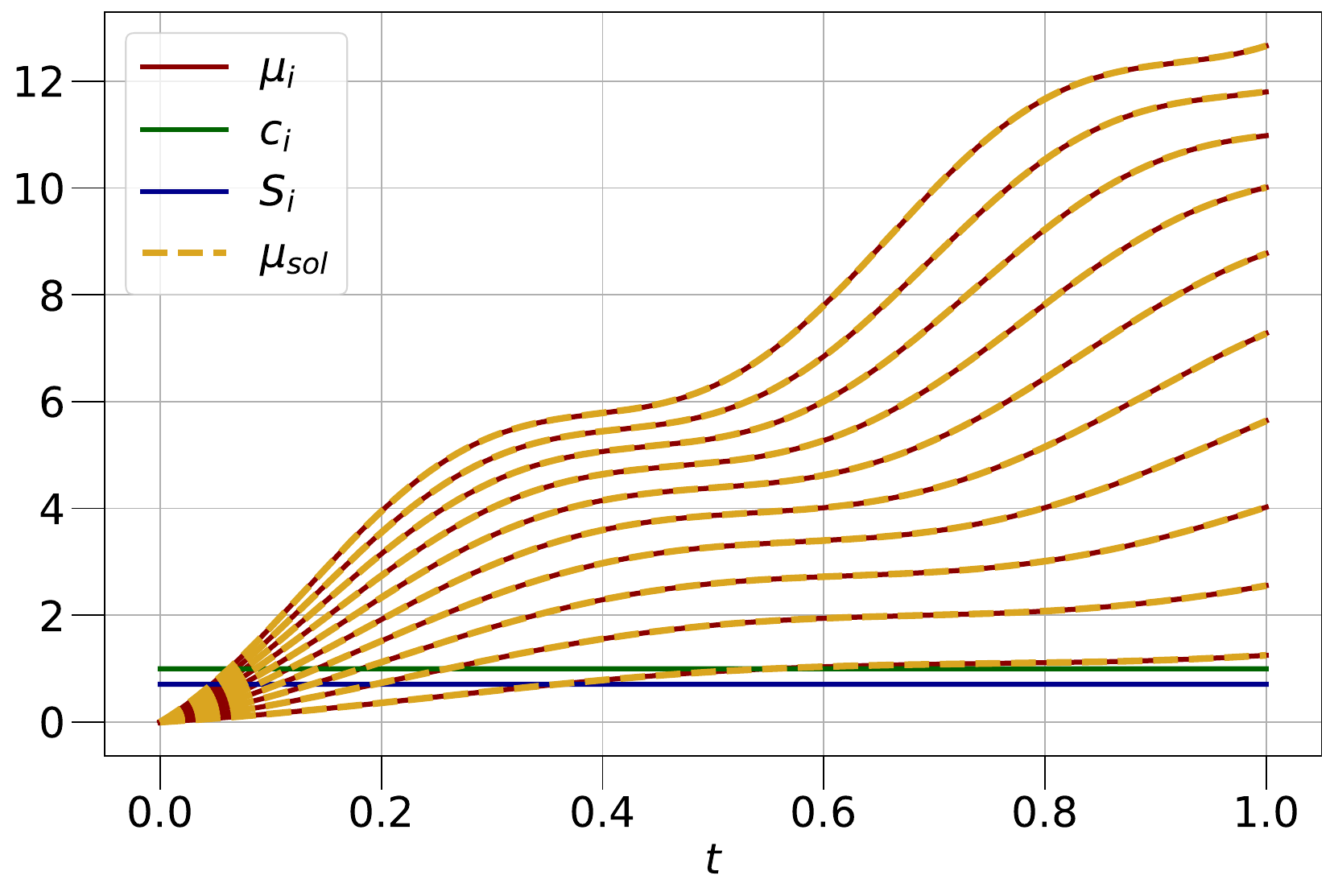}
  }
  \subfloat[$p=2$]{
    \includegraphics[width=0.3\linewidth]{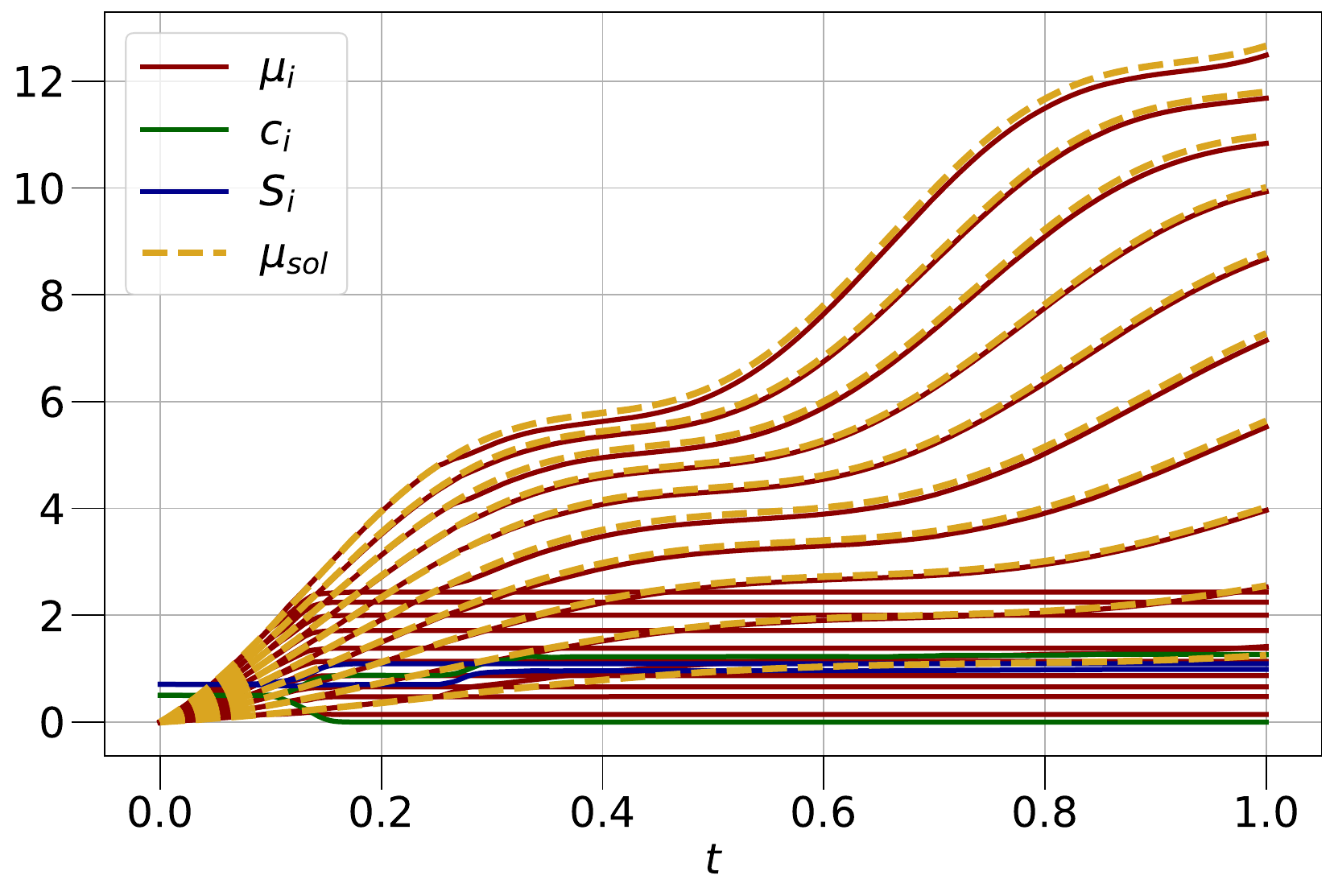}
  }
  \caption{Evolution of the parameters $\theta(t)$ of the decoder for $p\in\{1,2\}$. For $p=2$ we see a similar split as in the Fokker-Planck examples.}
  \label{fig:Transport parameters}
\end{figure}

\section{Conclusion}
We have proposed a scheme for dynamical sampling in nonlinear dynamical approximation schemes for PDEs. Our scheme is formulated in a general Hilbert space, and connects theoretical results on the approximation power of nonlinear approximations with results on the practical performance of the solver. The obtained error bounds show that the quality of the solution benefits from dynamically evolving the sample locations. Particularly powerful error bounds were obtained for the class of gradient flows. In any setting, the error bound involves a stability constant, which is maximized with every update of the sampling. As a consequence, the dynamical sampling is stable by construction. It is also sparse in the sense that few samples are needed, often fewer than the degree of freedom in the decoder class. This phenomenon is related to overparametrization of our decoders. Numerical experiments confirm that the dynamical sampling leads to an evolution that is both accurate and stable. In particular, we observed good results in relatively high-dimensional examples, where one is forced to consider sparse sampling due to the curse of dimensionality.

\section*{Acknowledgments}
The authors would like to warmly thank Albert Cohen and Mark Peletier for fruitful discussions.

\appendix
\newpage
\section{Additional Numerical Results}
\label{app:additional numerical results}
\subsection{Viscid Burger's Equation (VB 1D)}
\cref{fig:VB results,,fig:VB looped} show the results for the 1D viscid Burger's equation, given by
\begin{equation*}
  \dot u(t) = - u(t) \partial_x u(t) + \alpha \partial_x^2 u(t),
\end{equation*}
for $u: [0,T] \mapsto L^2(\mathbb{R})$ and $u(0) = u_0$. As decoder we take a shallow neural network given by \eqref{eq:SNN decoder} with $p=10$ and $\sigma(a_i,b_i)(x) := \tanh(a_i x + b_i)$. This gives $\theta = (\theta_i)_{i=1}^{10}$, $\theta_i = (a_i,b_i,c_i) \in \bR^3$, and thus an observation space $\Vn$ with $n=30$. We solve for $\alpha = 0.01$ until $T=5$, subject to initial condition
\begin{equation*}
  u_0(x) = \mathbbm{1}_{\{0<x<1\}}(x).
\end{equation*}
For the reference solution, we use an implicit-explicit scheme on a grid of 5000 equidistant points on $\Omega = [-3,5]$, with $\dt = 0.1 \, \d x$. For our scheme, we use $m=30$ observations with variance $\sigma=0.01$, time step $\dt=10^{-3}$ and regularization parameter $\varepsilon = 10^{-6}$. With these settings, we observe good numerical performance, both visually and in terms of the error $e(t)$. In particular, we see that the numerical solution recovers from a rather poor initial fit, even though the stability constant $\beta(t)$ is only $\mathcal{O}(10^{-2})$.
For a whole-line interpretation, the decoder and its parameter derivatives must belong to $L^2(\bR)$. A generic finite sum of $\tanh$ functions does not satisfy this requirement: its limits at both infinities must vanish, and the admissible parametrization must preserve this cancellation. Without these constraints, the abstract $L^2(\bR)$ analysis does not apply to this experiment; a bounded-domain interpretation requires specifying the domain and boundary treatment used in the implementation.
\begin{figure}[H]
  \centering
  \subfloat[$t=0$.]{
    \includegraphics[width=0.3\linewidth]{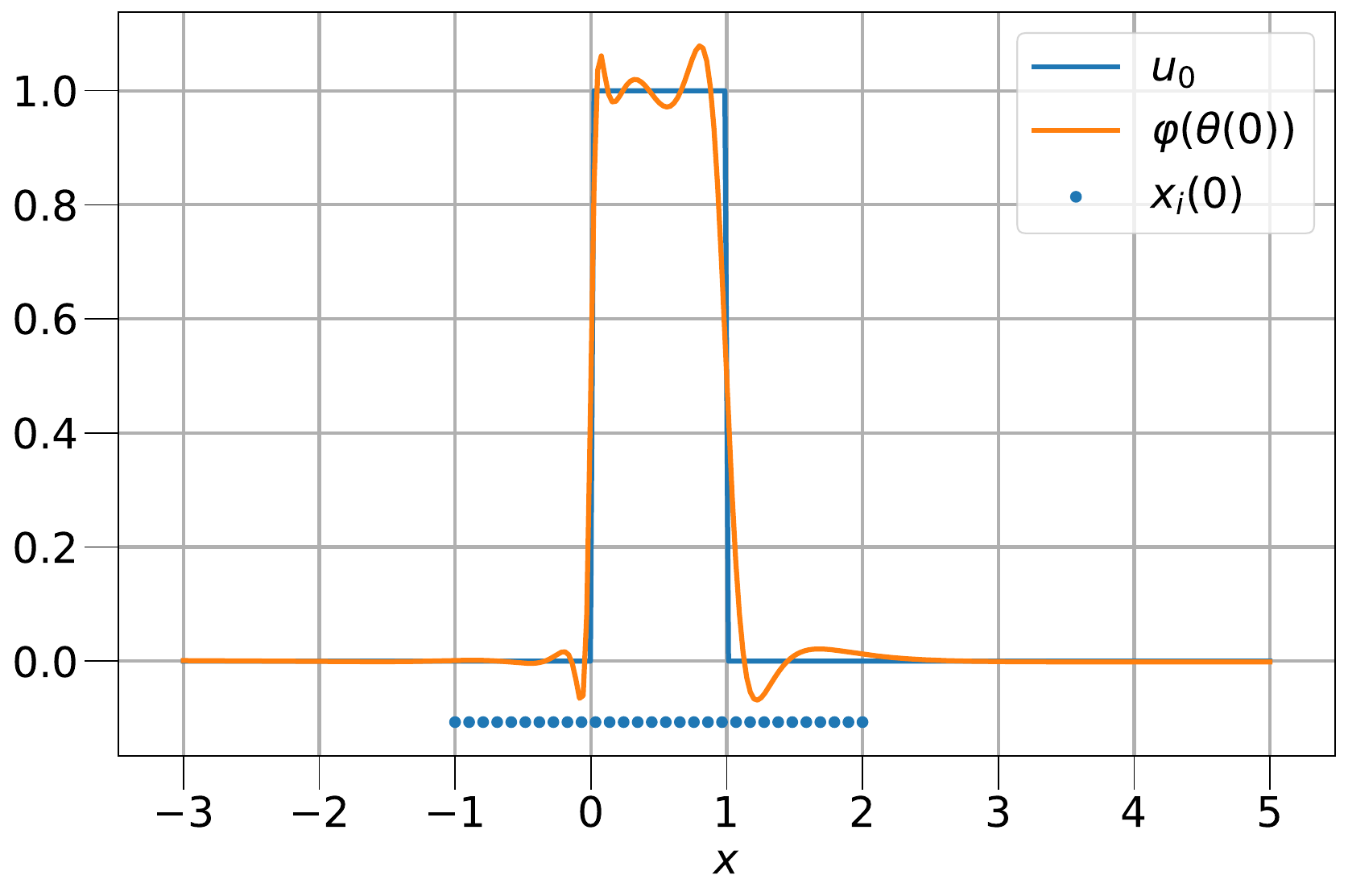}
    \label{fig:VB initial fit}
  }
  \subfloat[$t=2$.]{
    \includegraphics[width=0.3\linewidth]{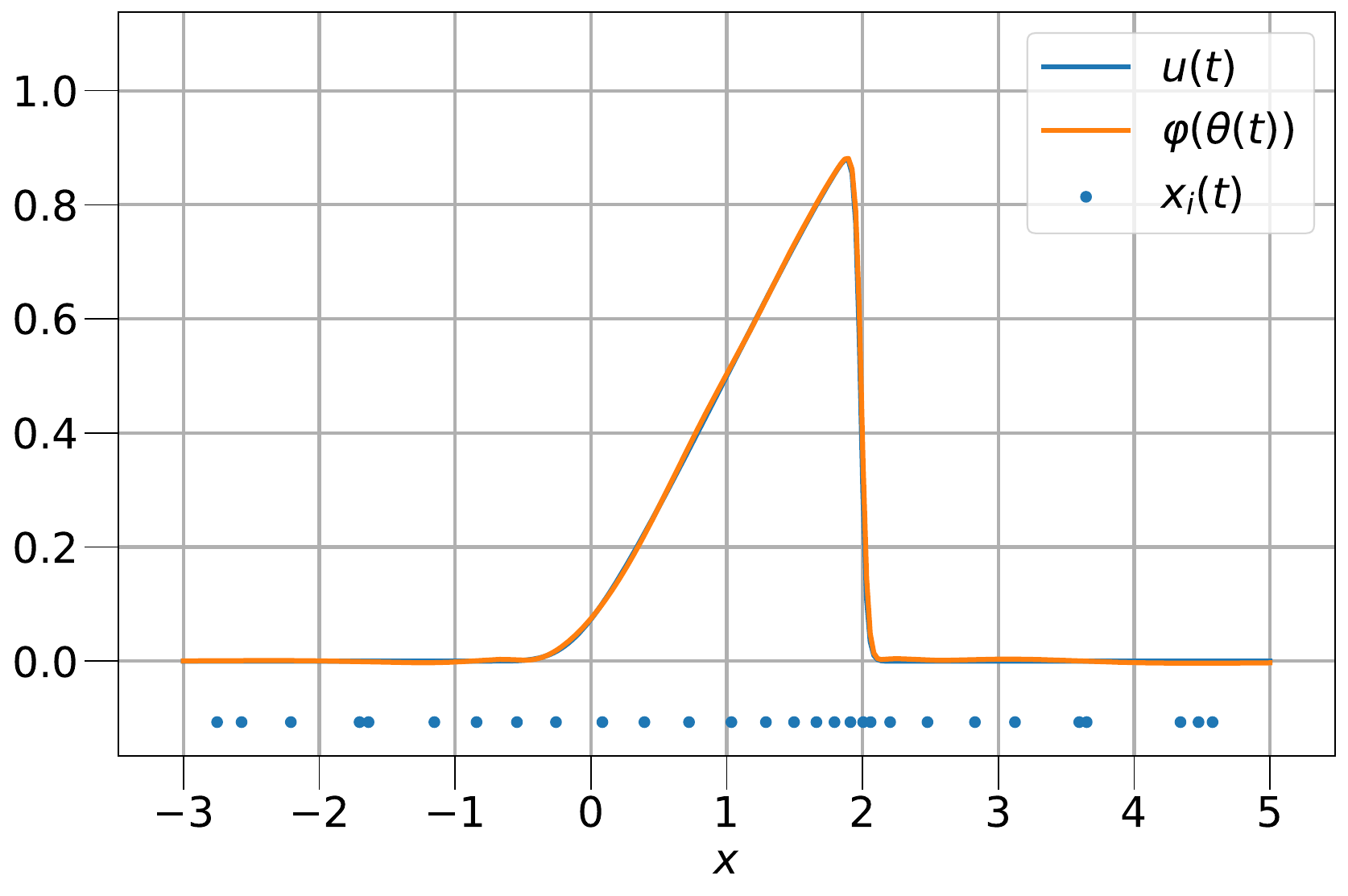}
    \label{fig:VB snapshot t=2}
  }
  \subfloat[$t=T=5$.]{
    \includegraphics[width=0.3\linewidth]{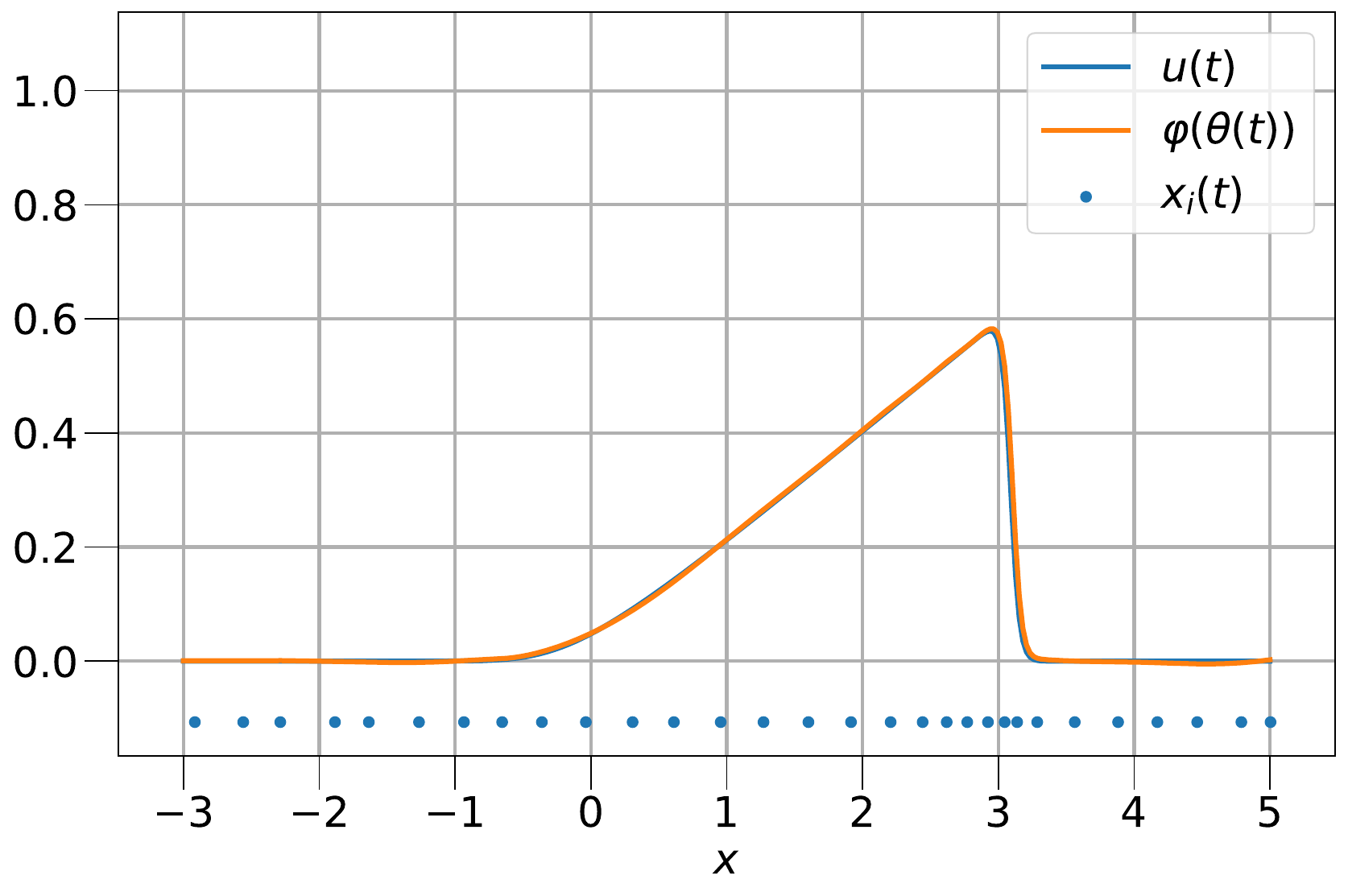}
    \label{fig:VB snapshot t=5}
  }
  \caption{VB 1D: Exact solution and its approximation ($n=m=30$,\, $\sigma=0.01$).}
  \label{fig:VB results}
\end{figure}

\begin{figure}[H]
  \centering
  \subfloat[$\beta(t)$.]{
    \includegraphics[width=0.3\linewidth]{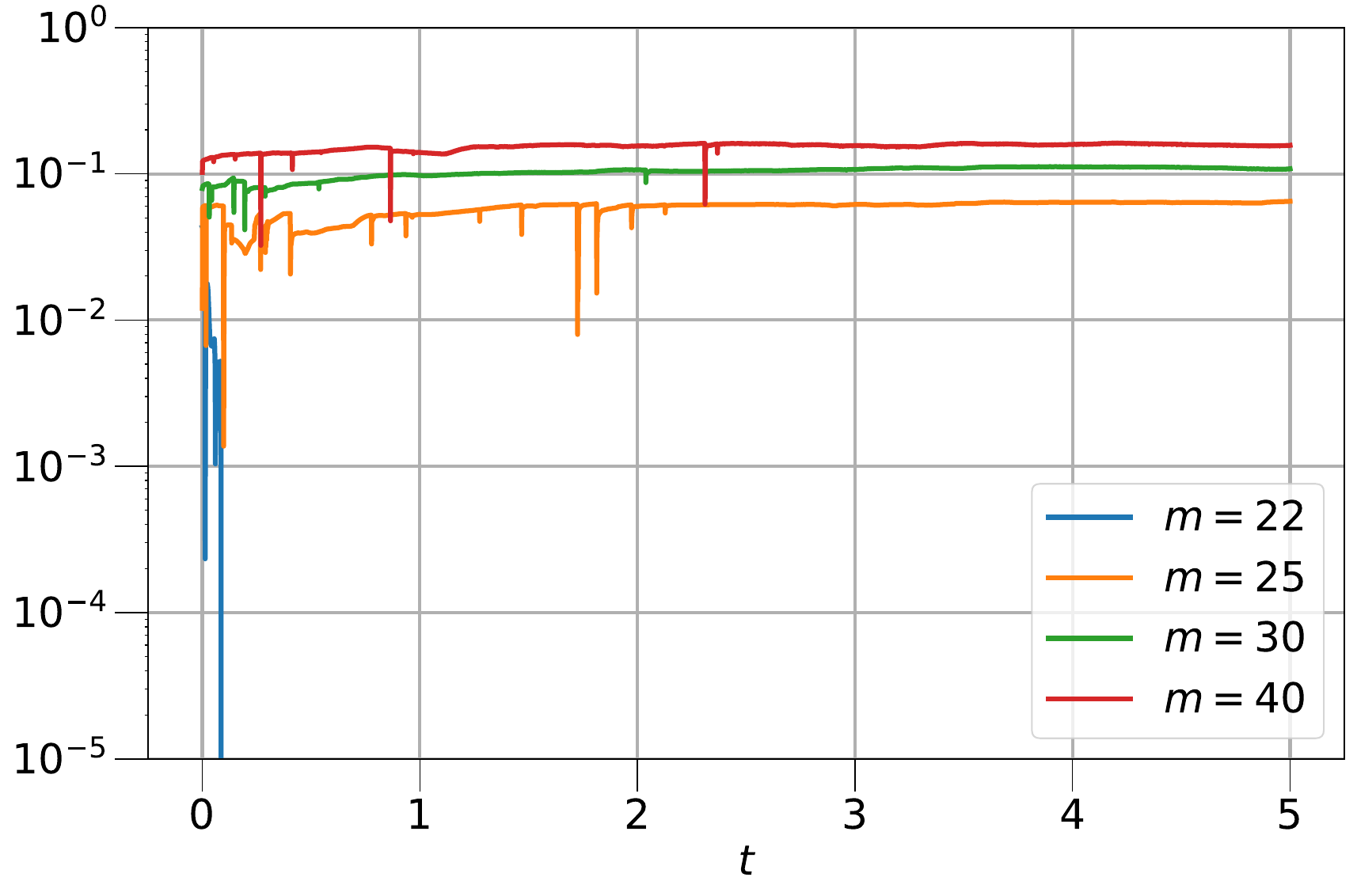}
    \label{fig:VB beta looped}
  }
  \subfloat[$e(t)$.]{
    \includegraphics[width=0.3\linewidth]{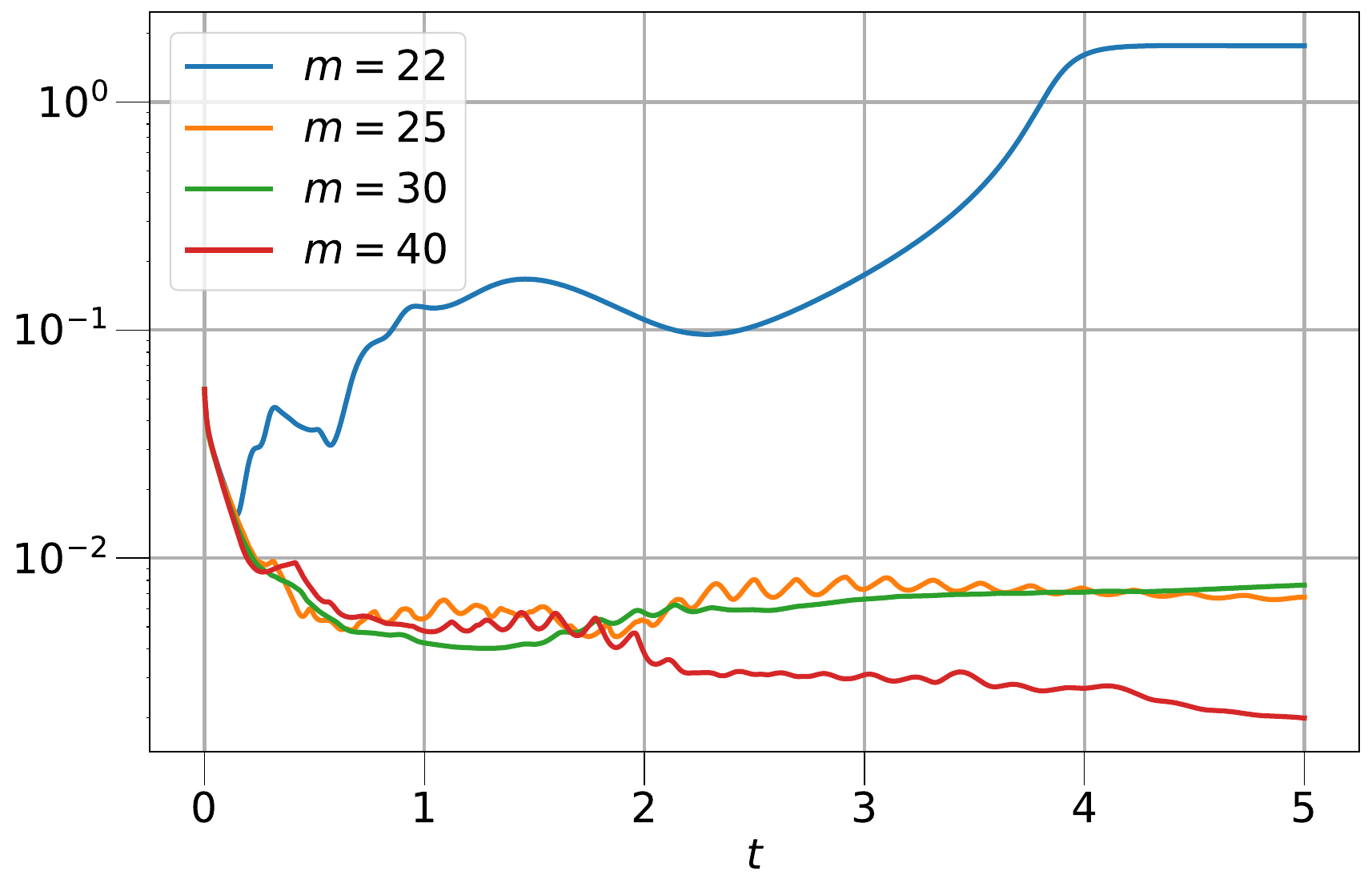}
    \label{fig:VB error looped}
  }
  \subfloat[$n_\eff(t)$.]{
    \includegraphics[width=0.3\linewidth]{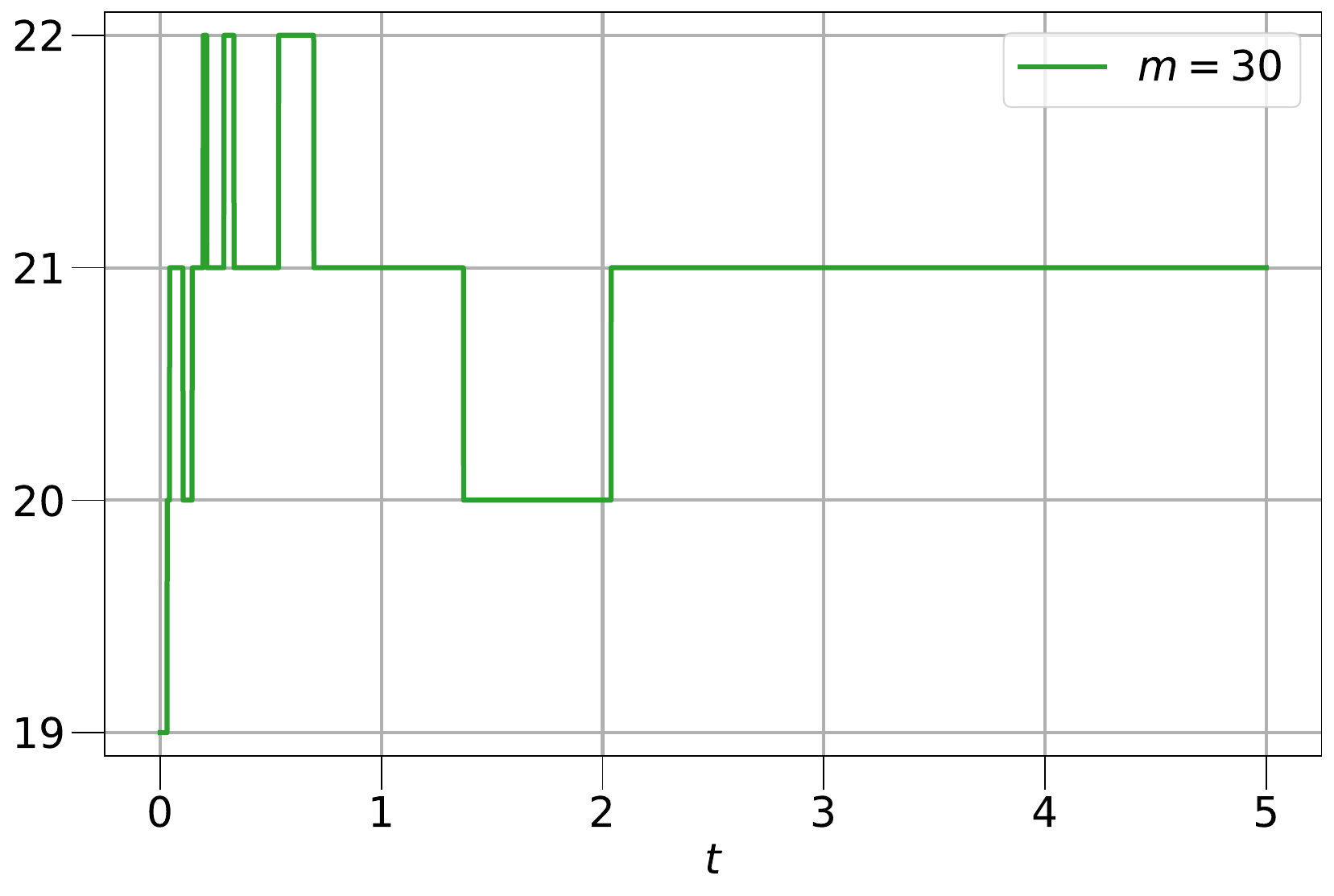}
    \label{fig:VB efective dimension}
  }
  \caption{VB 1D: Behaviour of the solution under a varying number of measurements $m$ $(\sigma = 0.01)$.}
  \label{fig:VB looped}
\end{figure}

\newpage
\section{PDEs with one-sided Lipschitz continuity}
\label{sec:example-PDEs}
In this section, we discuss the error bound \cref{eq:error-bound} from \cref{theorem:best-fit-estimator} for gradient flows, including a parabolic example. Classical variational formulations of parabolic PDEs are described, e.g., in section 6.1 of \cite{EG2013} (see in particular definition 6.5 and theorem 6.6 on existence and uniqueness). To apply our strong-formulation analysis, the one-sided Lipschitz condition \cref{eq:one-sided-Lip} must be checked for the strong realization $f$ in the chosen $V$-inner product; it does not follow from the variational formulation alone. For gradient flows on Hilbert spaces, proper, lower semicontinuous, $\lambda$-convex functionals yield this condition with $L=-\lambda$ (see, e.g., \cite[Section 4.1]{ambrosio2013users}). In this application we take $H=V$ with the same inner product, so that $\cJ=\bG_V$.

A proper, lower semicontinuous, $\lambda$-convex functional $\cF: V \to (-\infty,+\infty]$ satisfies
$$
  \cF((1-t)v + t w)\leq (1-t)\cF(v)+t\cF(w)- \frac \lambda 2 t(1-t)\norm{v-w}^2_V, \quad \forall (t, v,w)\in [0, 1]\times\dom(\cF)\times\dom(\cF).
$$
We denote by $\dom(\cF)\coloneqq\{v\in V\cond \cF(v)<+\infty\}$ the effective domain of $\cF$. It is nonempty since $\cF$ is proper. The Fréchet subdifferential $(D^{F}\cF)_v$ at a point $v\in\dom(\cF)$ is the set of $\xi\in V'$ defined by
$$
  (D^F\cF)_v \coloneqq \{ \xi\in V'\cond \cF(w) \geq \cF(v) + {}_{V'}\inner{\xi, w-v}_V + \frac \lambda 2 \norm{w-v}^2_V,\; \forall w\in V \}.
$$
We set $(D^F\cF)_v=\emptyset$ outside $\dom(\cF)$ and define
\begin{equation}
  \dom(D^F \cF) \coloneqq \{ v\in V \cond (D^F\cF)_v \neq \emptyset\},
\end{equation}
and for every $v\in \dom(D^F \cF)$, we denote by $\xi_v\in V'$ the element of minimal norm from $(D^F\cF)_v$. The set of Riesz representers of $(D^F\cF)_v$ are the (Fréchet) subgradients of $\cF$ at $v$,
\begin{equation}
  \label{eq:grad-F}
  \grad_V \cF(v)\coloneqq\{ \bG_V^{-1}(\xi) \cond \xi \in (D^F\cF)_v \}.
\end{equation}

Given an initial condition $u_0\in\dom(\cF)$, the gradient flow evolution associated to $(V, \cF, \bG_V)$ consists in searching for a locally absolutely continuous curve $u:\bR_+\to V$ that satisfies
\begin{equation}
  \begin{cases}
    \bG_V\dot u(t) & \in - (D^F\cF)_{u(t)}, \quad \text{for a.e.~} t>0 \\
    u(0)           & =u_0.
  \end{cases}
\end{equation}
By means of the Riesz lift, this can equivalently be expressed in the ambient space $V$ as
\begin{equation}
  \begin{cases}
    \label{equ:gradient_flow_Hilbert}
    \dot u(t) & \in - \grad_V \cF(u(t)), \quad \text{for a.e.~} t>0 \\
    u(0)      & =u_0.
  \end{cases}
\end{equation}
The following result is classical and can be found in, e.g., \cite[Theorem 4.1]{ambrosio2013users}.

\begin{theorem}
  Let $\cF: V \to (-\infty,+\infty]$ be a proper, lower semicontinuous, $\lambda$-convex functional.
  \label{thm:uniqueness-gf-solution}
  \begin{enumerate}
    \item If $u_0 \in \dom(\cF)$, then \cref{equ:gradient_flow_Hilbert} has a unique solution, and for almost every $t>0$, $\dot{u}(t) = - \bG_V^{-1}\xi_{u(t)}$.
    \item $\cF$ is a Lyapunov function, i.e., along solutions $u:[0,T]\to V$ the function is decreasing:
          $$\frac{\d}{\dt}\cF(u(t)) \leq 0, \quad \text{for a.e.~} t>0.
          $$
  \end{enumerate}
\end{theorem}

In connection with our main setting from \Cref{eq:PDE}, taking $H=V$ gives $\cJ=\bG_V$, and point 1 of \Cref{thm:uniqueness-gf-solution} yields
\begin{equation}
  \label{eq:F-grad-flow}
  F(v)=-\xi_v\in V',
  \qquad f(v)=-\bG_V^{-1}\xi_v\in V,
  \qquad v\in D(F)=D(f)=\dom(D^F\cF).
\end{equation}
One first interesting result about the dynamical approximation scheme \cref{eq:dirac-frenkel} applied to the current gradient flow setting is that it guarantees that $\cF$ is a Lyapunov function along the approximate solution $\varphi(\theta(t))$ provided $\varphi(\theta(t))\in \dom(D^F \cF)$ and the subgradient chain rule holds along the numerical trajectory. We record this in the following Lemma.

\begin{lemma}
  Let $t\in \bR_+\to \varphi(\theta(t))$ be the dynamical approximation scheme from \cref{eq:dirac-frenkel} applied to a gradient flow evolution of the type \cref{equ:gradient_flow_Hilbert}. If $\varphi(\theta(t))\in \dom(D^F \cF)$ for a.e.~$t>0$ and $\cF(\varphi(\theta(\cdot)))$ is locally absolutely continuous with the subgradient chain rule valid along this trajectory, then
  $$\frac{\d}{\dt}\cF(\varphi(\theta(t)))\leq 0, \quad \text{for a.e. } t>0.$$
\end{lemma}
\begin{proof}
  For almost every $t>0$, set $g(t):=\bG_V^{-1}\xi_{\varphi(\theta(t))}$. By the chain rule, and by our assumption that $\varphi(\theta(t))\in \dom(D^F \cF)$, we have that
  \begin{align}
    \frac{\d}{\dt}\cF(\varphi(\theta(t)))
     & = {}_{V'}\inner*{\xi_{\varphi(\theta(t))}, (D\varphi)_{\theta(t)}(\dot \theta(t))}_{V} = \inner{g(t), (D\varphi)_{\theta(t)}(\dot \theta(t))}_V,
  \end{align}
  and since, by \cref{eq:dirac-frenkel}, it holds that $(D\varphi)_{\theta(t)}(\dot \theta(t))=-P_{\Tn(t)} \prt*{g(t)}$, we deduce
  \begin{align*}
    \frac{\d}{\dt}\cF(\varphi(\theta(t)))
     & = - \inner{g(t), P_{\Tn(t)} g(t)}_V
    = - \norm*{P_{\Tn(t)} g(t)}_V^2\leq 0.
  \end{align*}
\end{proof}

For the current gradient flow case, we can apply \Cref{theorem:best-fit-estimator} on $D(f)=\dom(D^F\cF)$. Indeed, $\lambda$-convexity gives
\begin{equation}
  \label{eq:lambda-cvx}
  \inner{f(v) - f(w), v - w}_V \leq -\lambda \norm{v-w}^2_V, \quad \forall (v, w)\in \dom(D^F \cF)^2.
\end{equation}
Thanks to this, we can prove the same error bound as in \Cref{theorem:best-fit-estimator} at the cost of assuming that $\varphi(\theta(t))\in \dom(D^F \cF)$. We record this in the following lemma.

\begin{lemma}
  \label{lemma:best-fit-estimator-gf}
  If:
  \begin{itemize}
    \item $\cF$ is proper, lower semicontinuous and $\lambda$-convex,
    \item $u_0 \in \dom(\cF)$ and $\kappa$ is continuous on $[0,t]$,
    \item For all $0\leq s\leq t$, $\varphi(\theta(s))\in \dom(D^F \cF)$ and $\{\ell(s), \Tn(s)\}_Z$ are admissible for reconstruction,
  \end{itemize}
  then the approximation error bound \eqref{eq:error-bound} of the dynamical approximation scheme holds with $L=-\lambda$.
\end{lemma}

\begin{proof}
  The proof follows the same lines as the one of \cref{theorem:best-fit-estimator}. We only need to use the restricted one-sided Lipschitz property \eqref{eq:lambda-cvx}. For almost every $t$, we have
  \begin{align}
    \frac{1}{2}\frac{\d e^2}{\d t} (t)
     & =e(t)\frac{\d e}{\d t} \nonumber                                                                                                                                                                                            \\
     & = \inner*{\frac{\d}{\d t} (u - \varphi(\theta(t))), u - \varphi(\theta(t)) }_V \nonumber                                                                                                                                    \\
     & = \inner*{f(u(t)) -  f(\varphi(\theta(t)))+f(\varphi(\theta(t))) -\frac{\d}{\d t} \varphi(\theta(t)), u - \varphi(\theta(t)) }_V, \quad \text{(with $f$ as in \eqref{eq:F-grad-flow})} \label{eq:grad-F-cvx-inner-products} \\
     & \leq -\lambda e^2(t) +
    \norm*{f(\varphi(\theta(t))) -\frac{\d}{\d t} \varphi(\theta(t))}_V e(t), \quad \text{(by \cref{eq:lambda-cvx})}\nonumber                                                                                                      \\
     & \leq -\lambda e^2(t) +
    \kappa(t) e(t), \quad \text{(by \cref{eq:error-bound-f})} \nonumber
  \end{align}
  In \eqref{eq:grad-F-cvx-inner-products}, we have used that $\frac{\d u}{\d t}(t)=-\bG_V^{-1}\xi_{u(t)}=f(u(t))$ for a.e.~$t>0$ thanks to \cref{thm:uniqueness-gf-solution}. We also have used that, by assumption, $\varphi(\theta(t))\in \dom(D^F \cF)$ which implies that  $f(\varphi(\theta(t))) = -\bG^{-1}_V \prt*{\xi_{\varphi(\theta(t))}}$ is well-defined. The rest of the proof follows by the same Grönwall estimate as the one of \cref{theorem:best-fit-estimator}.
\end{proof}

A very important consequence of \cref{lemma:best-fit-estimator-gf} is that, when $\lambda>0$, we have that the effect of previous errors decays exponentially fast in the approximation error bound \eqref{eq:error-bound}. As a consequence, the error at time $t$ is essentially due to the approximation quality at that time. This contrasts with the case $L>0$, where the estimate amplifies previous errors exponentially in time; parabolic equations may fall into either case. This strong result on the error comes at the price of assuming that $\varphi(\theta(t)) \in \dom(D^F \cF)$. We next illustrate through an example that, in essence, this means that we need to assume that we work with PDE evolutions and approximations that have enough spatial regularity. Many well-known PDEs can be expressed as gradient flows in Hilbert spaces. Some of the most common examples include the heat flow equation and the biharmonic flow equation. As a concrete example, let us consider the Allen-Cahn equation posed on a bounded Lipschitz domain $\Omega\subset \bR^d$. The problem involves parameters $a, b>0$. The strong form of the equation reads
\begin{align}
  \label{eq:allen-cahn-sf}
  \dot u & = a \Delta u + b(u-u^3), \quad \text{in }\Omega \\
  u      & =0, \quad \text{on }\partial \Omega.
\end{align}
It is well-known (see, e.g., \cite[Example 2.9]{mielke2023introduction}), that this equation has a gradient flow structure $(L^2(\Omega), \cF_{AC}, \bG_V)$, where $V=L^2(\Omega)$ is endowed with its natural inner product and
\begin{equation}
  \cF_{AC}(u) =
  \begin{cases}
    \int_\Omega \prt*{\frac{a}{2}\abs{\nabla u}^2 + \frac{b}{4}\prt{u^2-1}^2} \d x & \text{  if  } u \in \dom(\cF_{AC})=H^1_0(\Omega) \cap L^4(\Omega) \\
    +\infty                                                                        & \text{ if } u \in L^2(\Omega) \setminus \dom(\cF_{AC}).
  \end{cases}
\end{equation}
The functional $\cF_{AC}$ is proper, lower semicontinuous and $\lambda$-convex with $\lambda=-b$. A sufficient regularity class for its subdifferential is
$$
  K:=H^2(\Omega)\cap H^1_0(\Omega)\cap L^6(\Omega)
  \subseteq\dom(D^F\cF_{AC}).
$$
For $u\in K$, the subdifferential consists of the single functional $\xi_u\in V'$ given by
\begin{equation}
  \label{eq:AllenCahn-differential}
  \xi_u(v)=-\int_\Omega\prt{a\Delta u+b(u-u^3)}v\,\d x,
  \qquad \forall v\in L^2(\Omega).
\end{equation}
Thus
$$
  \grad_V\cF_{AC}(u)=\{-a\Delta u-b(u-u^3)\},
  \qquad f(u)=a\Delta u+b(u-u^3),\qquad u\in K.
$$
For $v,w\in K$, integration by parts gives
\begin{align*}
  \inner{f(v)-f(w),v-w}_{L^2}
   & =-a\norm{\nabla(v-w)}_{L^2}^2+b\norm{v-w}_{L^2}^2
  -b\int_\Omega(v^3-w^3)(v-w)\,\d x                    \\
   & \leq b\norm{v-w}_{L^2}^2.
\end{align*}
Therefore, the gradient flow inclusion $\dot u\in-\grad_V\cF_{AC}(u)$ gives the strong form \cref{eq:allen-cahn-sf} whenever $u(t)\in K$ and the evolution equation holds. Existence and uniqueness of the gradient flow for $u_0\in\dom(\cF_{AC})$ follow from \cref{thm:uniqueness-gf-solution}. The error estimate applies with $L=b$ under the assumptions of \cref{lemma:best-fit-estimator-gf}.

\bibliographystyle{plain}
\bibliography{refs}

\end{document}